\documentclass[a4paper]{article}

\usepackage{amsmath,amsthm,amsfonts,amssymb}
\usepackage{tikz,tikz-cd,tikz-3dplot}
\usetikzlibrary{positioning}
\usepackage{geometry}
\usepackage{multirow}
\usepackage{graphicx}
\usepackage{caption}
\usepackage{subfigure}
\usepackage{float}
\usepackage{verbatim}
\usepackage{booktabs}
\usepgflibrary{shapes.misc}
\usepackage{cite}

\DeclareMathOperator{\Proj}{Proj}
\DeclareMathOperator{\Aut}{Aut}

\DeclareMathOperator{\Hom}{Hom}

\DeclareMathOperator{\im}{im}
\DeclareMathOperator{\Stab}{Stab}
\DeclareMathOperator{\Bl}{Bl}
\DeclareMathOperator{\ns}{ns}
\DeclareMathOperator{\Pic}{Pic}
\DeclareMathOperator{\Div}{div}
\DeclareMathOperator{\PGL}{PGL}
\DeclareMathOperator{\Sym}{Sym}

\DeclareMathOperator{\Bir}{Bir}
\DeclareMathOperator{\Conv}{Conv}

\theoremstyle{plain}
\newtheorem{thm}{Theorem}[section]
\newtheorem{lem}[thm]{Lemma}
\newtheorem{prop}[thm]{Proposition}
\newtheorem{cor}[thm]{Corollary}

\theoremstyle{definition}
\newtheorem{dfn}[thm]{Definition}
\newtheorem{ex}[thm]{Example}
\newtheorem{rmk}[thm]{Remark}
\newtheorem*{ack}{Acknowledgments}

\begin{document}

\colorlet{LightGray}{gray!30}

\tikzstyle{circ}=[draw,circle,inner sep=1mm]
\tikzstyle{rect}=[draw,rectangle,minimum size=2mm]
\tikzstyle{cross}=[draw,cross out]
\tikzstyle{greycirc}=[draw,circle,fill=black!15,inner sep=1mm]
\tikzstyle{blackcirc}=[draw,circle,fill=black,inner sep=1mm]
\tikzstyle{c}=[draw,circle]

\tikzset{vert/.style={path picture={ 
      \draw[black]
       (path picture bounding box.north) -- (path picture bounding box.south) ;       
       }}}       
       
\tikzset{oplus/.style={path picture={ 
      \draw[black]
       (path picture bounding box.south) -- (path picture bounding box.north) 
       (path picture bounding box.west) -- (path picture bounding box.east);
      }}} 

\title{Moduli of real pointed quartic curves}
\author{Sander Rieken}

\maketitle

\abstract{We describe a natural open stratum in the moduli space of smooth real
pointed quartic curves in the projective plane. This stratum consists of real isomorphism classes of pairs
$(C,p)$ with $p$ a real point on the curve $C$ such that the tangent line at $p$ intersects the curve in
two distinct points besides $p$. We will prove that this stratum consists of 20 connected components. Each of these components has a real toric structure defined by an involution in the Weyl group of
type $E_7$.}

\section{Introduction}

A classical result found in Zeuthen \cite{Zeuthen} is the classification of smooth real plane quartic curves. The set of real points of such a curve consists of at most four ovals in the real projective plane and the six possible configurations are shown below in Figure \ref{realquarticcurves}. The space of real plane quartic curves is the projective space $P_{4,3}(\mathbb{R})=\mathbb{P}\Sym^4(\mathbb{R}^3)$ of dimension $14$. The discriminant subspace $\Delta(\mathbb{R}) \subset P_{4,3}(\mathbb{R})$ that consists of singular real quartic curves is of codimension one. It was proved by Klein \cite{Klein} that each of the six types of smooth real plane quartic curves determines a connected component in the space $P_{4,3}(\mathbb{R})-\Delta(\mathbb{R})$. We will be interested in the moduli space
\[ \mathcal{Q}^\mathbb{R} = \PGL(3,\mathbb{R}) \backslash \left( P_{4,3}(\mathbb{R})-\Delta(\mathbb{R}) \right)\]
whose points represent real isomorphism classes of such quartics. This space also consists of six connected components since the group $\PGL(3,\mathbb{R})$ is connected.  

\begin{figure}[H]
\centering
\setlength{\tabcolsep}{12pt}
\begin{tabular}{ccc}
\includegraphics[scale=0.16,trim= 200 60 200 60,clip]{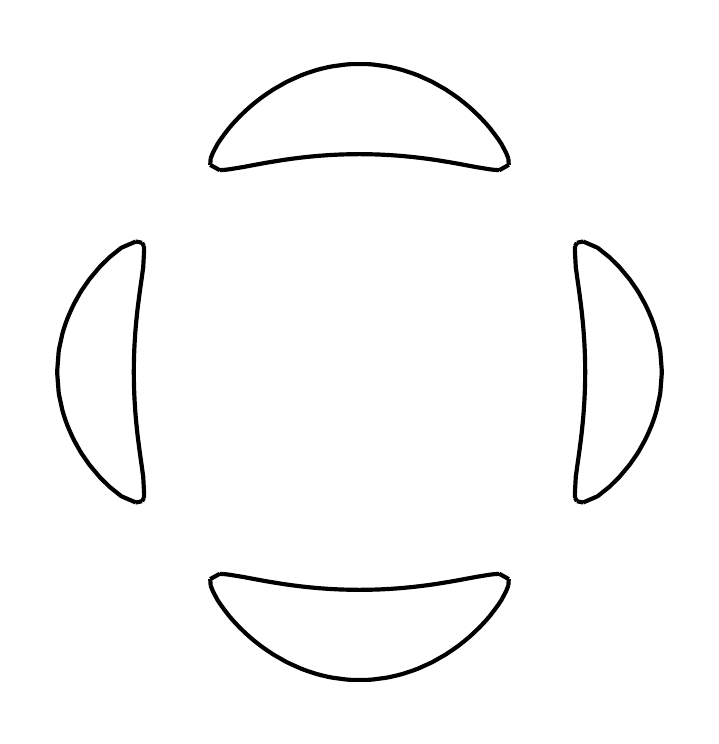} & \includegraphics[scale=0.16,trim= 200 60 200 60,clip]{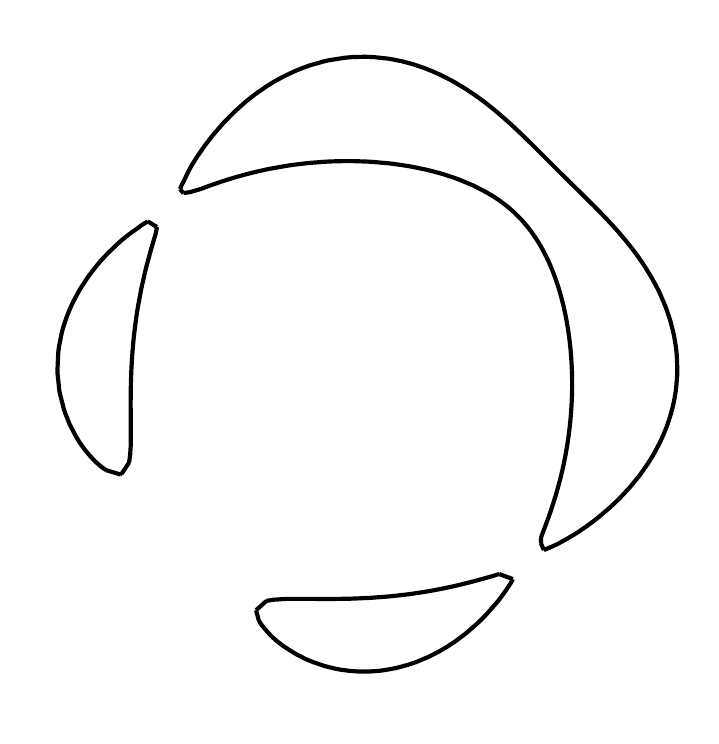} &
\includegraphics[scale=0.16,trim = 200 60 200 60,clip]{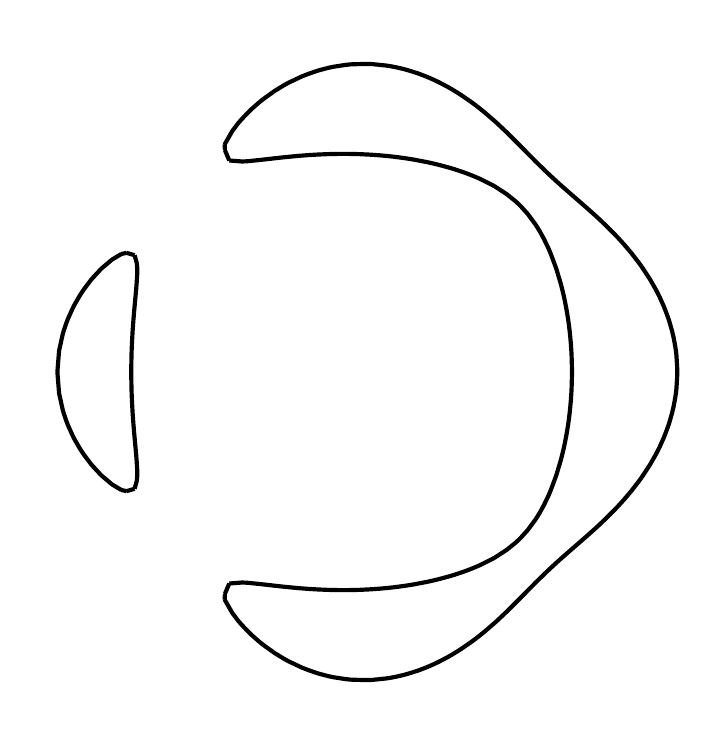} \\

\includegraphics[scale=0.16,trim = 200 60 200 60,clip]{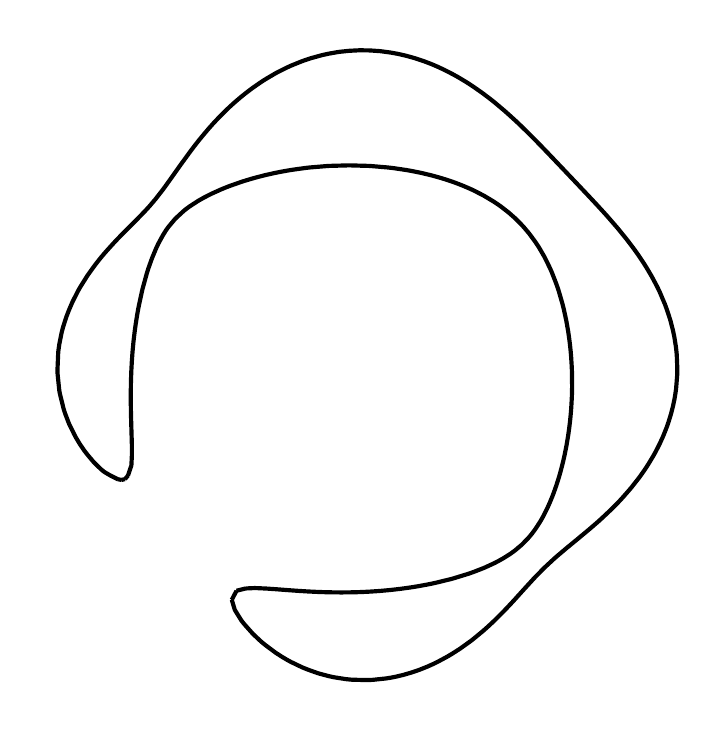} &
\includegraphics[scale=0.16,trim = 200 60 200 60,clip]{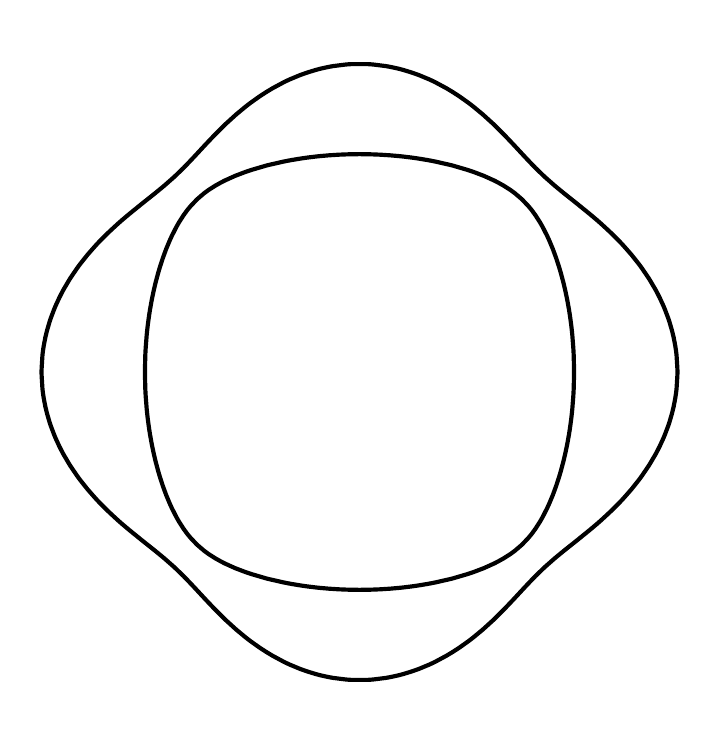} &
 \\
\end{tabular}
\caption{The six topological types of smooth real plane quartic curves obtained by deforming a union of two ellipses.}
\label{realquarticcurves}
\end{figure}

In this article we use modern techniques from the theory of root systems and del Pezzo surfaces to study a variation on the classification by Zeuthen. Suppose we have a smooth real plane quartic curve together with a \emph{general} real point on the curve. By general point we mean that the tangent line at this point intersects the curve in two other, distinct points. In other words the tangent line at this point is not a bitangent or a flex line. These two other points of intersection can be both real or form a pair of complex conjugate points. Our first result is the classification of real quartic curves with a general point, similar to the one given by Zeuthen. 

\begin{thm}\label{introthm1}
The moduli space $(\mathcal{Q}_1^\circ)^\mathbb{R}$ of smooth real plane quartic curves with a general point consists of $20$ connected components. Representative curves for these twenty components are shown in Figure \ref{20classes} below.
\end{thm} 

\begin{figure}[H]
\centering
\setlength{\tabcolsep}{0pt}
\begin{tabular}{cccc@{\hspace{10pt}}c}
\includegraphics[scale=0.08,trim = 250 0 250 0,clip]{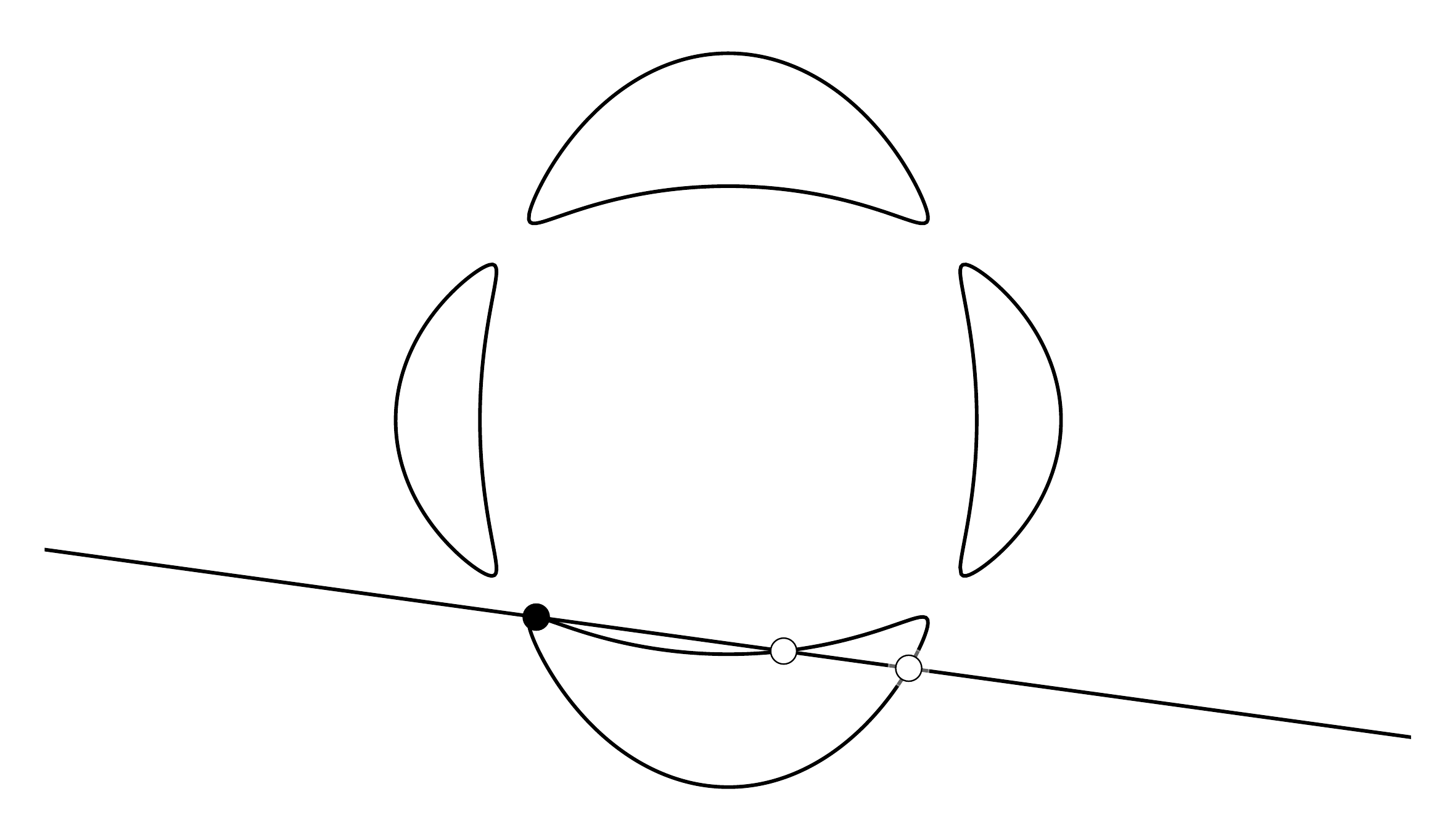}&  
\includegraphics[scale=0.08,trim = 250 0 250 0,clip]{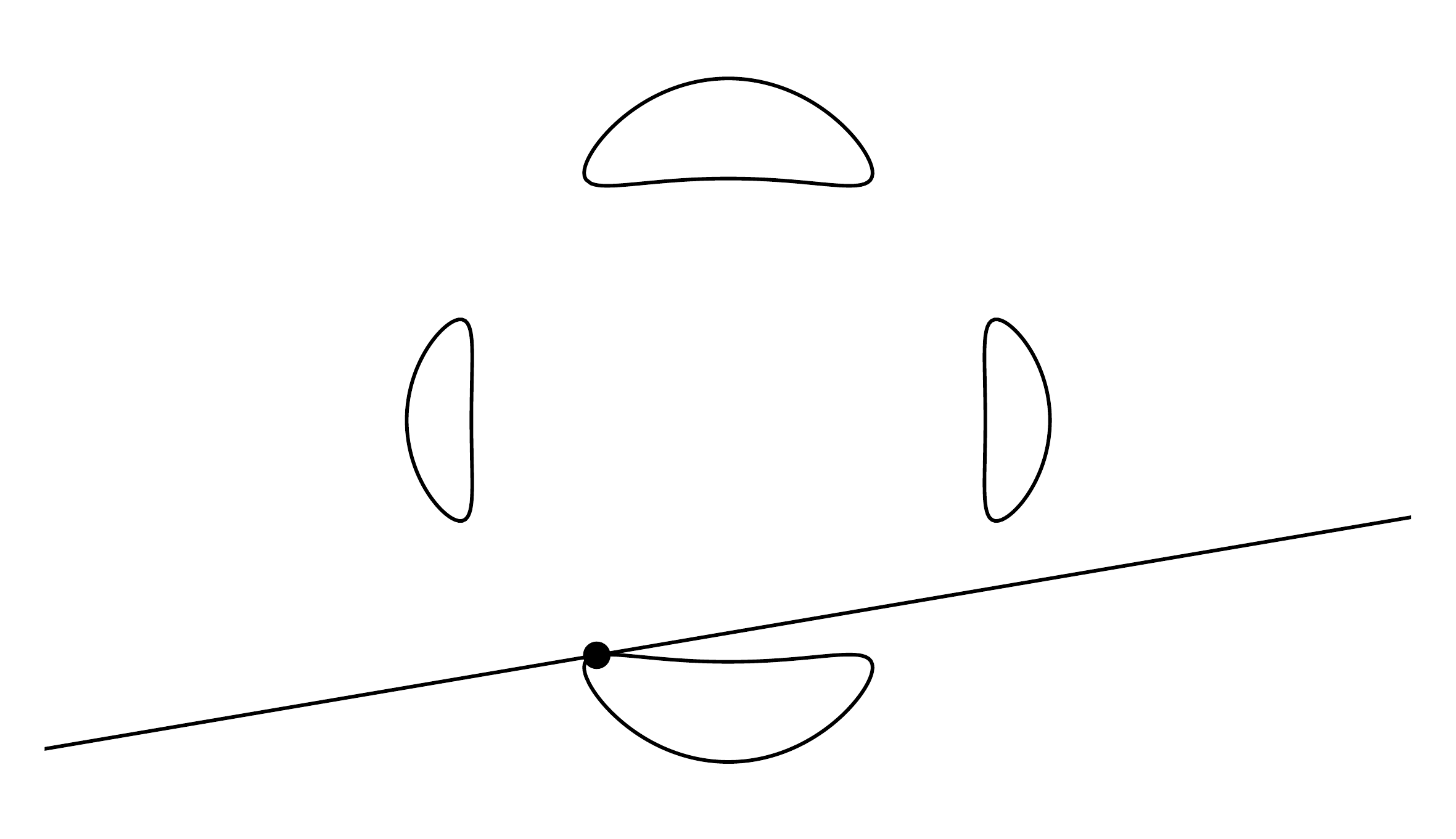}&
\includegraphics[scale=0.08,trim = 250 0 250 0,clip]{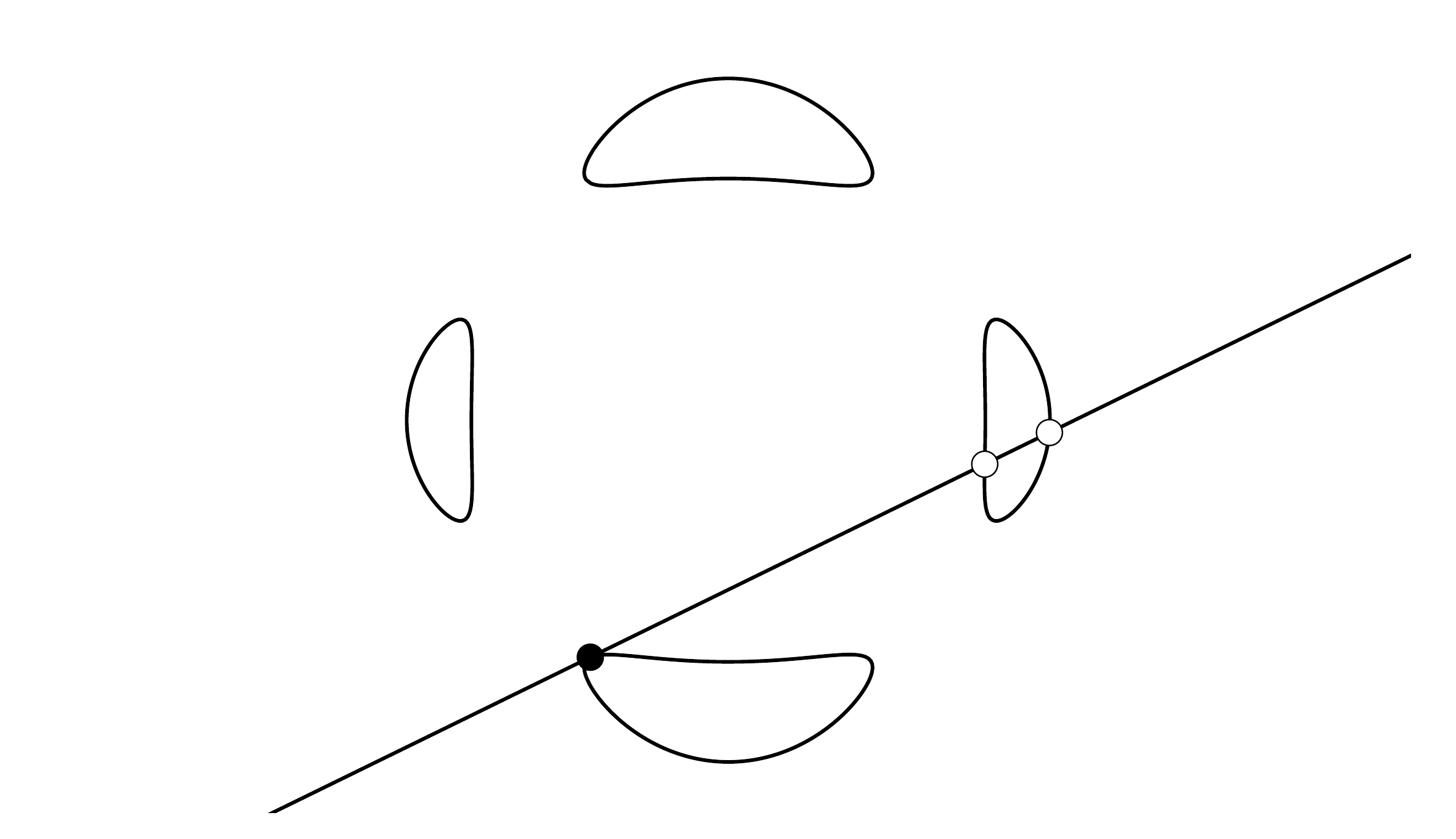}&  
\includegraphics[scale=0.08,trim = 250 0 250 0,clip]{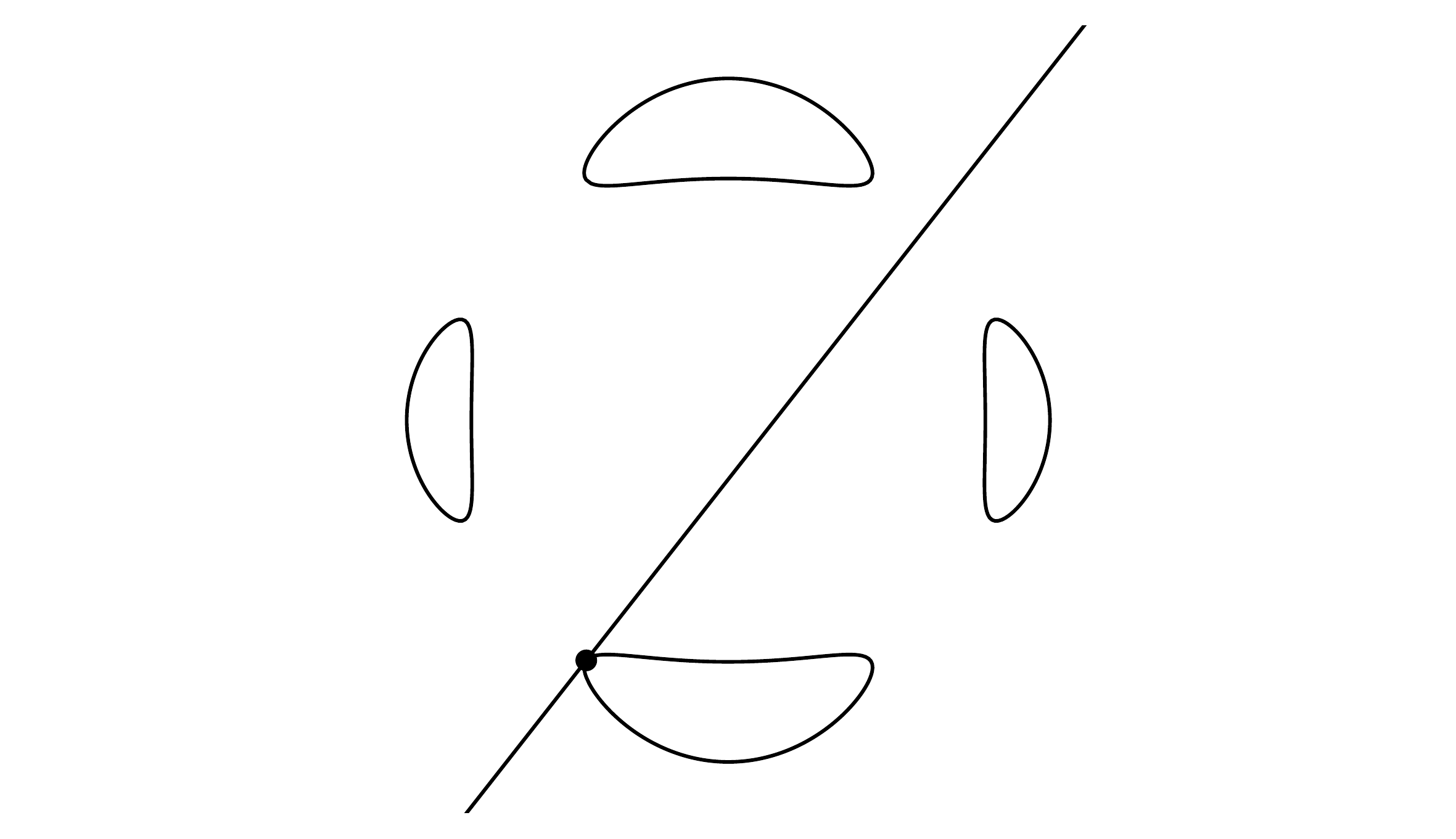}& 
\includegraphics[scale=0.08,trim = 250 0 250 0,clip]{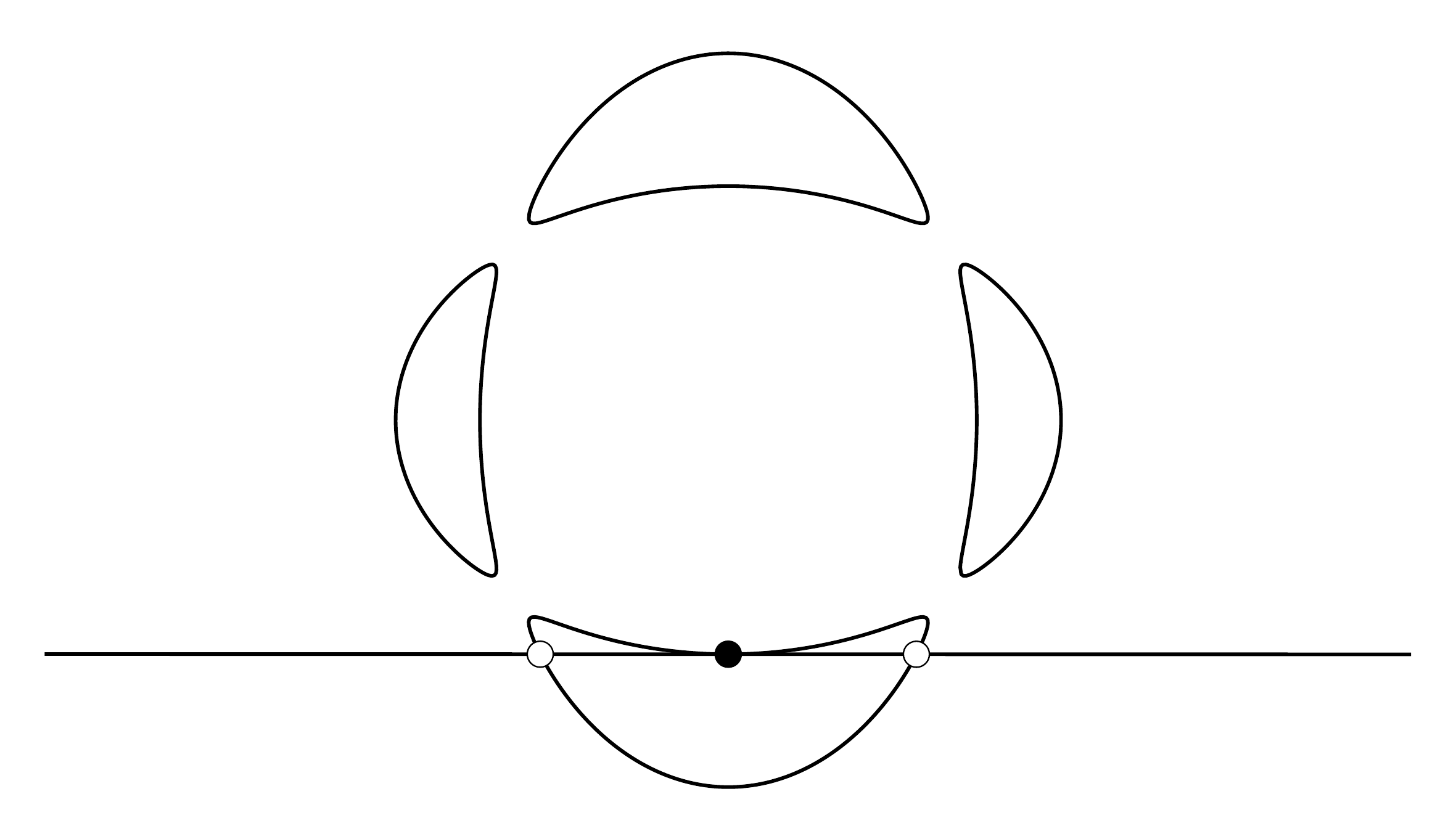} \\
\includegraphics[scale=0.08,trim = 250 0 250 0,clip]{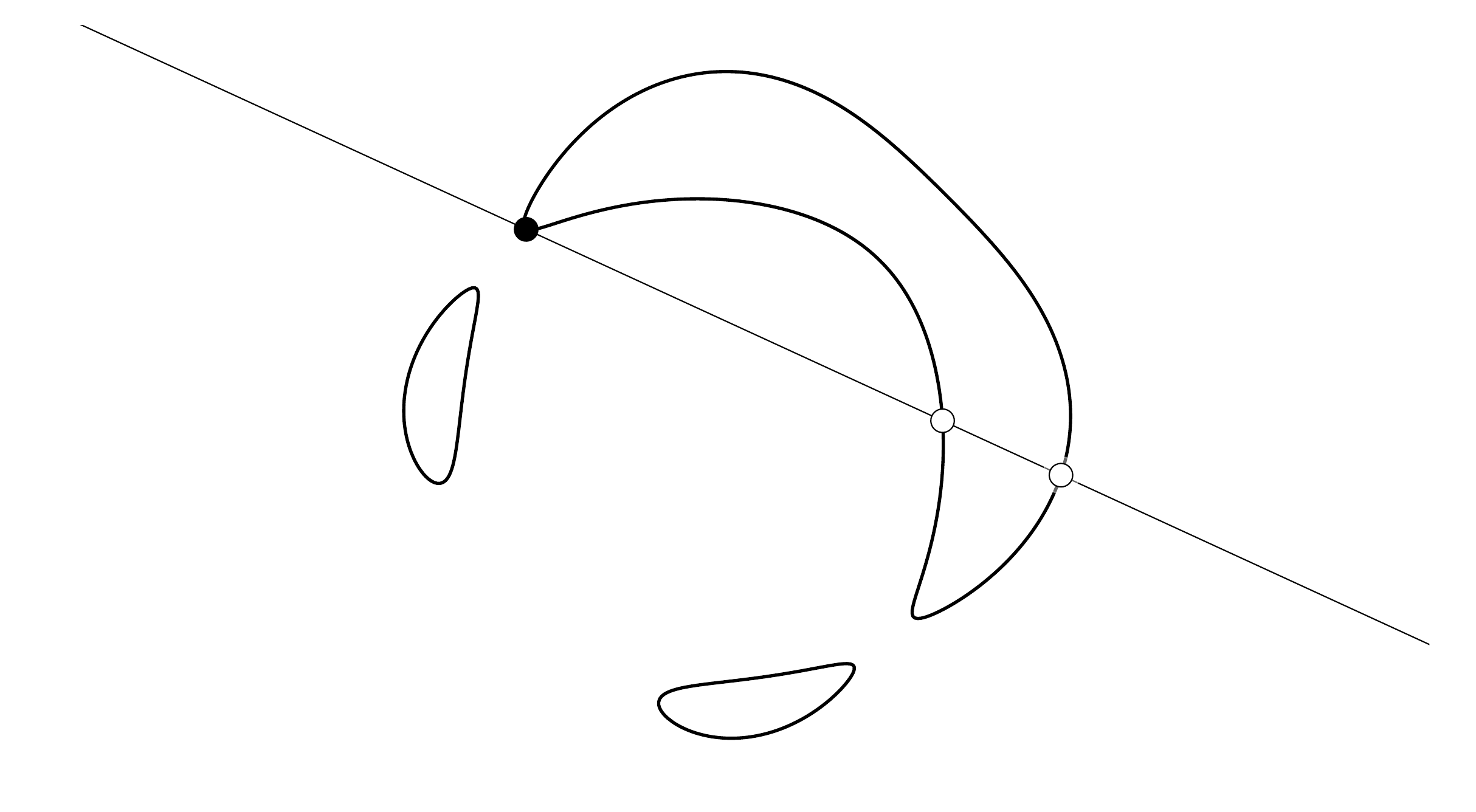}& 
\includegraphics[scale=0.08,trim = 250 0 250 0,clip]{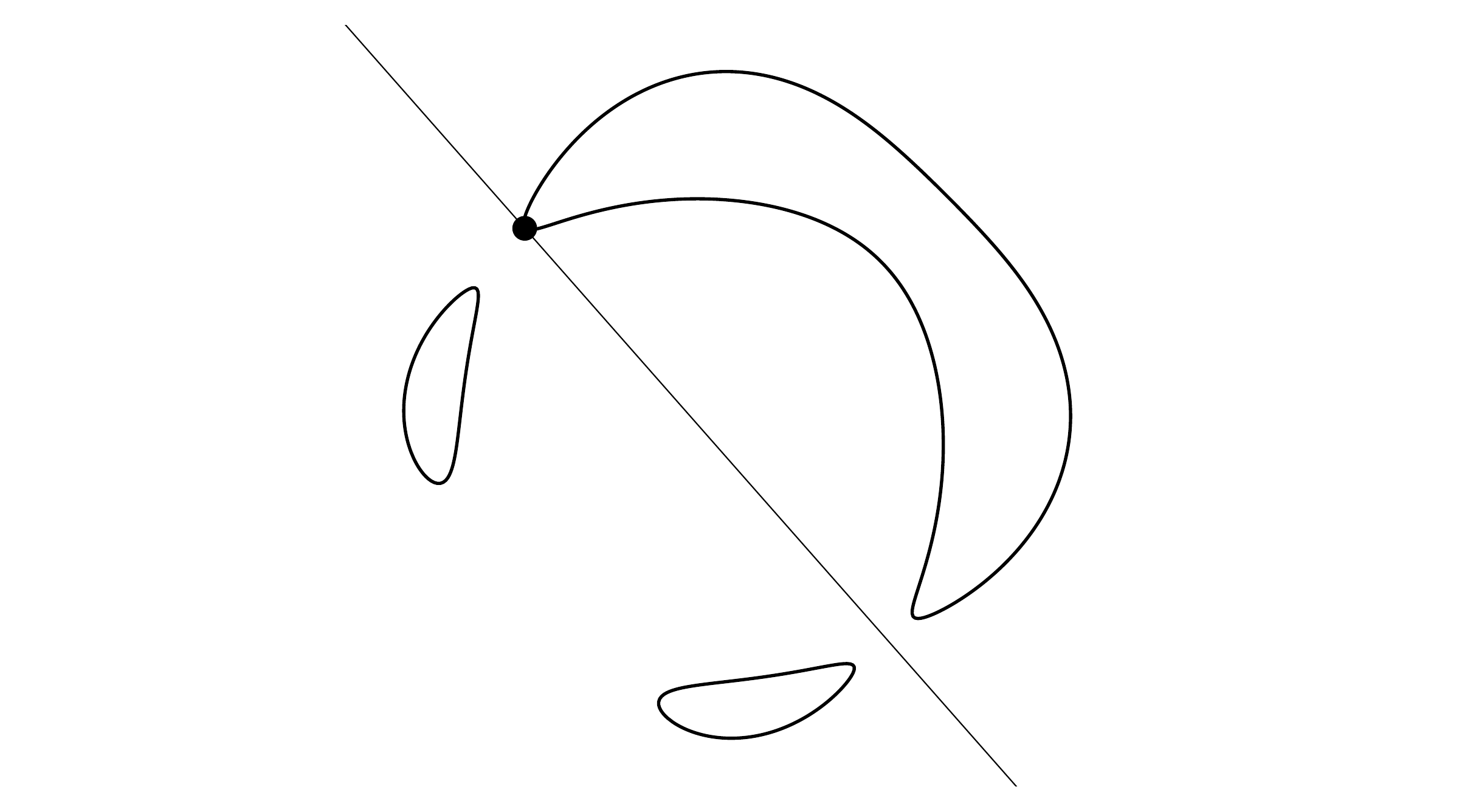}&  
\includegraphics[scale=0.08,trim = 250 0 250 0,clip]{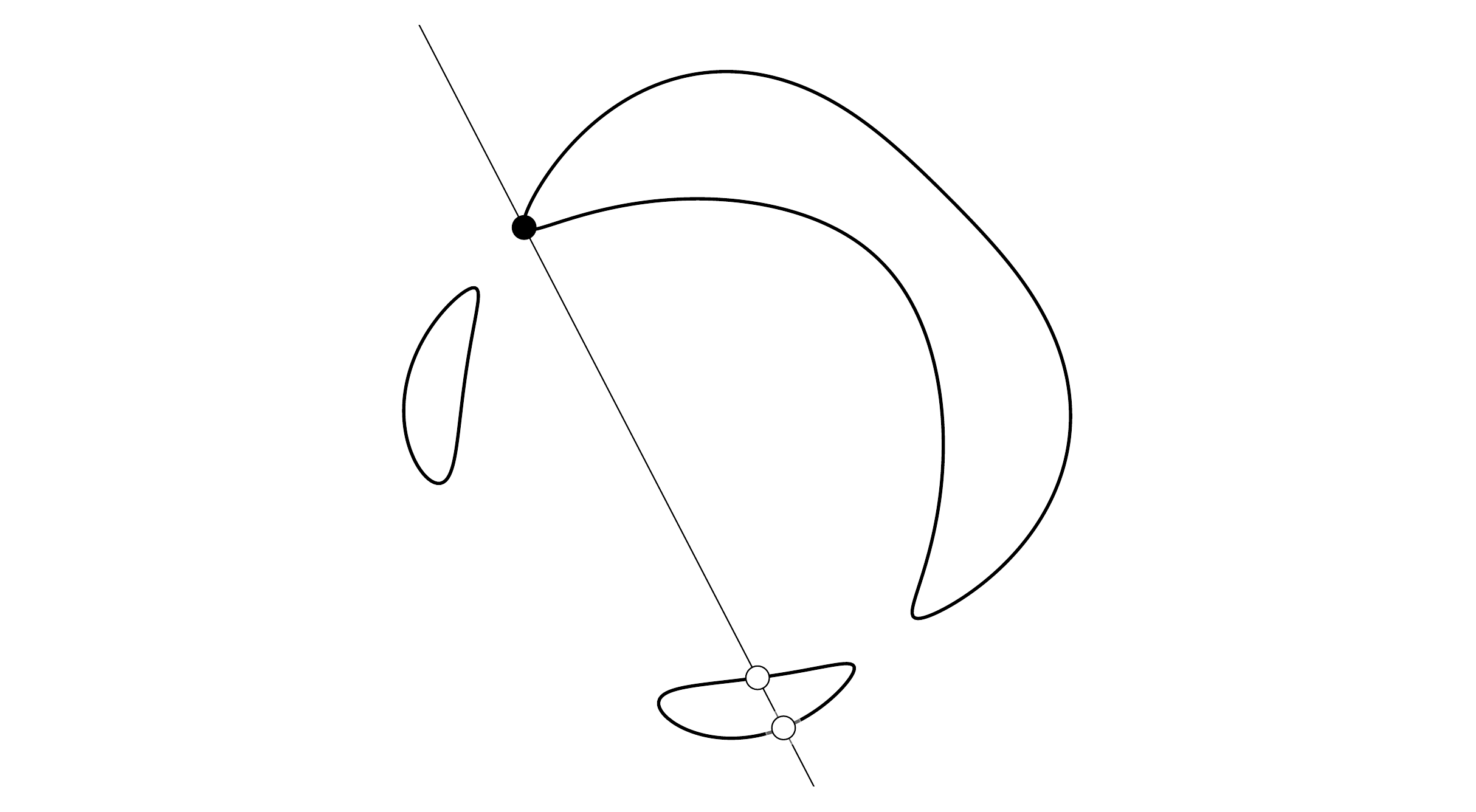}&&
\includegraphics[scale=0.08,trim = 250 0 250 0,clip]{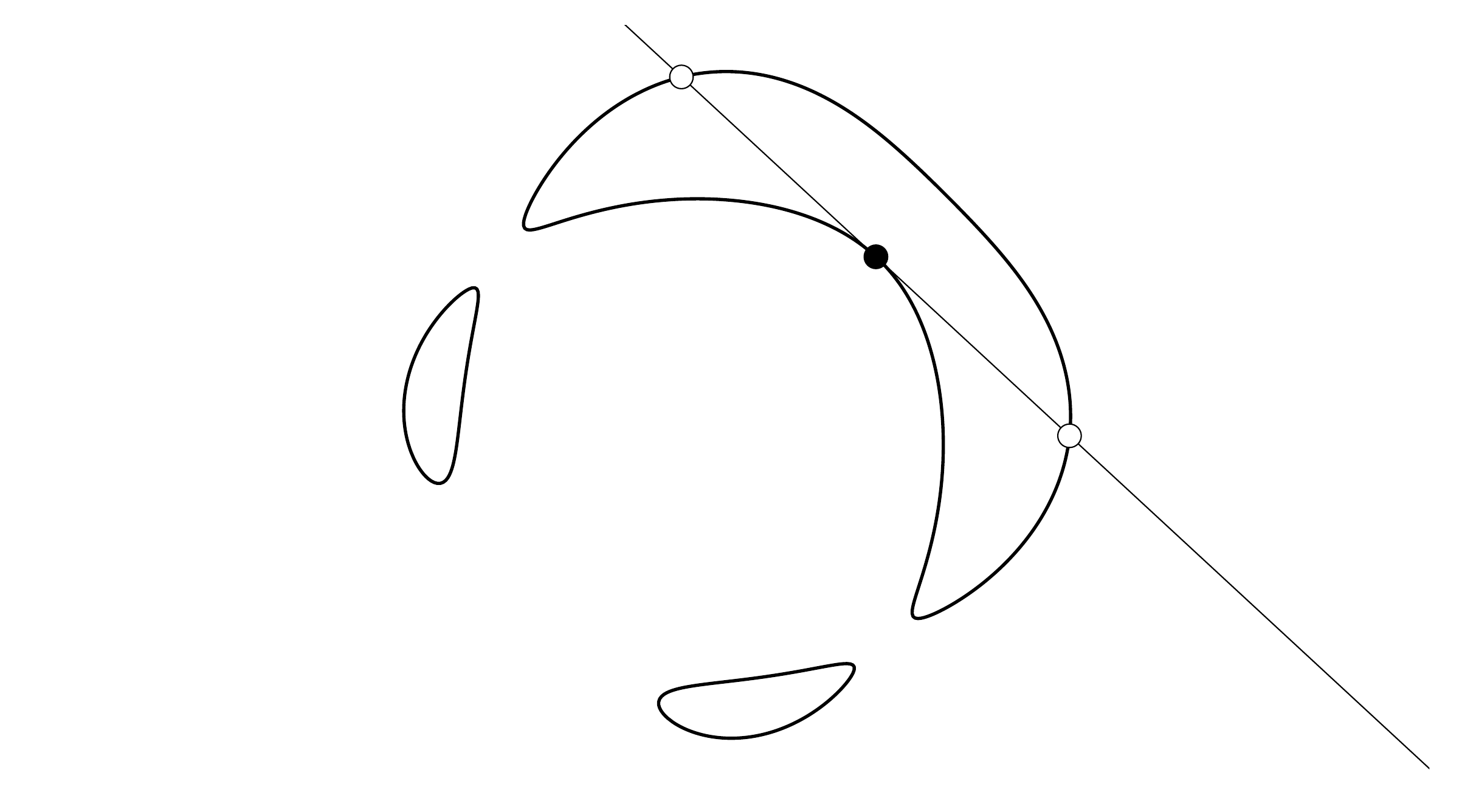}\\
\includegraphics[scale=0.08,trim = 250 0 250 0,clip]{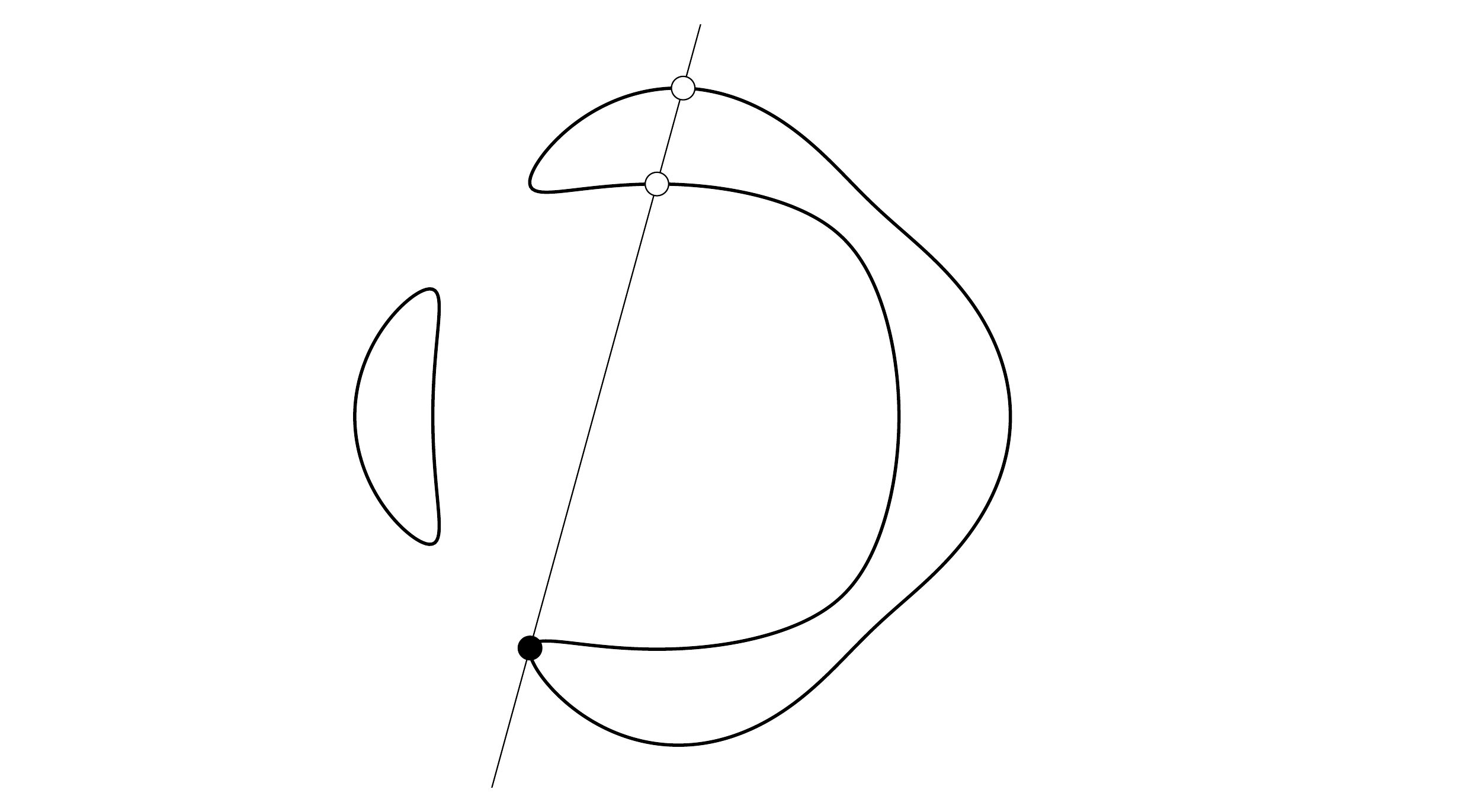}& 
\includegraphics[scale=0.08,trim = 250 0 250 0,clip]{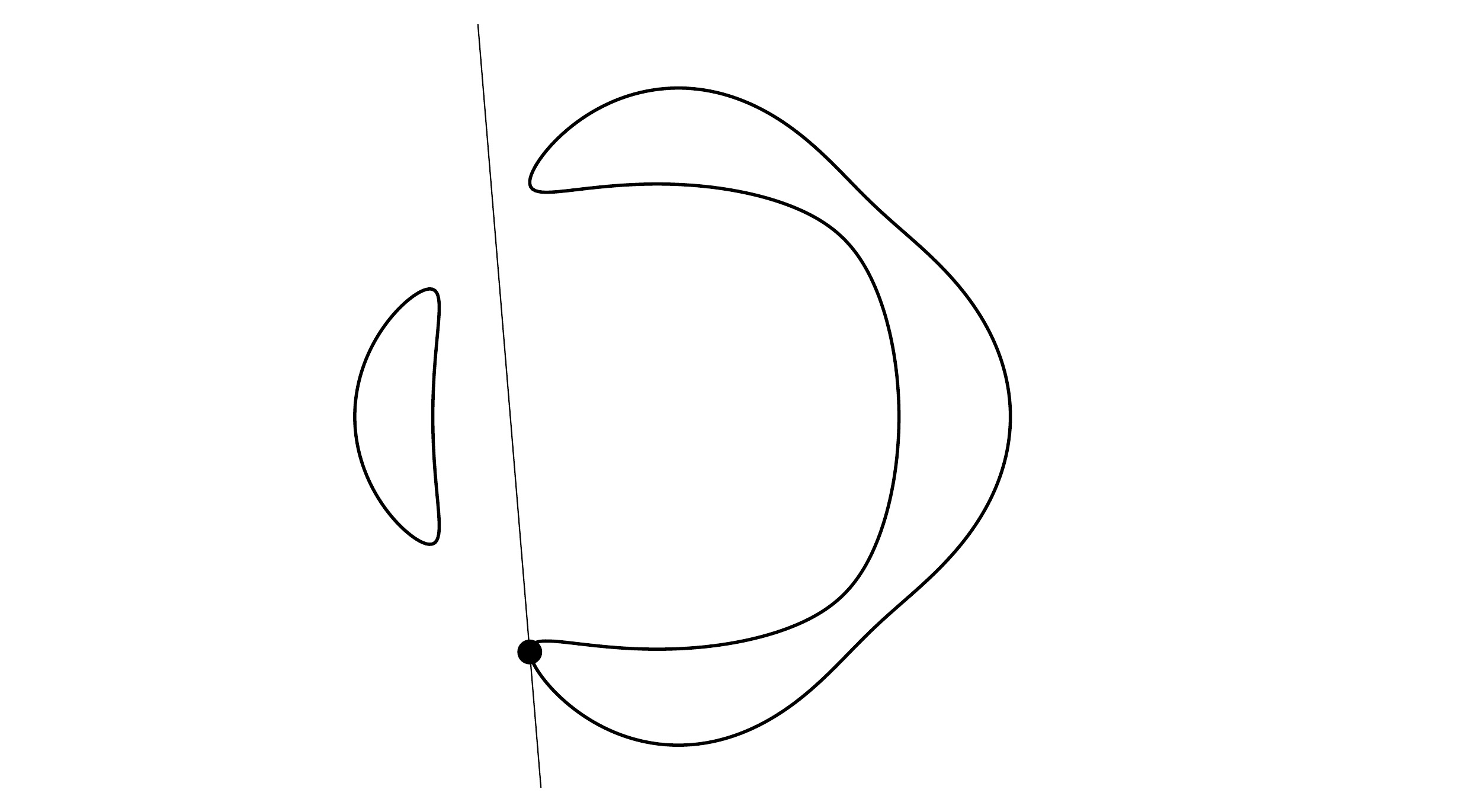}&
\includegraphics[scale=0.08,trim = 250 0 250 0,clip]{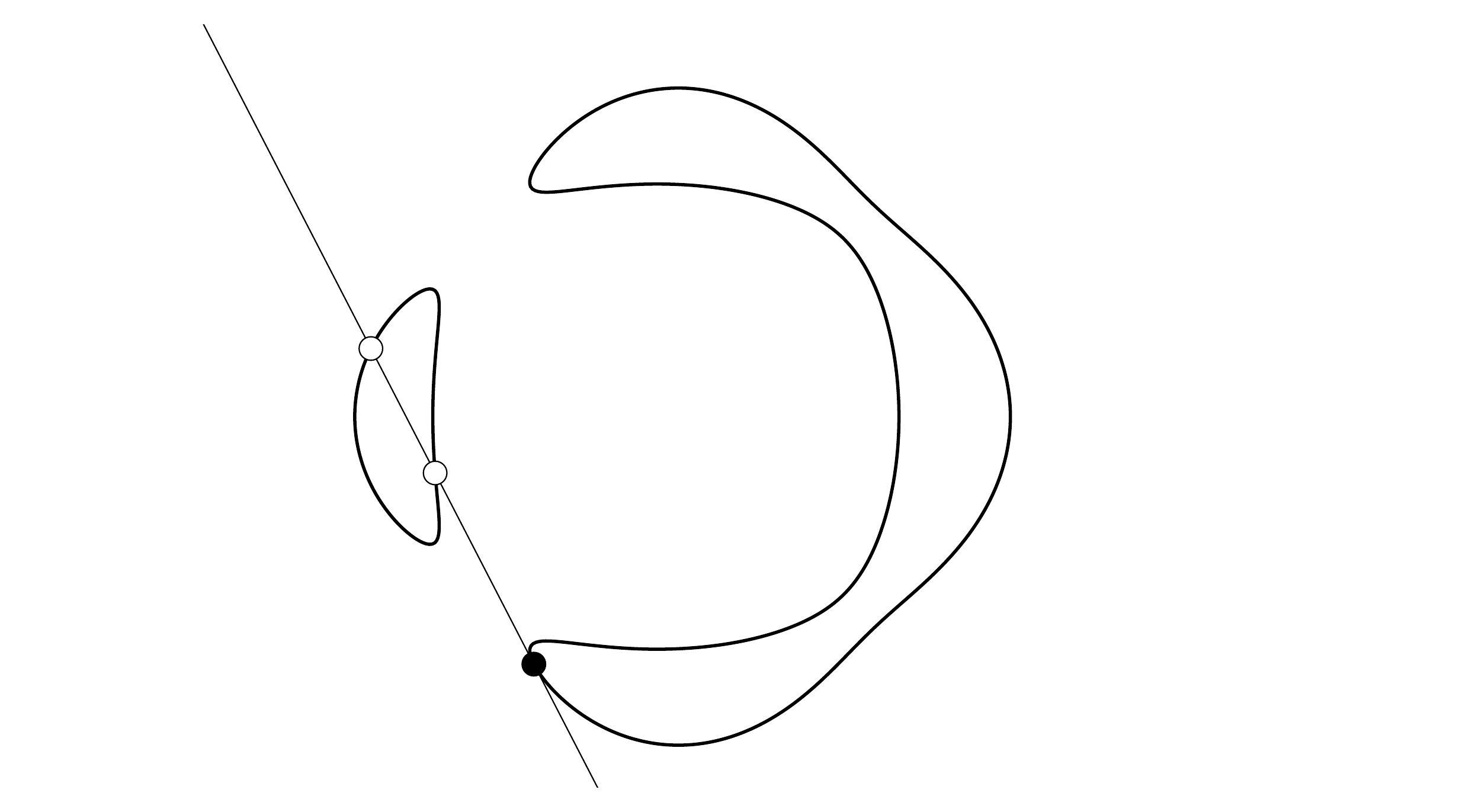}&&
\includegraphics[scale=0.08,trim = 250 0 250 0,clip]{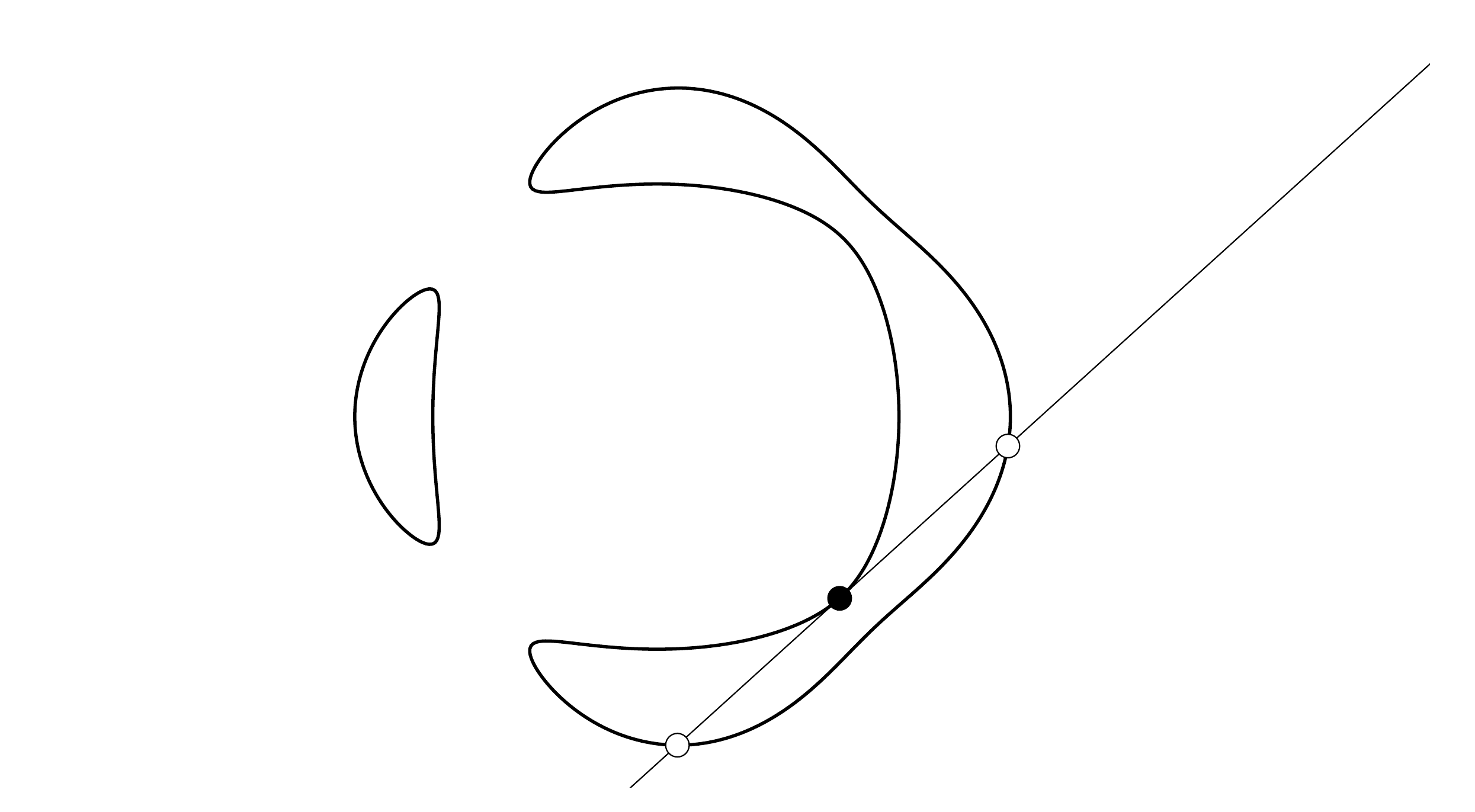}\\
\includegraphics[scale=0.08,trim = 250 0 250 0,clip]{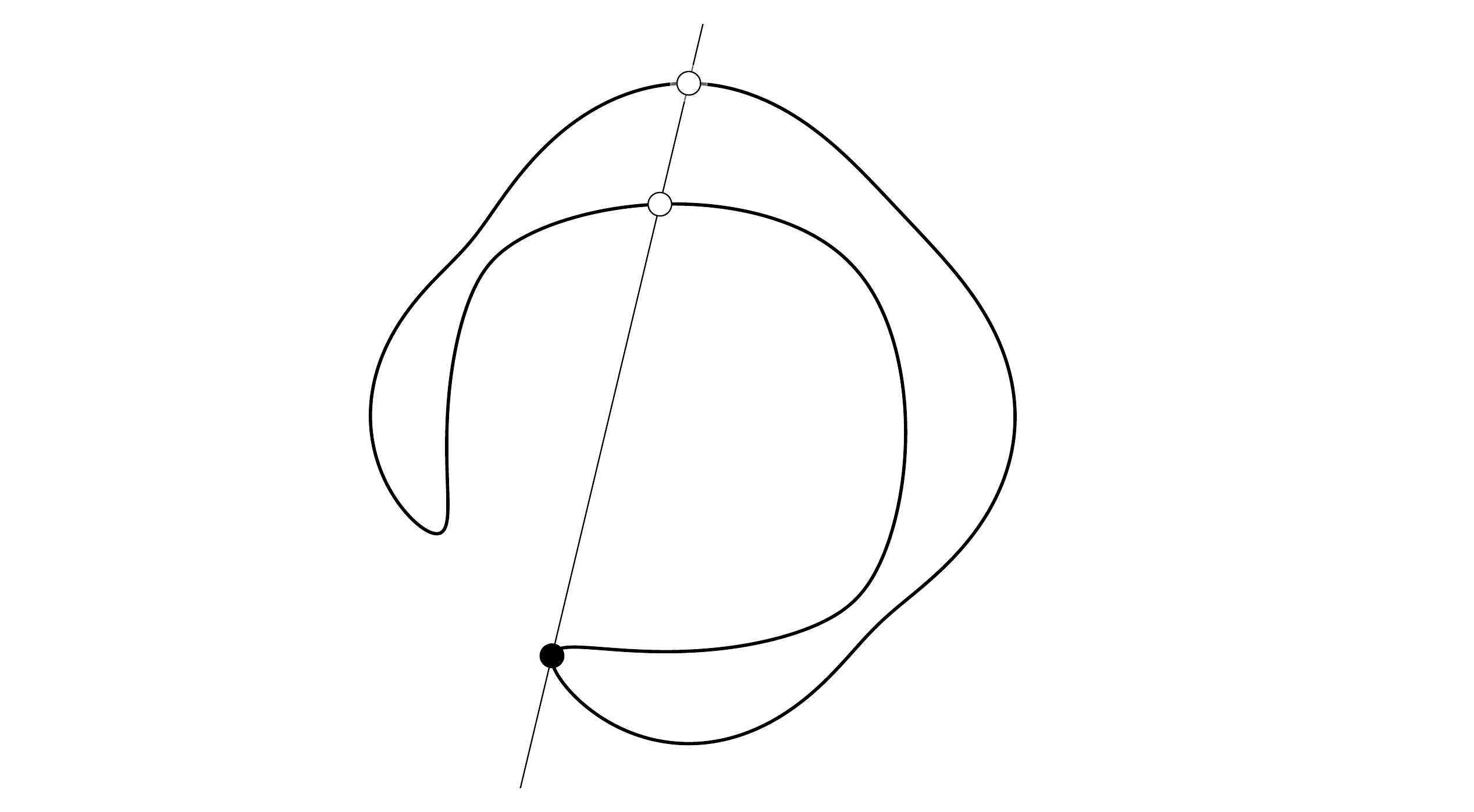}&
\includegraphics[scale=0.08,trim = 250 0 250 0,clip]{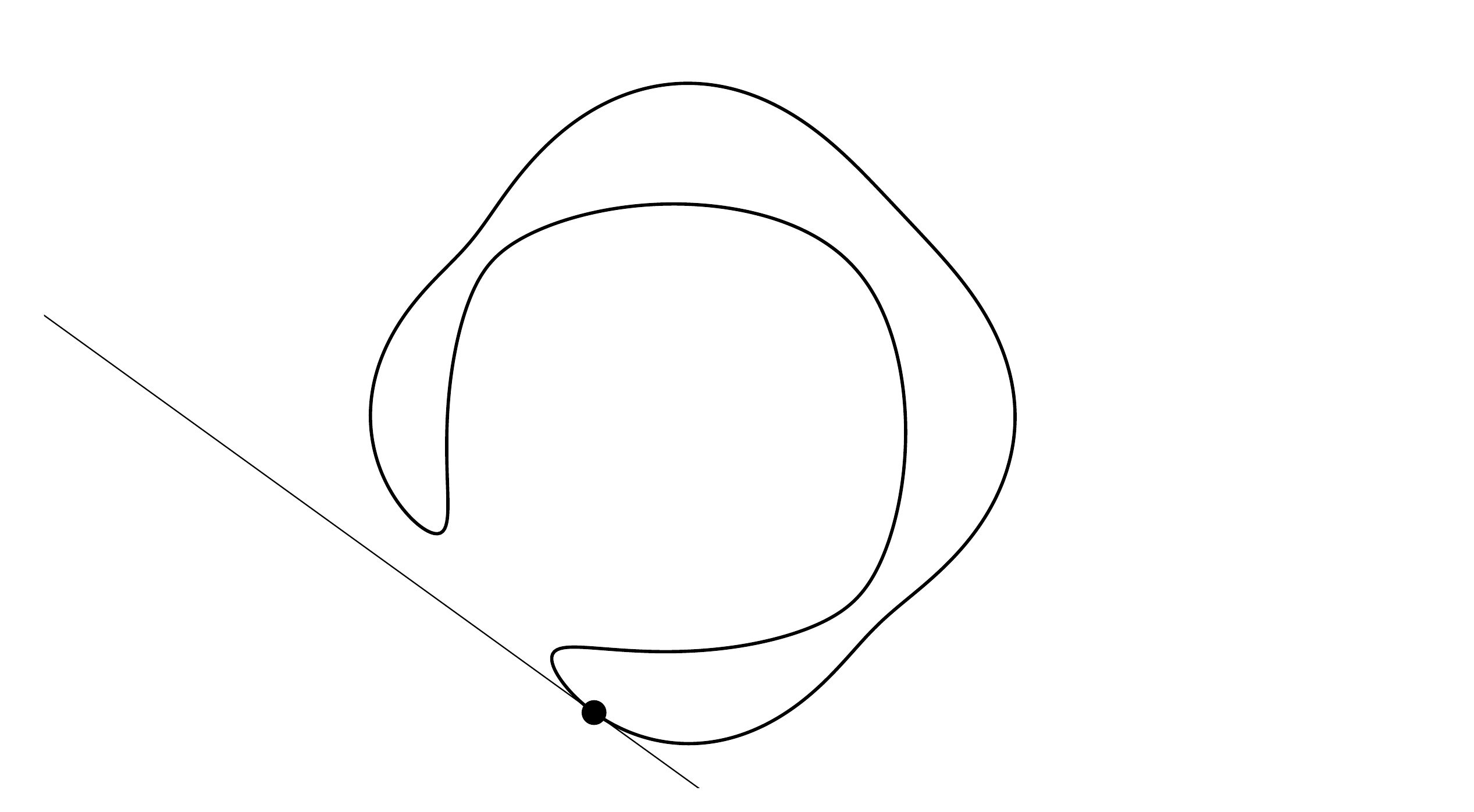}&&&
\includegraphics[scale=0.08,trim = 250 0 250 0,clip]{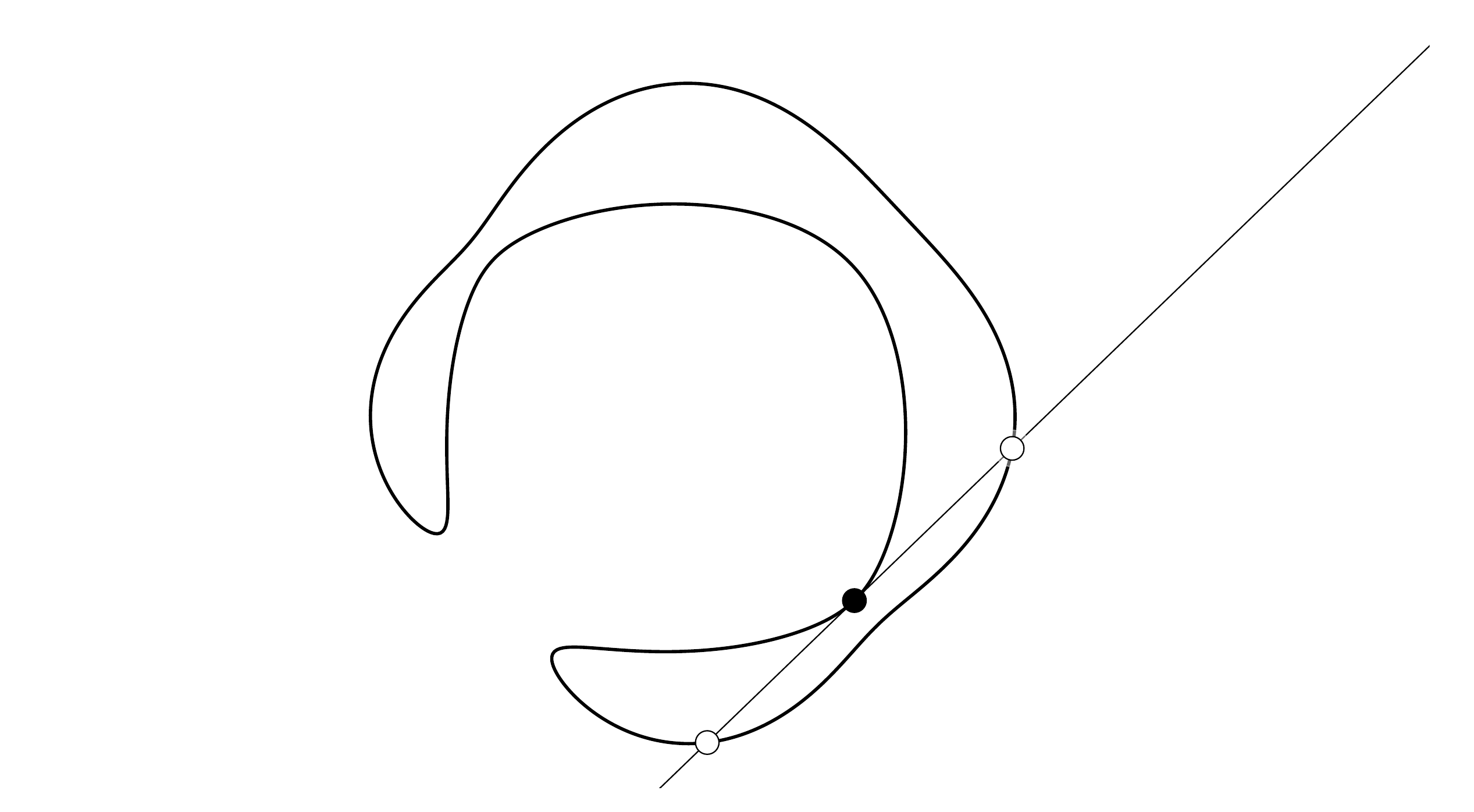}\\
\includegraphics[scale=0.08,trim = 250 0 250 0,clip]{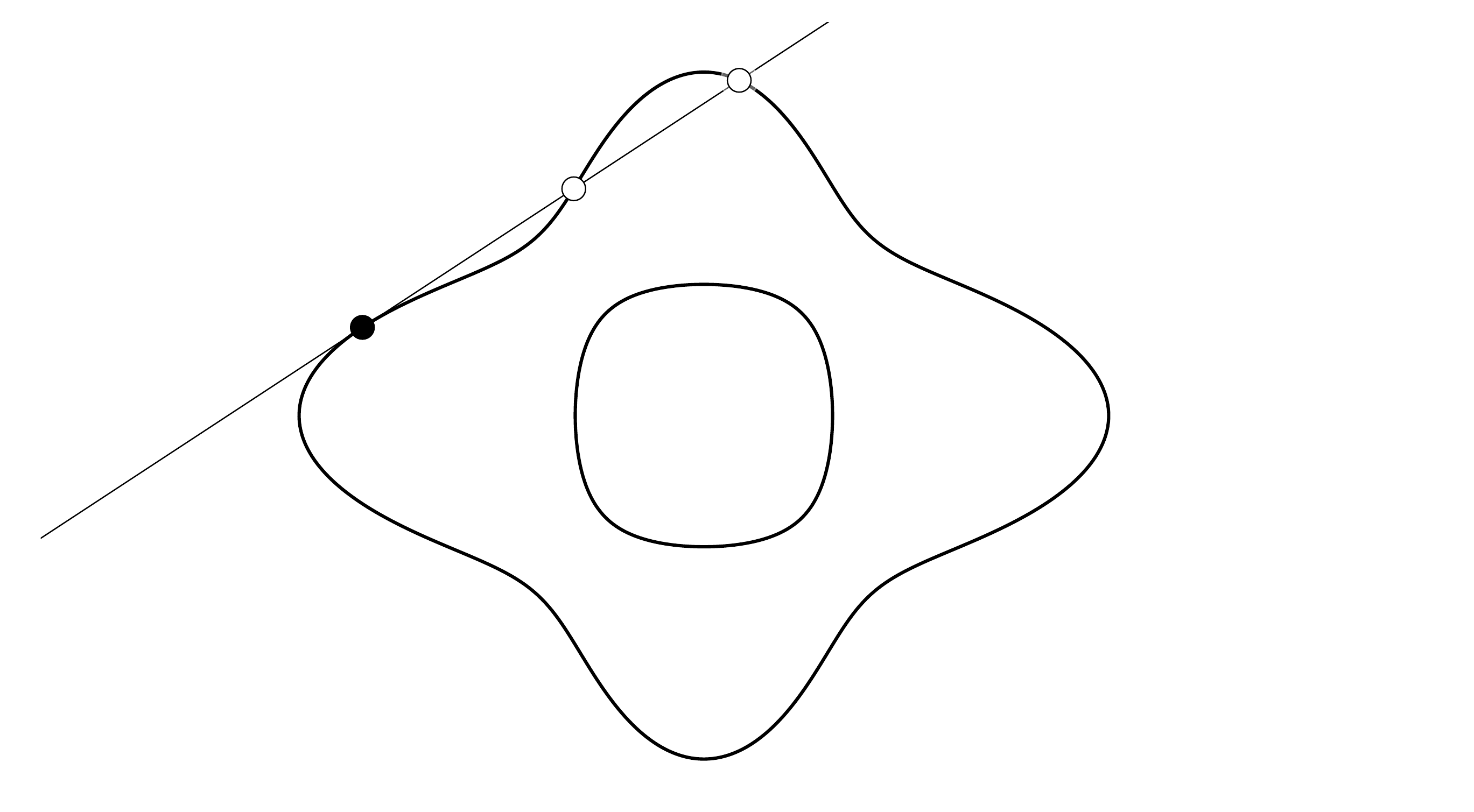}&
\includegraphics[scale=0.08,trim = 250 0 250 0,clip]{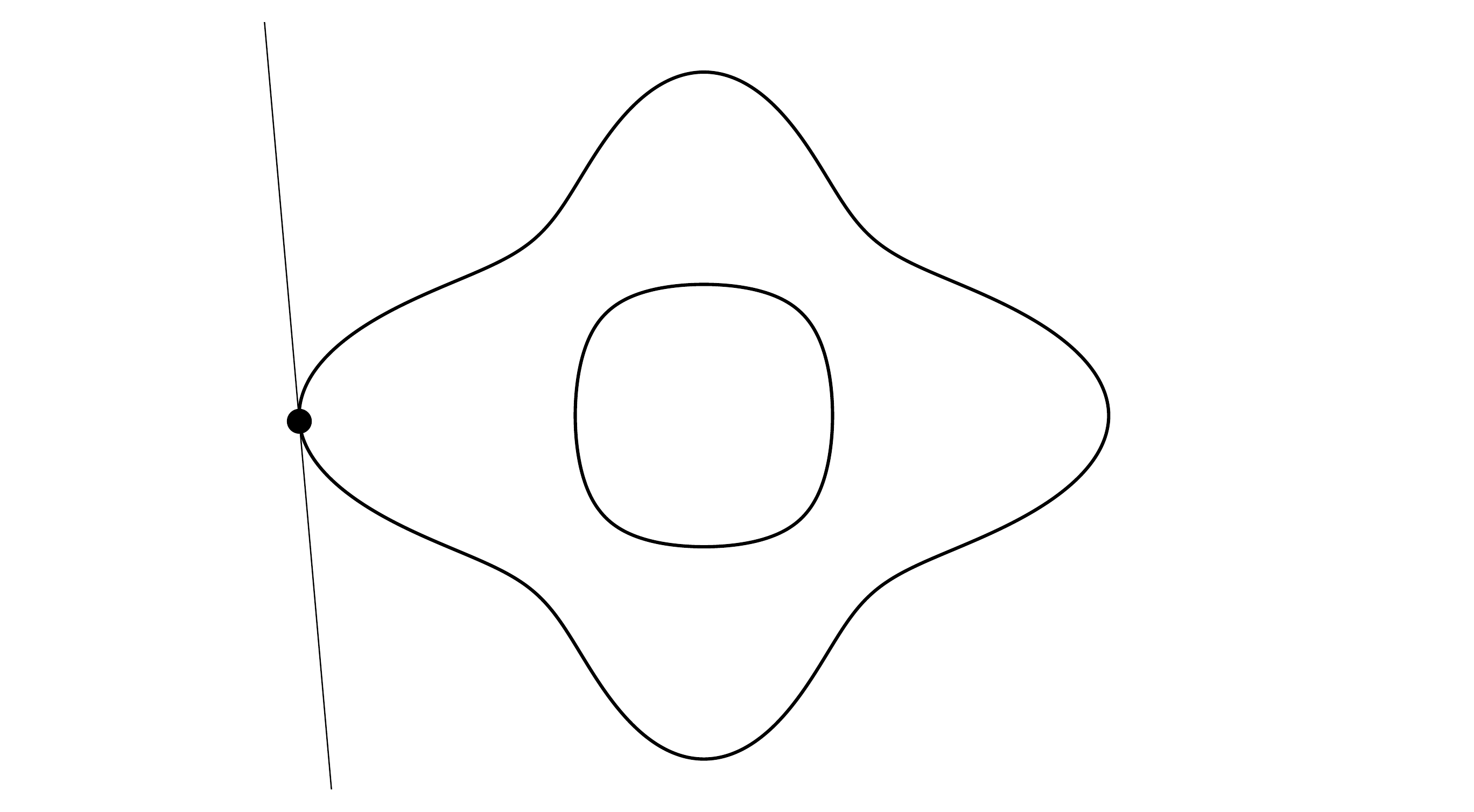}&&
\includegraphics[scale=0.08,trim = 250 0 250 0,clip]{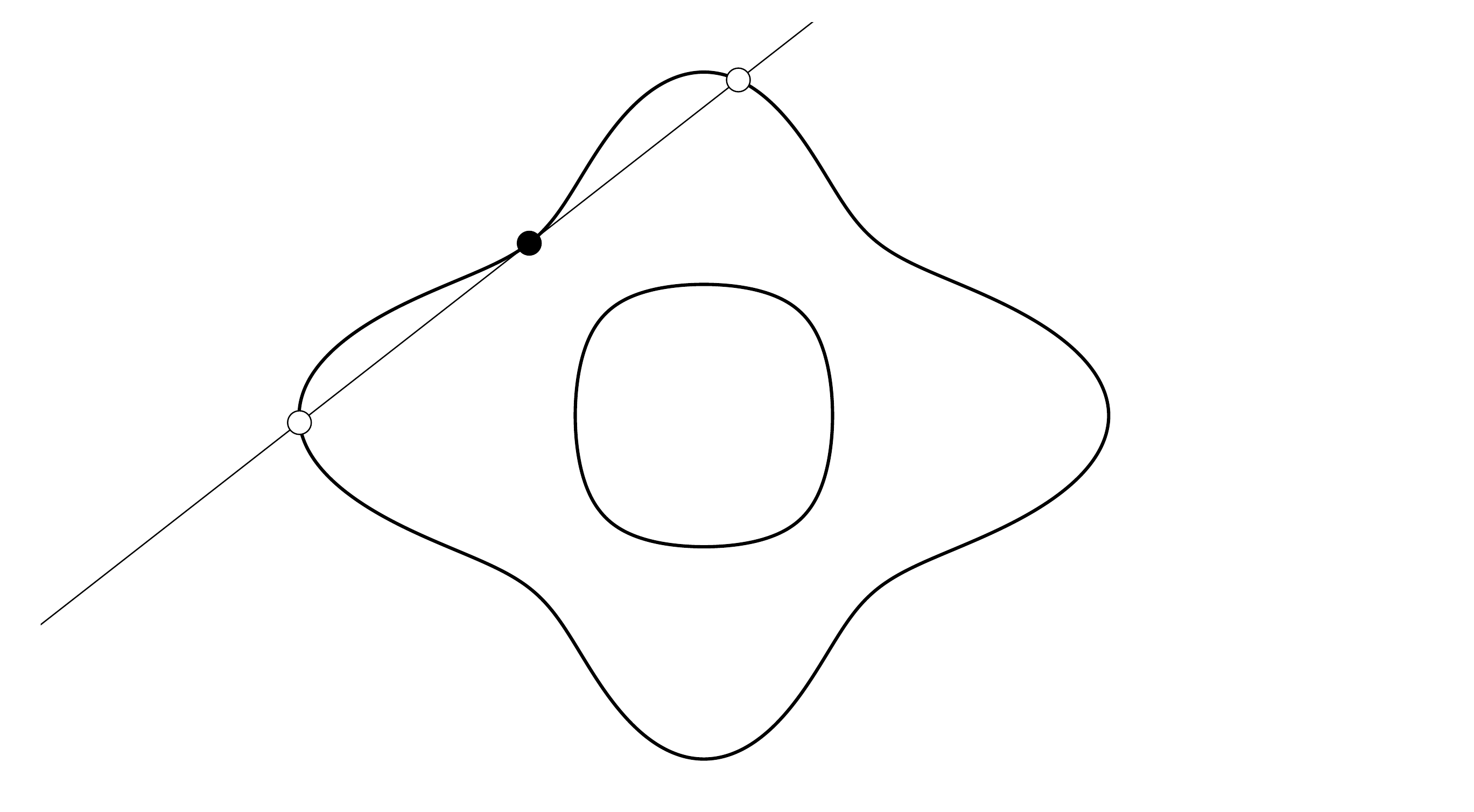}&
\includegraphics[scale=0.08,trim = 250 0 250 0,clip]{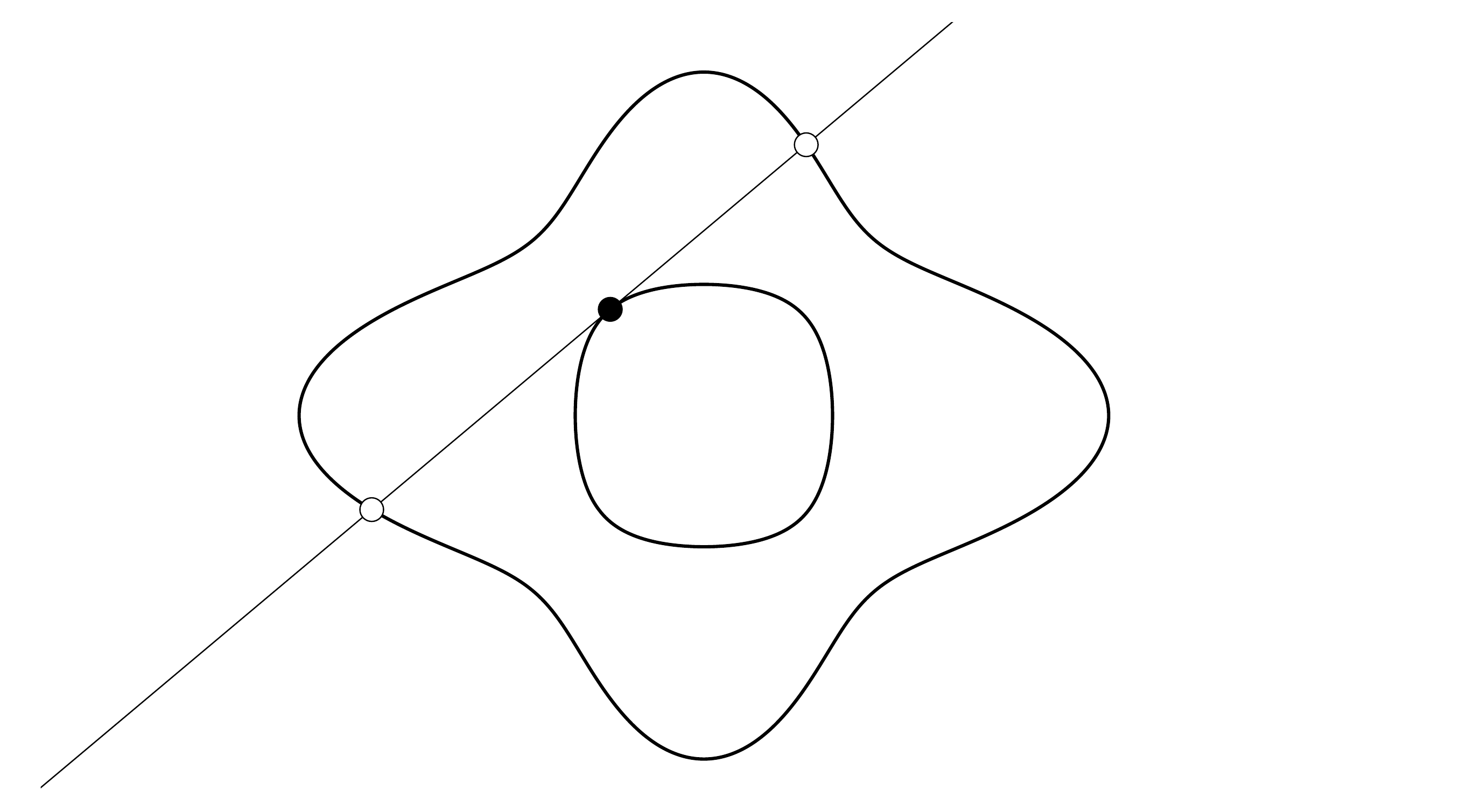}\\
\end{tabular}
\caption{The real points of representative curves for the $20$ connected components of the moduli space $(\mathcal{Q}_1^\circ)^\mathbb{R}$.}
\label{20classes}
\end{figure}

This classification is more subtle then the topological one: the second and fourth picture from the first row are topologically equivalent but represent different components in the moduli space. Something similar occurs in the study of smooth real plane curves of higher degree. For example there exist two smooth real plane curves of degree six with the same configuration of ovals that are non-isomorphic. For a discussion of this phenomenon we refer to \cite{DegtyarevKharlamov} Section $4.8$. 

Our second result concerns the structure of the $20$ components of the moduli space. For this we start by extending work of Looijenga \cite{Looijenga3}, \cite{Looijenga2} on moduli spaces of complex del Pezzo pairs to the real numbers. We focus on del Pezzo surfaces of degree two as these are related to plane quartic curves. A similar treatment could be given for del Pezzo surfaces of other degrees. To formulate our result we introduce some notation. Let $Q$ be a root lattice of type $E_7$ and define the complex adjoint torus $\mathbb{T}=\Hom(Q,\mathbb{C}^\ast)$. The Weyl group $W$ of type $E_7$ acts by reflections on $Q$ and thus on $\mathbb{T}$. The set of fixed points in $\mathbb{T}$ of a reflection in $W$ is called a toric mirror. We denote by $\mathbb{T}^\circ$ the complement of all toric mirrors in $\mathbb{T}$ for this group action of $W$. Now we can state our second result.

\begin{thm}\label{introthm}
Let $(\mathcal{Q}_1^\circ)^\mathbb{R}$ be the moduli space representing real isomorphism classes of pairs $(C,p)$ with $C$ a smooth real plane quartic curve and $p\in C(\mathbb{R})$ a general real point. There is an isomorphism of real orbifolds
\begin{equation}\label{introperiod}
(\mathcal{Q}_1^\circ)^\mathbb{R} \longrightarrow \left( W\backslash \mathbb{T}^\circ \right)(\mathbb{R})
\end{equation}
where by definition the right hand side consists of all $W$-orbits of $t\in \mathbb{T}^\circ$ such that $u\cdot t = \bar{t}$ for some involution $u\in W$.
\end{thm}

This paper is organized in two parts. The first part consists of Sections 2 through 6 and is dedicated to the proof of Theorem \ref{introthm}. The necessary preliminaries on del Pezzo surfaces and involutions in Weyl groups are presented in Sections 2 through 5 and the proof of Theorem \ref{introthm} is presented in Section 6. 

The second part consists of Section 7 through 9 and is dedicated to the proof of Theorem \ref{introthm1}. The main idea of the proof is to use the isomorphism of Equation \ref{introperiod} to study the moduli space $(\mathcal{Q}_1^\circ)^\mathbb{R}$ using results on root systems and involutions in Weyl groups on the right hand side. In Section $7$ we study the real points of such torus quotients and their connected components for general root systems of type $ADE$. For type $E_7$ we then prove that there are twenty connected components. In Sections 8 and 9 we relate these twenty components to the pictures of Figure \ref{20classes}. This completes the proof of Theorem \ref{introthm1}. A possible alternative approach to obtain the classification of Theorem \ref{introthm1} would be to consider trigonal curves with a single node in the Hirzebruch surface $\Sigma_2$. This construction is described in\cite{Zvonilov}. However this approach would not give a description of the moduli space and its components.

\begin{ack}
The author would like to thank Gert Heckman for suggesting this research topic and for many inspiring and useful discussions, and Professor Looijenga and Professor Kharlamov for interesting discussions. This research was supported by NWO free competition grant number 613.000.909. 
\end{ack}

\section{Involutions in Coxeter groups}\label{SectInvCox}

In order to make the right-hand side of Equation \ref{introperiod} more explicit we need to understand the conjugacy classes of involutions in the Weyl group of type $E_7$. Weyl groups can be realized as finite Coxeter groups. The classification of conjugacy classes of involutions in a Coxeter group was done by Richardson \cite{Richardson} and Springer \cite{Springer}. Before this the classification of conjugacy classes of elements of finite Coxeter groups was obtained by Carter \cite{Carter}. In this section we give a brief review of these results. 

\begin{dfn}\label{Coxeter}A Coxeter system is a pair $(W,S)$ with $W$ a group presented by a finite set of generators $S=\{s_1,\ldots,s_r\}$ subject to relations
\begin{align*}
(s_is_j)^{m_{ij}}=1 \quad \text{with} \quad 1\leq i,j \leq r
\end{align*}
where $m_{ii}=1$ and $m_{ij}=m_{ji}$ are integers $\geq 2$. We also allow $m_{ij}=\infty$ in which case there is no relation between $s_i$ and $s_j$. These relations are encoded by the Coxeter graph of $(W,S)$. This is a graph with $r$ nodes labeled by the generators. Nodes $i$ and $j$ are not connected if $m_{ij}=2$ and are connected by $m_{ij}-2$ edges otherwise. If $m_{ij}=\infty$ we connect the vertices by a thick edge.
\end{dfn} 

For a Coxeter system $(W,S)$ we define an action of the group $W$ on the real vector space $V$ with basis $\{e_s\}_{s\in S}$. First we define a symmetric bilinear form $B$ on $V$ by the expression
\[ B(e_i,e_j) = 2 \cos \left(\frac{\pi}{m_{ij}}\right). \]
Then for each $s_i\in S$ the reflection: $s_i(x)=x-B(e_i,x)e_i$ preserves this form $B$. In this way we obtain a homomorphism $W\rightarrow GL(V)$ called the geometric realization of $W$. For each subset $I \subseteq S$ we can form the standard parabolic subgroup $W_I<W$ generated by the elements $\{s_i ; i \in I \}$ acting on the subspace $V_I$ generated by $\{e_i\}_{i\in I}$. We say that $W_I$ (or also $I$) satisfies the $(-1)$-condition if there is a $w_I\in W_I$ such that $w_I \cdot x = -x$ for all $x\in V_I$. The element $w_I$ necessarily equals the longest element of $(W_I,S_I)$. This implies in particular that $W_I$ is finite. Let $I,J\subseteq S$, we say that $I$ and $J$ are $W$-equivalent if there is a $w\in W$ that maps $\{e_i \}_{i \in I}$ to $\{ e_j\}_{j \in J}$. Now we can formulate the main theorem of \cite{Richardson}.

\begin{thm}[Richardson]Let $(W,S)$ be a Coxeter system and let $\mathcal{J}$ be the set of subsets of $S$ that satisfy the $(-1)$-condition. Then 
\begin{enumerate}
\item If $c \in W$ is an involution, then $c$ is conjugate in $W$ to $w_I$ for some $I \in \mathcal{J}$.
\item Let $I,J \in \mathcal{J}$. The involutions $w_I$ and $w_J$ are conjugate in $W$ if and only if $I$ and $J$ are $W$-equivalent.
\end{enumerate}
\end{thm}

This theorem reduces the problem of finding all conjugacy classes of involutions in $W$ to finding all $W$-equivalent subsets in $S$ satisfying the $(-1)$-condition. First we determine which subsets $I\subseteq S$ satisfy the $(-1)$-condition, then we present an algorithm that determines when two subsets $I,J\subseteq S$ are $W$-equivalent. If $(W_I,S_I)$ is irreducible and satisfies the $(-1)$-condition then it is of one of the following types
\begin{align}\label{list}
A_1,B_n,D_{2n},E_7,E_8,F_4,G_2,H_3,H_4,I_2(2p)
\end{align}
with $n,p\in \mathbb{N}$ and $p\geq 4$. If $(W_I,S_I)$ is reducible and satisfies the $(-1)$-condition then $W_I$ is the direct product of irreducible, finite standard parabolic  subgroups $(W_i,S_i)$ from (\ref{list}). The Coxeter diagrams of the $(W_i,S_i)$ occur as disjoint subdiagrams of the types in the list of the diagram of $(W,S)$. The element $w_I$ is the product of the $w_{I_i}$ which act as $-1$ on the $V_{I_i}$. Now let $K \subseteq S$ be of finite type and let $w_K$ be the longest element of $(W_K,S_K)$. The element $\tau_K=-w_K$ defines a diagram involution of the Coxeter diagram of $(W_K,S_K)$ which is non-trivial if and only if $w_K \neq -1$. If $I,J\subseteq K$ are such that $\tau_K I = J$ then $I$ and $J$ are $W$-equivalent. To see this, observe that $w_Kw_I \cdot I = w_K \cdot (-I) = \tau_K I=J$. Now we define the notion of elementary equivalence.

\begin{dfn}We say that two subsets $I,J\subseteq S$ are elementary equivalent, denoted by $I \vdash J$, if $\tau_K I = J$ with $K=I\cup \{\alpha \} = J\cup \{ \beta \}$ for some $\alpha,\beta \in S$.
\end{dfn}

It is proved in \cite{Richardson} that $I$ and $J$ are $W$-equivalent if and only if they are related by a chain of elementary equivalences: $I = I_1 \vdash I_2 \vdash \ldots \vdash I_n = J$. This provides a practical algorithm to determine all the conjugacy classes of involutions in a given Coxeter group $(W,S)$ using its Coxeter diagram
\begin{enumerate}
\item Make a list of all the subdiagrams of the Coxeter diagram of $(W,S)$ that satisfy the $(-1)$-condition. These are exactly the disjoint unions of diagrams in the list (\ref{list}). Every involution in $W$ is conjugate to $w_K$ with  $K$ a subdiagram in this list. 
\item Find out which subdiagrams of a given type are $W$-equivalent by using chains of elementary equivalences.
\end{enumerate}

\begin{ex}[$E_7$]\label{invE7}We use the procedure described above to determine all conjugacy classes of involutions in the Weyl group of type $E_7$. This result will be used many times later on. Since $W_7$ contains the element $-1$ the conjugacy classes of involutions come in pairs $\{u,-u\}$. We label the vertices of the Coxeter diagram as in Figure \ref{labele7}

\begin{figure}
\begin{center}
\begin{tikzpicture}
  \tikzstyle{every node}=[circle,draw]  
    \node[label=60:$1$] (1) at ( 0,0) {};
    \node[label=60:$2$] (2) at ( 1,0) {};    
    \node[label=60:$3$] (3) at ( 2,0) {};
    \node[label=60:$4$] (4) at ( 3,0) {};
    \node[label=60:$7$] (7) at ( 2,1) {};
    \node[label=60:$5$] (5) at ( 4,0) {};
    \node[label=60:$6$] (6) at ( 5,0) {}; 
    \draw [-] (1) -- (2) -- (3) -- (4) -- (5) -- (6);
    \draw [-] (3) -- (7);   
\end{tikzpicture}
\caption{The labelling of the nodes of the $E_7$ diagram}
\label{labele7}
\end{center}
\end{figure}
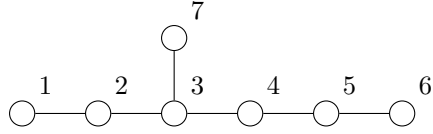

It turns out that all involutions of a given type are equivalent with the exception of type $A_1^3$: here there are two non-equivalent involutions as seen in Figure \ref{A13}.  The types of involutions that occur are
\begin{equation}\label{pairs}
\{ 1,E_7\} \ ,\ \{A_1,D_6\} \ ,\ \{A_1^2,D_4A_1\} \ ,\ \{A_1^3,A_1^4\} \ ,\ \{D_4,A_1^{3\prime}\}.
\end{equation}

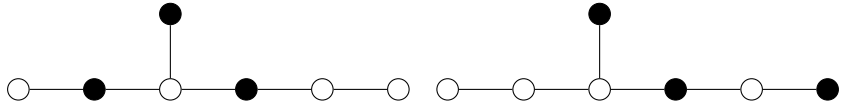
\begin{figure}
\centering
\subfigure{
\begin{tikzpicture}
	\node[circ] (1) at (0,0) {};
	\node[blackcirc] (2) at (1,0) {};
	\node[circ] (3) at (2,0) {};
	\node[blackcirc] (4) at (3,0) {};
	\node[circ] (5) at (4,0) {};
	\node[circ] (6) at (5,0) {};
	\node[blackcirc] (7) at (2,1) {};	
	\draw [-] (1) -- (2) -- (3) -- (4) -- (5) -- (6);
	\draw [-] (3) -- (7);	
\end{tikzpicture} 
}
\subfigure{
\begin{tikzpicture}
	\node[circ] (1) at (0,0) {};
	\node[circ] (2) at (1,0) {};
	\node[circ] (3) at (2,0) {};
	\node[blackcirc] (4) at (3,0) {};
	\node[circ] (5) at (4,0) {};
	\node[blackcirc] (6) at (5,0) {};
	\node[blackcirc] (7) at (2,1) {};	
	\draw [-] (1) -- (2) -- (3) -- (4) -- (5) -- (6);
	\draw [-] (3) -- (7);	
\end{tikzpicture} 
}
\caption[test]{The involutions $A_1^3$ (left) and $A_1^{3\prime}$ (right).}
\label{A13}
\end{figure}

For example, consider the two subdiagrams of type $A_1$ with vertices $\{1\}$ and $\{2\}$. The diagram automorphism $\tau_{\{1,2\}}$ which is of type $A_2$ exchanges the vertices $\{1\}$ and $\{2\}$, so they are elementary equivalent. One shows in a similar way that all diagrams of type $A_1$ are equivalent. 
\end{ex}

\section{Del Pezzo surfaces}

In this section we review the theory of del Pezzo surfaces over the real and the complex numbers. For del Pezzo surfaces over the complex numbers two excellent references for the proofs in this section are \cite{Manin} or \cite{DolOrt}.

\subsection{Complex del Pezzo surfaces}\label{ComplexDelPezzo}

\begin{dfn}A del Pezzo surface $Y$ is a smooth, complex projective surface whose anticanonical system $|-K_Y|$ is ample. The degree of $Y$ is the self-intersection number: $d=K_Y\cdot K_Y$ of the canonical class in the Picard group $\Pic(Y)$ of $Y$. It is an integer with $1\leq d \leq 9$.
\end{dfn}

A del Pezzo surface can be realized as the blowup of a configuration of points in the projective plane. This is expressed by the following theorem.

\begin{thm}\label{DP}
A del Pezzo surface of degree $d$ is isomorphic to either
\begin{enumerate}

\item The blowup $Y=\Bl_B \mathbb{P}^2$ of the projective plane in a set
\[ B=\{P_1,\ldots,P_{r}\} \subset \mathbb{P}^2(\mathbb{C}) \] 
of $r=9-d$ points in general position  ($1\leq d \leq 9$). A point set is in general position if no $3$ points are collinear, no $6$ are on a conic and no $8$ are on a cubic which is singular at one of these points. 

\item The smooth quadric $\mathbb{P}^1\times \mathbb{P}^1$ in which case $d=8$.
\end{enumerate}
\end{thm}
From now on we only consider del Pezzo surfaces of the first kind. Exhibiting a del Pezzo surface as a blowup $\pi:Y\rightarrow \mathbb{P}^2$ fixes a basis of the Picard group $\Pic(Y)$. This basis consists of the classes $E_i=\pi^{-1}(P_i)$ with $1\leq i \leq r$ of the exceptional curves over the blown up points and the class $E_0$ of the strict transform of a general line in $\mathbb{P}^2$. The anticanonical class expressed in this basis of $\Pic(Y)$ is given by
\[ -K_Y = 3E_0-E_1-\ldots-E_r. \]
It is represented by the strict transform of a cubic in $\mathbb{P}^2$ through the points $B=\{P_1,\ldots,P_r\}$. We also write $\Pic^0(Y)$ for the orthogonal complement of $-K_Y$ in $\Pic(Y)$. From the description of a del Pezzo surface as a blowup of $\mathbb{P}^2$ it follows that the Picard group $\Pic(Y)$ is isomorphic to the hyperbolic lattice $\mathbb{Z}_{1,r}$ of rank $r+1$ and signature $(1,r)$. It has a basis $\{e_0,\ldots,e_r\}$ with inner product defined by the relations
\[
\left\{
\begin{array}{ll}
e_0 \cdot e_0 =1 & \\
e_i\cdot e_i = -1 & \text{for} \ 1\leq i \leq r \\
e_i\cdot e_j = 0 &  \text{for} \ i\neq j.  
\end{array}
\right.
\]
An isomorphism $\phi: \mathbb{Z}_{1,r} \rightarrow \Pic(Y)$ is called a \emph{marking} of the del Pezzo surface $Y$ if it maps the element $k=-3e_0+e_0+\ldots+e_r$ to the canonical class $K_Y$ of $\Pic(Y)$. An isomorphism $(Y,\phi)\cong (Y^\prime,\phi^\prime)$ of marked del Pezzo surfaces is an isomorphism $F:Y\rightarrow Y^\prime$ such that the following diagram commutes.
\begin{center}
\begin{tikzcd}
\mathbb{Z}_{1,r} \arrow{r}{\phi} \arrow{dr}{\phi^\prime} & \Pic(Y) \arrow{d}{F_\ast} \\
& \Pic(Y^\prime)
\end{tikzcd}
\end{center} 

Exhibiting a del Pezzo surface as a blowup $\pi:Y\rightarrow \mathbb{P}^2$ is equivalent to adding a marking to $Y$; from the marking $\phi$ we recover the blowup map by blowing down the exceptional curves $\phi(e_i)$ for $1\leq i \leq r$. This determines a set $B=\{P_1,\ldots,P_r\}$ of $r$ points in general postition in $\mathbb{P}^2$. If two marked del Pezzo surfaces are isomorphic then the corresponding point sets $B$ and $B^\prime$ are related by an element of $\PGL(3,\mathbb{C})$. As a consequence the elements of the space 
\begin{equation}
\widetilde{\mathcal{DP}}_d=\left( (\mathbb{P}^2)^r - \Delta \right) / \PGL(3,\mathbb{C})
\end{equation}
represent isomorphism classes of marked del Pezzo surfaces of degree $d=9-r$. Here $\Delta$ denotes the set of configurations of $r$ points in $\mathbb{P}^2$ \emph{not} in general position in the sense of Theorem \ref{DP}. For an $r$-tuple of points in $\mathbb{P}^2$ in general position with $r\geq 4$ there is a unique element of $\PGL(3,\mathbb{C})$ that maps the points to the configuration of points represented by the columns of the matrix
\[ 
\begin{pmatrix}
1 & 0 & 0 & 1 & x_1 & \ldots & x_{r-4} \\
0 & 1 & 0 & 1 & y_1 & \ldots & y_{r-4} \\
1 & 1 & 1 & 1 & 1 & \dots & 1
\end{pmatrix}.
\]
This implies that $\widetilde{\mathcal{DP}}_d$ is isomorphic to an open subset of $(\mathbb{A}^2)^{r-4}$.

\subsection{The Cremona action of the Weyl group}\label{Weylgroup}

The stabilizer of the element $k=-3e_0+e_1+\ldots+e_r$ in the orthogonal group $O(\mathbb{Z}_{1,r})$ is a finite Coxeter group of type
\begin{equation}\label{delpezzotype}A_1A_2,A_4,D_5,E_6,E_7,E_8 
\end{equation} 
for $r=3,\ldots,8$. A set of generators $S=\{s_1,\ldots,s_r\}$ for $W_r$ is given by the reflections in the simple roots
\[ \alpha_1=e_1-e_2 \ , \ldots, \ \alpha_{r-1}=e_{r-1}-e_{r} \ , \ \alpha_{r}=e_0-e_1-e_2-e_3. \]
These root span a lattice $Q_r$ that is precisely the orthogonal complement of $k^\perp$ in $\mathbb{Z}_{1,r}$. The group $W_r$ acts on marked del Pezzo surfaes by composing with the marking: $w\cdot (X,\phi)=(X,\phi \circ w^{-1})$ for $w\in W_r$. This action is simply transitive so that the orbit space $\mathcal{DP}_d=W_r \backslash \widetilde{\mathcal{DP}}_d$ is a coarse moduli space for del Pezzo surfaces of degree $d$.

We now describe this action on the set of blown up points in $\mathbb{P}^2$. Suppose $(X,\phi)$ is a marked del Pezzo surface and $\pi:X\rightarrow \mathbb{P}^2$ is the corresponding blowing up map with $B\subseteq \mathbb{P}^2$ the set of blown up points. If $\phi^\prime$ is another marking of $X$ then $\phi^\prime = \phi \circ w$ for some element $w\in W_r$. The element $w$ defines a birational transformation $\rho(w)$ of $\mathbb{P}^2$ in the following way: first blow up $\mathbb{P}^2$ in the points of $B$. Then blow down the exceptional curves $\phi^\prime(e_i)=\phi(w\cdot e_i)$ for $1\leq i \leq r$. This determines a new set of points $B^\prime$ and blowup map $\pi^\prime:X\rightarrow \mathbb{P}^2$ corresponding to $\phi^\prime$ such that the following diagram commutes.
\begin{center}
\begin{tikzcd}
X  \arrow{d}{\pi} \arrow{dr}{\pi^\prime} & \\
\mathbb{P}^2 \arrow[dashed]{r}{\rho(w)} & \mathbb{P}^2
\end{tikzcd}
\end{center}
In this way we obtain a homomorphism of the Weyl group $W_r$ to the group of birational transformations of $\mathbb{P}^2$
\[\rho: W_r \rightarrow \Bir(\mathbb{P}^2)\]
We can calculate this representation on the set $S$ of simple reflections. The element $\rho(s_i)$ with $1 \leq i \leq r-1$ acts on $\mathbb{Z}_{r,r}$ by
\begin{align*}
e_i &\mapsto e_j \\
e_j &\mapsto e_i \\
e_k &\mapsto e_k \quad k\neq i,j. 
\end{align*}
so that it corresponds to the transposition of the points $P_i$ and $P_{i+1}$. The element $s_r$ gives a more interesting transformation. It acts on $\mathbb{Z}_{1,r}$ by
\begin{align*}
e_1 &\mapsto e_0-e_2-e_3 \\
e_2 &\mapsto e_0-e_1-e_3 \\
e_3 &\mapsto e_0-e_1-e_2 \\
e_i &\mapsto e_i   \quad 4\leq i \leq r.
\end{align*} 
Geometrically this means that $\rho(s_r)$ is obtained by first blowing up $P_1,P_2$ and $P_3$ and then blowing down the strict transforms of the lines connecting them. This birational transformation $\rho(s_r)$ is called the standard triangular Cremona transformation based in $P_1,P_2$ and $P_3$. A simple calculation shows that $s_7(2e_0-e_1-e_2-e_3)=e_0$ so that the image of a conic through $P_1,P_2,P_3$ under the standard triangular Cremona transformation is a line. If we assume that these points are 
\[P_1=(1:0:0) \ ,\ P_2=(0:1:0) \ ,\ P_3=(0:0:1).\]
then $\rho(s_r)$ is given by $(x:y:z) \mapsto (yz:xz:xy)$. To summarise: the group $W_r$ acts on $(\mathbb{P}^2)^r-\Delta$ by permuting the points and by standard triangular Cremona transformations centered in triples of distinct points. 

\subsection{Del Pezzo surfaces of degree two}\label{dpdegree2}
Suppose that $Y$ is a del Pezzo surface of degree two so that it is isomorphic to the blowup of the projective plane $\mathbb{P}^2$ in $7$ points. The anticanonical system of $Y$ defines a morphism 
\[ \vert -K_Y \vert :Y \rightarrow \mathbb{P}^2.\]
It is a double cover of $\mathbb{P}^2$ branched along a smooth quartic curve $C\subset \mathbb{P}^2$. Conversely a smooth quartic $C=\{ f(x,y,z)=0 \}$ determines a del Pezzo surface $Y$ of degree two by the formula
\begin{equation}\label{ddelpezzo}
Y = \{ w^2 = f(x,y,z) \} \subset \mathbb{P}(2,1,1,1).
\end{equation} 
Consequently every del Pezzo surface $Y$ has a special involution that corresponds to the deck transformation of the double cover $Y\rightarrow \mathbb{P}^2$. This is called the Geiser involution. In terms of Equation \ref{ddelpezzo} this involution is given by
\[ \rho_Y : [w:x:y:z] \mapsto [-w:x:y:z].\]  
If we choose a marking $\phi:\Pic(Y)\rightarrow \mathbb{Z}_{1,7}$ then the induced involution on the lattice $\mathbb{Z}_{1,7}$ is given by
\begin{equation}\label{geisereq}
\rho:x \mapsto -x+(x\cdot k)k.
\end{equation}
It fixes the element $k$ and acts as $-1$ on $k^\perp$ so that it corresponds to the central element $-1\in W(E_7)$. An element $e\in \mathbb{Z}_{1,7}$ that satisfies $e\cdot e=-1$ and $e\cdot k=-1$ is called \emph{exceptional}. The set $\mathcal{E}$ of exceptional elements forms a single $W(E_7)$-orbit and consists of the $56$ elements
\begin{enumerate}
\item $e_i$ with $1\leq i \leq 7$, the class of the exceptional divisor $E_i$.
\item $l_{ij}=e_0-e_i-e_j$, the class of the strict transform of the line $L_{ij}$ through $P_i$ and $P_j$. 
\item $c_{ij}=-k-l_{ij}=2e_0-e_1-\ldots-\hat{e}_i-\ldots-\hat{e}_j-\ldots -e_7$, the class of the strict transform of the conic  $C_{ij}$ through $5$ of the $7$ points. 
\item $k_i=-k-e_i=3e_0-e_1-\ldots-2e_i-\ldots -e_7$, the class of the strict transform of the cubic $K_i$ through $6$ points with a node at a seventh point. 
\end{enumerate}
The elements of $\mathcal{E}$ come in $28$ pairs $(e_i,k_i)$, $(l_{ij},c_{ij})$ whose elements are interchanged by the involution $\rho$. The geometric meaning of this is as follows. An exceptional element $E\in \Pic(Y)$ corresponds to a exceptional curve on the del Pezzo surface $Y$ and there are $56$ of these. The two elements of a pair $(E,\rho_Y(E))$ are mapped to a single bitangent of the quartic curve by the anticanonical map. This accounts for all $28$ bitangents of a smooth plane quartic curve.

\subsection{Real del Pezzo surfaces}\label{realdpsurfaces}
We review some results of Wall \cite{Wall} on real del Pezzo surfaces. Other references on this subject are Koll\'ar \cite{Kollar} and Russo \cite{Russo}. A real del Pezzo surface is a pair $(Y,\chi_Y)$ with $Y$ a complex del Pezzo surface and $\chi_Y:Y\rightarrow Y$ a real form on $Y$. The real points $Y(\mathbb{R})$ of $Y$ are the fixed points under $\chi_Y$. The action of $\chi$ induces an involution $\chi_Y^\ast$ on the Picard group $\Pic(Y)$ which preserves the canonical class $K$ and the intersection product. By fixing a marking $\phi:\mathbb{Z}_{1,r}\rightarrow \Pic(Y)$ we obtain an involution of the lattice $\mathbb{Z}_{1,r}$  by the formula
\[ \chi = \phi^{-1}\circ \chi_Y^\ast \circ \phi.\]
The involution $\chi$ preserves the element $k\in \mathbb{Z}_{1,r}$. As we have seen such an involution corresponds to an involution $u$ in the Weyl group $W_r$. The conjugacy class of this involution in $W_r$ is an important invariant of the real structure on $Y$.

A real del Pezzo surface $Y$ of degree two is the double cover of the projective plane $\mathbb{P}^2$ ramified over a smooth real plane quartic curve $C\subset \mathbb{P}^2$ so that
\[ Y = \{ w^2 = f(x,y,z) \}. \]
We fix the sign of $f$ so that $f>0$ on the orientable interior part of $C(\mathbb{R})\subset \mathbb{P}^2(\mathbb{R})$. By using the deck transformation $\rho_Y$ of the cover we see that there are two real forms of $Y$:
\begin{equation}
\begin{aligned}
\chi_Y^+: \left[ w:x:y:z \right] & \mapsto [\bar{w}:\bar{x}:\bar{y}:\bar{z}] \\
\chi_Y^-: [w:x:y:z] & \mapsto [-\bar{w}:\bar{x}:\bar{y}:\bar{z}]. \\
\end{aligned}
\end{equation}\label{realstructuredelpezzo}
These real forms satisfy: $\chi_Y^-=\rho_Y \circ \chi_Y^+$ and we denote the real point sets of $\chi_Y^+$ and $\chi_Y^-$ by $Y^+(\mathbb{R})$ and $Y^-(\mathbb{R})$ respectively. Note that $Y^+(\mathbb{R})$ is an orientable surface while $Y^-(\mathbb{R})$ is non-orientable. 

In \cite{Wall} Wall determines the correspondence between the conjugacy classes of the $u\in W(E_7)$ and the topological type of $Y(\mathbb{R})$. The results are shown in Table \ref{realdelpezzotable}. We use the notation $kX$ for the disjoint union and $\#_k X$ for the connected sum of $k$ copies of a real surface $X$. From this table we see that except for the classes of $D_4$ and $A_1^{3\prime}$ the conjugacy class of $u\in W(E_7)$ determines the topological type of the real plane quartic curve $C(\mathbb{R})$.

\begin{table}

\[
\begin{array}{ccll}
\toprule
j& C(\mathbb{R}) & u\in W(E_7) & Y(\mathbb{R}) \\
\midrule
\multirow{2}{*}{1} & \multirow{2}{*}{\begin{tikzpicture}\draw (45:10pt) circle [radius=5pt];\draw (135:10pt) circle [radius=5pt];\draw (-135:10pt) circle [radius=5pt];\draw (-45:10pt) circle [radius=5pt];\end{tikzpicture}} & 1 & \#_8\mathbb{P}^2(\mathbb{R}) \\
& & E_7 & 4S^2\\
\\
\multirow{2}{*}{2} & \multirow{2}{*}{\begin{tikzpicture}\draw (0:8pt) circle [radius=5pt];\draw (120:8pt) circle [radius=5pt];\draw (240:8pt) circle [radius=5pt];\end{tikzpicture}} & A_1 & \#_6\mathbb{P}^2(\mathbb{R}) \\
& & D_6 & 3S^2\\
\\
\multirow{2}{*}{3} & \multirow{2}{*}{\begin{tikzpicture}\draw (0:7pt) circle [radius=5pt]; \draw (180:7pt) circle [radius=5pt];\end{tikzpicture}} & A_1^2 & \#_4\mathbb{P}^2(\mathbb{R}) \\
& & D_4A_1 & 2S^2\\
\\
\multirow{2}{*}{4} & \multirow{2}{*}{\begin{tikzpicture}\draw (0:0) circle [radius=5pt];\end{tikzpicture}} & A_1^3 & \#_2\mathbb{P}^2(\mathbb{R}) \\
& & A_1^4 & S^2\\
\\
\multirow{2}{*}{5} & \multirow{2}{*}{ \begin{tikzpicture}[baseline = -3pt]\draw (0,0) circle [radius=8pt];\draw (0,0) circle [radius=4pt];\end{tikzpicture}} & D_4 & S^2 \sqcup \#_2\mathbb{P}^2(\mathbb{R}) \\
& & A_1^{3\prime} & S^1 \times S^1 \\
\\
 \multirow{2}{*}{6} & \multirow{2}{*}{ $\emptyset$ } & D_4 &  2\mathbb{P}^2(\mathbb{R}) \\
&  & A_1^{3\prime} & \emptyset \\
\bottomrule
\end{array}
\]
\caption{The real topological types of real del Pezzo surfaces of degree two and their corresponding involutions in the Weyl group $W(E_7)$.}
\label{realdelpezzotable}

\end{table}

\section{Moduli of del Pezzo pairs}

In this section we study del Pezzo surfaces obtained by blowing up $r$ points on a fixed plane singular cubic. The strict transform of this cubic is a singular anti-canonical curve on the del Pezzo surface. This is the situation studied by Looijenga for del Pezzo surfaces of degree two in \cite{Looijenga3} and for general del Pezzo surfaces in \cite{Looijenga2}. In this article we restrict ourselves to the case of del Pezzo surfaces of degree two. 

\begin{dfn}A del Pezzo pair of degree two is a pair $(Y,Z)$ consisting of a del Pezzo surface $Y$ of degree two and a singular anti-canonical curve $Z\subset Y$. We denote the moduli space of del Pezzo pairs of degree two by $\mathcal{DPP}_2$. By adding a marking to the del Pezzo surface $Y$ we obtain a  marked del Pezzo pair $(Y,Z,\phi)$ with $\phi:\mathbb{Z}_{1,7}\rightarrow \Pic(Y)$ a marking of $Y$. The moduli space of marked del Pezzo pairs is denoted by $\widetilde{\mathcal{DPP}}_2$. 
\end{dfn}

The smooth points of an irreducible plane cubic admit a group law. For smooth cubics this is well known. A similar construction for the group law can be applied to singular irreducible cubics as follows. Let $Z$ be a irreducible plane cubic curve and let $O$ be an inflection point of $Z$. The map

\begin{align*}
Z^{\ns}(\mathbb{C}) &\rightarrow \Pic^0(Z) \\
P &\mapsto [P]-[O] 
\end{align*}
is a bijection and defines a group law on $Z^{\ns}(\mathbb{C})$. For a nodal cubic it is well known that there is an isomorphism of groups: $\Pic^0(Z)\cong \mathbb{C}^\ast$. It is unique up to multiplication by an element of $\Aut(\mathbb{C}^\ast) \cong \{\pm1\}$. Similarly for a cuspidal cubic we have an isomorphism $\Pic^0(Z)\cong \mathbb{C}$ that is unique up to multiplication by an element of $\Aut(\mathbb{C}) \cong \mathbb{C}^\ast$.  A useful property of the group law is the following.

\begin{prop}\label{curves}
Let $Z$ be a plane cubic curve and let $P_1,\ldots P_{3d}$ be points on $Z^{\ns}$. Then $\sum_{i=1}^{3d}P_i=0$ if and only if $\{P_1,\ldots,P_{3d}\}=C^{ns}\cap D$ for some plane curve $D$ of degree $d$. In particular three points of $Z^{\ns}$ add up to zero if and only if they are colinear.
\end{prop}
\begin{proof}The condition $\sum_{i=1}^{3d}P_i=0$ is equivalent to $\sum_{i=1}^{3d}\left( [P_i]-[O] \right)=0$ in $\Pic^0(C)$. This implies that the divisor $\sum_{i=1}^{3d}\left(P_i -O\right)$ is principal of the form $\Div(f/g^d)$ with $g$ the equation of the flex line at $O$ and $f$ a homogeneous polynomial of degree $d$ which defines the curve $D$. 
\end{proof}

Suppose that $(Y,Z,\phi)$ is a marked del Pezzo pair and $\pi:Y\rightarrow \mathbb{P}^2$ is the corresponding blowup map. The image $\pi(Z)$ is a plane cubic through the seven points $B=\{P_1,\ldots,P_7\}$. In Table \ref{Pezzostrata} we distinguish four cases according to the type (nodal or cuspidal) of $Z_B$  and the location of the points $B$. We will use the symbols used in Kodaira's classification of the singular fibers of an elliptic pencil to denote the type of the curve $Z$.

\begin{table}
\centering
\begin{tabular}{l  p{9cm}}
\toprule
$Z$ & $Z_B$ \\
\midrule
$I_1$ & Irreducible cubic with a node and $B\subset Z^{\ns}$. \\
$II$ & Irreducible cubic with a cusp and $B\subset Z^{\ns}$. \\
$I_2$ & Irreducible cubic with a node that coincides with a blown up point or reducible cubic consisting of a conic and a line intersecting in two points.\\ 
$III$ & Irreducible cubic with a cusp that coincides with a blown up point or reducible cubic consisting of a conic and a tangent line.\\
\bottomrule
\end{tabular}
\caption{Strata in the moduli space of del Pezzo pairs $(Y,Z)$ of degree two according to the Kodaira type of $Z$.}
\label{Pezzostrata}
\end{table}

These four types of $Z\subset Y$ each define a stratum in the moduli space of del Pezzo pairs. The stratum of type $I_1$ where $Z_B$ is an irreducible nodal cubic and $B\subset Z_B^{\ns}$ is generic and defines an open subset $\mathcal{DPP}^\circ_2 \subset \mathcal{DPP}_2$. For now we assume that $Z$ is of type $I_1$ and we identify $Z$ with $Z_B$ so that we can make use of the group law on the singular cubic $Z_B$. By composing the marking $\phi:\mathbb{Z}_{1,7} \rightarrow \Pic(Y)$ with the restriction homomorphism $\Pic(Y)\rightarrow \Pic(Z)$ we obtain a map that assigns $e_i\mapsto [P_i]$ for $1\leq i \leq r$ and $e_0\mapsto 3[O]$ where $[O]$ is an inflection point of $Z$. Restricting this map to the root lattice $Q<\mathbb{Z}_{1,7}$ induces a homomorphism $\chi \in \Hom(Q,\Pic^0(Z))$ characterized by the relations
\begin{equation}\label{pointsroots}
\begin{aligned}
\chi(e_i-e_{i+1}) &= [P_i]-[P_{i+1}] \\
 \chi(e_0-e_1-e_2-e_3) &= 3[O]-[P_1]-[P_2]-[P_3].
\end{aligned}
\end{equation}  
\begin{prop}No root lies in the kernel of the homomorphism $\chi:Q\rightarrow \Pic^0(Z)$.
\end{prop}
\begin{proof}
From the construction of $\chi$ and Proposition \ref{curves} we see that
\begin{align*}
\chi(e_i-e_j)=0 &\Leftrightarrow P_i=P_j \\
\chi(e_0-e_i-e_j-e_k)=0 &\Leftrightarrow P_i,P_j,P_k \ \text{are colinear} \\
\chi(2e_0-e_1-\ldots -\hat{e}_i-\ldots -e_7)=0 &\Leftrightarrow P_1,\ldots,\hat{P}_i,\ldots,P_7 \ \text{are conconic}
\end{align*}
so that the condition that the points are in general position is equivalent to $\chi(\alpha)\neq 0$ for all roots $\alpha \in R$. 
\end{proof}
After fixing an isomorphism $\Pic^0(Z)\cong \mathbb{C}^\ast$ we can identify the space 
\[ \Hom(Q,\Pic^0(Z)) \]
with the complex torus $\mathbb{T}=\Hom \left(Q,\mathbb{C}^\ast \right)$. This identification is not canonical but is unique up to multiplication by an element of $\Aut(\mathbb{C}^\ast) \cong \{\pm 1 \}$ which acts on $\mathbb{T}$. The Weyl group $W$ of type $E_7$ acts on $\mathbb{T}$ by its natural action on $Q$ and we denote the complement of the toric mirrors for this action by $\mathbb{T}^\circ$.  

\begin{thm}[Looijenga]\label{BigTheoremC}
Let $(Y,Z,\phi)$ be a marked del Pezzo pair of degree two with $Z$ a nodal anti-canonical curve. The association 
\[ (Y,Z,\phi)\mapsto \left( \chi:Q\rightarrow \Pic^0(Z) \right) \]
extends to an isomorphism of orbifolds
\begin{equation}\label{markediso} 
\widetilde{\mathcal{DPP}}_2^\circ  \rightarrow \{ \pm 1 \} \backslash \mathbb{T}^\circ. 
\end{equation}
The left hand side is the open stratum of the moduli space of marked del Pezzo pairs of degree two with Z of type $I_1$. Similarly we have an isomorphism of orbifolds
\[
\mathcal{DPP}_2^\circ  \rightarrow W \backslash \mathbb{T}^\circ 
\]
\end{thm}

\begin{proof}
Let $\chi$ be an element of $\mathbb{T}^\circ=\Hom(Q,\mathbb{C}^\ast)^\circ$. We construct an inverse to the map of Equation \ref{markediso} by constructing seven points on a fixed nodal cubic $Z$. Fix an isomorphism $\mathbb{C}^\ast \rightarrow Z^{\ns}$ by choosing one of the three inflection points $O$ on $Z$ as a unit element. The group law then satisfies $t_it_jt_k=1$ if and only if the corresponding points $P_i,P_j,P_k$ on $Z^{\ns}$ are colinear. Since the seven points should satisfy (\ref{pointsroots}) they must also satisfy the equality
\begin{equation}\label{pointform}
P_i = \chi(e_i-e_0/3)
\end{equation}\label{Piequality}
where we consider $\chi$ as an element of $\Hom(Q\otimes_\mathbb{Z} \mathbb{C},\mathbb{C}^\ast)$. This determines the seven points uniquely up to addition of an inflection point of $Z^{\ns}$ (or equivalently multiplication by a third root of unity of $\mathbb{C}^\ast$). Blowing up these seven points determines a marked del Pezzo surface $Y$ and the pullback of $Z$ under the blowup map defines a nodal anti-canonical curve on $Y$ isomorphic to $Z$.  
 \end{proof}
 
To conclude this section we obtain explicit descriptions of the standard triangular Cremona transformation centered in three points on an irreducible plane nodal cubic $Z_B$ in terms of the coordinate $t\in \mathbb{C}^\ast \cong Z^{\ns}(\mathbb{C})$.  
The Cremona map $\rho(s_7)$ centered in the points $P_1,P_2,P_3$ of $Z$ with coordinate $t$ maps $Z$ to another nodal cubic $Z^\prime$ which can be mapped back to $Z$ with new coordinate $t^\prime$ by an element  of $\PGL(3,\mathbb{C})$. If $t_i,t_j,t_k,t_1,t_2,t_3\in Z$ are distinct points lying on a conic, then $t_i^\prime,t_j^\prime,t_k^\prime$ lie on a line by the properties of the standard triangular Cremona transformation so that
\begin{equation*}
1=t_it_jt_kt_1t_2t_3 =t_i^\prime t_j^\prime t_k^\prime.
\end{equation*} 
Similarly, the standard triangular Cremona transformation maps the line $L_{12}$ to $t_3^\prime$, so that for a point $t_i$ on $L_{12}$:
\begin{equation*}
1 = t_it_1t_2 = t_3^\prime  t_i^{\prime -1}  .
\end{equation*}
From these formulas we compute
\begin{align}\label{cremformula}
t^\prime = \left\{ \begin{array}{ll}
t(t_1t_2t_3)^{-2/3} & t=t_1,t_2,t_3\\
t(t_1t_2t_3)^{1/3} & t \ \text{general}
\end{array} \right.
\end{align} which determines $t^\prime$ up to multiplication by a third root of unity. These formulas can also be derived by computing the action of $s_7\in W$ on Equation \ref{pointform}. 

\section{Strata of smooth pointed quartic curves}

We have seen in Section \ref{dpdegree2} that the moduli space $\mathcal{DP}_2$ of del Pezzo surfaces of degree two and the moduli space $\mathcal{Q}$ of plane quartic curves are isomorphic. In this section we relate the moduli space $\mathcal{DPP}_2$ of del Pezzo pairs of degree two and its strata to the moduli space of smooth pointed plane quartics $\mathcal{Q}_1$. We first define this latter space $\mathcal{Q}_1$. 

\begin{dfn}Let $k$ be the field of real or complex numbers. A pointed plane quartic curve is a pair $(C,p)$ with $C$ a plane quartic curve and $p\in C(k)$. The space $\Gamma$ of smooth pointed quartics curve is defined by
\[ \Gamma(k) = \left\{ (C,p) \ ; \ p\in C(k) \right\} \subset P_{4,3}(k) - \Delta(k) \times \mathbb{P}^2. \]
The group $\PGL(3,k)$ acts on $\Gamma$ and the quotient 
\[ \mathcal{Q}_1= \PGL(3,k) \big\backslash \Gamma(k) \]
represents isomorphism classes of smooth pointed plane quartics.  
\end{dfn}

To a pointed quartic $(C,p)$ we can associate a del Pezzo pair $(Y,Z)$ in the following way. The del Pezzo surface $Y$ of degree two is defined by
\begin{equation}\label{antican}
Y = \left\{ w^2 = f(x,y,z) \right\} \subset \mathbb{P}(2,1,1,1)
\end{equation}
in weighted projective space. The morphism defined by the anti-canonical map is realized by the projection map $\psi:Y\rightarrow \mathbb{P}^2$ given by: 
\[ [w:x:y:z]\mapsto [x:y:z].\] 
Every anti-canonical curve on $Y$ is the pullback under $\psi$ of a line in $\mathbb{P}^2$. We define $Z=\psi^{-1} T_pC$ to be the pullback of the tangent line to $C$ at $p$. It is a singular anti-canonical curve on $Y$ of arithmetic genus $1$. Its Kodaira type is determined by the type of the intersection divisor $D=(C\cdot T_pC)$ defined below.  

\begin{dfn}
Let $D=\sum_{i=1}^k d_i(p_i)$ be a divisor on a curve $C$ with the $p_i$ distinct and ordered in such a way that $d_1\geq \ldots \geq d_r$. The type of $D$ is the $r$-tuple $\underline{d}=(d_1,\ldots,d_r)$.
\end{dfn} 

There are four possibilities for the type of $D$ corresponding to the types for $Z$ in Table \ref{Pezzostrata}. Similarly we obtain four strata in the space $\Gamma$. 
\begin{table}[h]
\begin{displaymath}
\begin{array}{l l l l}
\toprule
\text{Stratum} & D & Z   & \text{Codim.} \\
\midrule
\Gamma^{\circ} & (2,1,1) & I_1 & 0 \\
\Gamma^{\text{bit}} & (2,2) & I_2 & 1\\
\Gamma^{\text{flex}} & (3,1) & II & 1\\
\Gamma^{\text{hflex}} & (4) & III & 2\\
\bottomrule
\end{array}
\end{displaymath}
\caption{Strata in the space of pointed quartics}
\end{table}
The strata $\Gamma^{\text{bit}}$ and $\Gamma^{\text{flex}}$ where the point $p$ is respectively a bitangent and an inflection point have codimension one and the stratum $\Gamma^{\text{hflex}}$ where $p$ is a hyperflex has codimension two in the space $\Gamma$. 

\section{Moduli of real del Pezzo pairs of degree two}

Let $(C,p)$ be a smooth real pointed plane quartic curve. By the results of the previous section and Section \ref{realdpsurfaces} we can associate to $(C,p)$ a real del Pezzo pair $(Y,Z)$ with real form $\chi_Y^-$ such that $Y^-(\mathbb{R})$ in nonorientable. The real form restricts to $Z$ which is a real curve of arithmetic genus $1$ on $Y$ and $Z^{\ns}(\mathbb{R})\neq \emptyset$. If $(C,p)$ is in the open stratum $\Gamma^\circ$ then the tangent line $T_pC$ intersects $C$ in two other distinct points which can both be real or a pair of complex conjugate points. In both cases the curve $Z$ is of type $I_1$ (it has a single node). Since $Z^{\ns}(\mathbb{R})\neq \emptyset$ there are two possibilities for the real form induced by $\chi_Y^-$ on $Z^{\ns}(\mathbb{C}) \cong \mathbb{C}^\ast$. Either it maps: $t\mapsto \bar{t}$ and $Z^{\ns}(\mathbb{R})\cong \mathbb{R}^\ast$ or $t\mapsto \bar{t}^{-1}$ and $Z^{\ns}(\mathbb{R})\cong S^1$. An example of both is given in Figure \ref{planemodels}.

\begin{figure}[h]
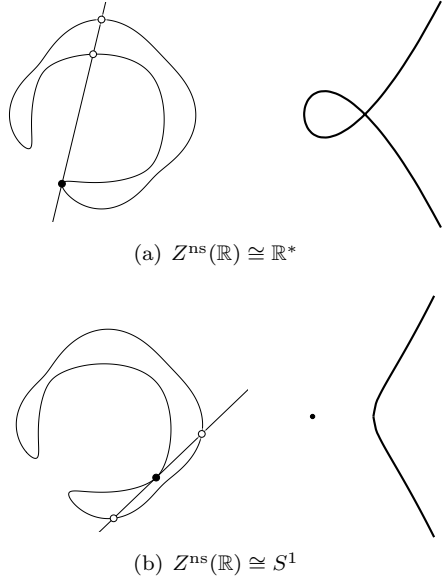

\centering
\subfigure[$Z^{\ns}(\mathbb{R})\cong \mathbb{R}^\ast$]
{
\includegraphics[trim = 250 0 250 0,clip,width=3.3cm]{Q1MUL0}
\qquad
\begin{tikzpicture}[scale=0.8]
\draw[scale=1,domain=-1.5:1.5,smooth,variable=\t, thick]
plot ({(\t)^2-1},{(\t)^3-(\t)});
\end{tikzpicture}
}

\subfigure[$Z^{\ns}(\mathbb{R})\cong S^1$]
{
\includegraphics[trim = 250 0 250 0,clip,width=3.3cm]{Q1ADD}
\qquad
\begin{tikzpicture}[scale=0.8]
\draw[scale=1,domain=1:2,smooth,variable=\t, thick]
plot ({\t},{sqrt((\t)^3-(\t)^2)});
\draw[scale=1,domain=1:2,smooth,variable=\t , thick]
plot ({\t},{-sqrt((\t)^3-(\t)^2)});
\draw[fill=black] (0,0) circle (1pt);
\end{tikzpicture}
}
\caption{The two possibilities for $Y^{\ns}(\mathbb{R})$ for a quartic curve with one oval.}
\label{planemodels}
\end{figure}

\begin{thm}\label{bigrealthm}The map $(C,p)\mapsto \left( \chi:\Pic^0 Y\rightarrow \Pic^0 Z\right)$ extends to an isomorphism:
\[ (\mathcal{Q}_1^{\circ})^\mathbb{R} \cong (W \backslash \mathbb{T}^\circ)(\mathbb{R}) \] 
where $\mathbb{T}^\circ$ denotes the complement in $\mathbb{T}=\Hom(Q,\mathbb{C}^\ast)$ of the mirrors of the action of the Weyl group $W$ of type $E_7$.    
\end{thm}

\begin{proof}
A lot of work has already been done in the proof of Theorem \ref{BigTheoremC}. First we need to show that the element $\chi:Q\rightarrow \mathbb{C}^\ast$ we associate to $(C,p)$ is a real point of $W\backslash \mathbb{T}^\circ$. By definition this means that $w\cdot \chi = \bar{\chi}$ for some element $w\in W$. The real structure $\chi_Y^-$ of Equation \ref{realstructuredelpezzo} acts on $\Pic^0(Y)\cong Q$ as an involution $u\in W$. Since $Y^-(\mathbb{R})$ is nonorientable we see from Table \ref{realdelpezzotable} that this involution is of type $1,A_1,A_1^2,A_1^3$ or $D_4$. The action of the restriction of $\chi_Y^-$ to $\Pic^0(Y)\cong \mathbb{C}^\ast$ is one of $t\mapsto \bar{t}^{\pm 1}$ so the element $\chi$ satisfies $u\cdot \chi = \bar{\chi}^{\pm 1}$. Since the Weyl group  $W$ of type $E_7$ contains $-1$ we can rewrite this as $\pm u \cdot \chi = \bar{\chi}$ so that $\chi$ is indeed a real element of $W\backslash \mathbb{T}^\circ$. 

Conversely, let $\chi$ be a real point of $(W\backslash \mathbb{T}^\circ)(\mathbb{R})$. By Proposition \ref{toricTits} we can assume that $u\cdot \chi = \bar{\chi}^{\pm 1}$ with $u\in W$ an involution of type $1,A_1,A_1^2,A_1^3$ or $D_4$. As in the proof of Theorem \ref{BigTheoremC} we fix a real nodal cubic $Z$ in $\mathbb{P}^2$ and an isomorphism $Z^{\ns}(\mathbb{C})\cong \mathbb{C}^\ast$ by choosing a real inflection point. The real form of $Z$ is then equivalent to one of $t\mapsto \bar{t}^{\pm 1}$. As in Equation \ref{Piequality} the element $\chi$ determines seven points in $\mathbb{C}^\ast$ by the formula $t_i=\chi(e_i-\frac{e_0}{3})$ which we interpret as points on $Z^{\ns}(\mathbb{C})$. Since $\chi$ is real these points satisfy
\[ u \cdot (t_1,\dots,t_7) = (\bar{t}_1^{\pm 1},\ldots, \bar{t}_7^{\pm 1}).\] 
where the involution $u\in W$ acts by the Cremona action of the Weyl group as a birational involution of $\mathbb{P}^2$. This involution lifts to an anti-holomorphic involution of the del Pezzo surface $Y$ obtained by blowing up the seven points. These two construction are inverse to each other.

\end{proof}

\section{Reflection groups and real tori}\label{reflection}

In this section we study the connected components of the space $(W\backslash \mathbb{T}^\circ)(\mathbb{R})$ where $\mathbb{T}$ is the complex torus $\mathbb{T}=\Hom(Q,\mathbb{C}^\ast)$ for $Q$ a root lattice of type $ADE$. For type $E_7$ this space has $20$ connected components which we describe explicitly as quotients of real subtori of $\mathbb{T}$.

\subsection{Reflection groups and root systems}\label{WeylRoots}

We start by recalling some facts about reflection groups and root systems. Our main reference is \cite{Bourbaki} Chapter VI. Let $V$ be a real, finite dimensional vector space of dimension $n$ with an inner product $(\cdot,\cdot)$. For every nonzero $\alpha \in V$ we define the reflection $s_\alpha \in O(V)$ by
\begin{displaymath}
s_\alpha(x)=x-2\frac{(\alpha,x)}{(\alpha,\alpha)}\alpha \ 
\end{displaymath}
for all $x\in V$. The mirror $H_\alpha$ is defined as the fixed point locus of the reflection $s_\alpha$. A \emph{root system} $R \subset V$ is a finite set of nonzero vectors called \emph{roots} that satisfy the following properties
\begin{description}
\item[R1] The $\mathbb{R}$-span of $R$ is $V$. 
\item[R2] If $\alpha \in R$ then $R \cap \mathbb{R}\alpha = \{\alpha,-\alpha\}$.
\item[R3] $s_\alpha R = R$ for all $\alpha \in R$. 
\item[R4] $2\frac{(\beta,\alpha)}{(\alpha,\alpha)} \in \mathbb{Z}$ for all $\alpha,\beta \in R$.
\end{description}

A system of simple roots $\Delta = \{ \alpha_1,\ldots,\alpha_r\} \subseteq R$ is a basis for $V$ such that every root is an integral linear combination $\sum_{i=1}^r c_i\alpha_i$ of simple roots of the same sign. It is known that such a simple system always exists. From now on we assume we have fixed a system of simple roots $\Delta \subset R$. For every root $\alpha \in R$ we define the coroot $\alpha^\vee$ by
\begin{displaymath}
\alpha^\vee = \frac{2\alpha}{(\alpha,\alpha)}.
\end{displaymath}
The set of coroots $R^\vee$ is again a root system (the coroot system) with corresponding coroot lattice $Q^\vee=\mathbb{Z}R^\vee$. A root system $R$ is called \emph{irreducible} if it is non-empty and cannot be decomposed as an orthogonal direct sum $R=R_1\oplus R_2$ of two non-empty root systems $R_1$ and $R_2$. Let $R$ be an irreducible root system. We define the \emph{highest root} $\tilde{\alpha}$ of $R$ with respect to $\Delta$ as the unique root such that $\sum_{i=1}^r c_i$ is maximal. We also define $\alpha_0=-\tilde{\alpha}$. The Weyl group $W$ is the group generated by the reflections $s_\alpha$ with $\alpha \in R$ or equivalently by the simple reflection $s_\alpha$ with $\alpha \in \Delta$. It is a finite group and acts simply transitively on the connected components of $V\setminus \cup H_{\alpha}$ which are called \emph{chambers}. The fundamental chamber $\mathcal{C}$ with respect to a given system of simple roots $\Delta$ is defined by
\begin{displaymath}
\mathcal{C} = \left\{ x \in V \ \vert \ (\alpha_i,x)>0 \ \text{for} \  1 \leq i \leq r \right\}.
\end{displaymath}
Its closure $\bar{\mathcal{C}}$ is a fundamental domain for the action of $W$ on $V$. 

The \emph{affine Weyl group} $W_a$ is the group generated by the affine reflections $s_{\alpha,k}$ with $\alpha \in R$ and $k\in \mathbb{Z}$ defined by 
\begin{displaymath}
s_{\alpha,k}(x)=x - (\alpha,x)\alpha^\vee + k\alpha^\vee .
\end{displaymath}
The mirror of $s_{\alpha,k}$ is the affine hyperplane $H_{\alpha,k}=\{ x \in V ; (\alpha, x)=k \}$. The affine Weyl group $W_a$ is the semidirect product of $W$ by the coroot lattice: $W_a = Q^\vee \rtimes W$. This allows us to write $s_{\alpha,k}=t(k\alpha^\vee)s_\alpha$ where $t(k\alpha^\vee)$ denotes translation over $k \alpha^\vee$ in $V$. The group $W_a$ acts simply transitively on the connected components of the space $V^\circ = V \setminus \cup H_{\alpha,k}$ which are called \emph{alcoves}. The fundamental alcove $\mathcal{A}$ is the simplex given by
\begin{displaymath}
\mathcal{A} = \left\{ x \in V \ \vert \ (\tilde{\alpha},x) < 1, \ (\alpha_i,x)>0 \ \text{for} \  1 \leq i \leq r \right\}
\end{displaymath}
and its closure $\bar{\mathcal{A}}$ is a fundamental domain for the action of $W_a$ on $V$. The $r+1$ closed facets $\bar{\mathcal{A}}_i$ of $\bar{\mathcal{A}}$ are given by
\begin{displaymath}
\bar{\mathcal{A}}_i = 
\begin{cases}
H_{\alpha_i}\cap \bar{\mathcal{A}} & \text{if} \ 1\leq i \leq r\\
H_{\tilde{\alpha},1} \cap \bar{\mathcal{A}} & \text{if} \ i=0
\end{cases}
\end{displaymath}

A reducible root system $R$ can be decomposed into a direct sum of irreducible root systems $\{R_i\}_{i\in I}$ for some finite index set $I$. The Weyl group $W(R)$ of $R$ is the direct product of the Weyl groups $\{W(R_i)\}_{i \in I}$.  This decomposition is unique up to permutation of the factors. A fundamental domain for the action of $W(R)$ on $V$ is now the direct product of the fundamental chambers of the factors. Similarly for the affine Weyl group $W_a(R)=Q^\vee \rtimes W(R)$ a fundamental domain on $V$ is the product of the fundamental alcoves of the factors.

We want to determine the stabilizer $\Stab_{W_a}(x)$ of an $x\in V$ in the affine Weyl group. Since all points in the orbit $W_a \cdot x$ have conjugate stabilizers, we can assume that $x\in \bar{\mathcal{A}}$. The stabilizer $\Stab_{W_a}(x)$ is the group generated by the reflections in the mirrors $H_{\alpha,k}$ that contain $x$. It is a Weyl group with root system $R(x)$ and system of simple roots $\Delta(x)$ given by
\begin{equation}\label{BorelSiebenthal}
R(x) = \left\{ \alpha \in R\ ; \ (\alpha,x) \in \mathbb{Z} \right\} \ ,\ \Delta(x) = \left\{ \alpha_i \ ; 0\leq i \leq r , \ x \in \bar{\mathcal{A}}_i \right\}.
\end{equation}
These root systems can be reducible, even if the root system $R$ is irreducible.

\subsection{The extended affine Weyl group}

The coweight lattice $P^\vee$ is defined by 
\begin{displaymath}
P^\vee = \{ \mathbb{Z} \in V ; (\mathbb{Z},\alpha)\in \mathbb{Z} \quad  \forall \alpha \in R\}
\end{displaymath}
and contains $Q^\vee$ as a subgroup of finite index. It has a basis $\{ \varpi_1^\vee,\ldots,\varpi_r^\vee \}$ dual to the basis of simple roots of $R$, so that $( \alpha_i, \varpi_j^\vee) = \delta_{ij}$. The \emph{extended affine Weyl group} $W_a^\prime$ is defined as the semidirect prodct $P^\vee \rtimes W$ with $P^\vee$ acting on $V$ by translations. We will prove that $W_a^\prime$ is the extension of $W_a$ by a finite subgroup of the automorphisms of the fundamental alcove. 

Let $n_i=(\tilde{\alpha},\varpi_i^\vee)$ be the coefficient of $\alpha_i$ in the highest root $\tilde{\alpha}$. For notation it is convenient to define $\varpi^\vee_0=0\in P^\vee$ and $n_0=1$. The fundamental alcove $\mathcal{A}$ is the open $n$-simplex with vertices $\{\varpi_i^\vee / n_i\}_{i=0}^r$. Let $J$ be the set of indices $0\leq i \leq r$ such that $n_i=1$. The vertices $\varpi_i^\vee$ with $i\in J$ or equivalently: $R(\varpi_i^\vee/n_i)\cong R$ are called \emph{special}. Put $R_0=R$ and let $w_0$ be the longest element of $W$ with respect to the basis of simple roots in equation $\{\alpha_i\}_{i=1}^r$. We also define for every $i\in J\setminus \{0\}$ the root system $R_i$ generated by the simple roots
\begin{equation}\label{roots}
\{ \alpha_1,\ldots,\hat{\alpha_i},\ldots,\alpha_r\}.
\end{equation}
Let $w_i$ be the longest element of the Weyl group $W(R_i)$ with respect to the basis of simple roots (\ref{roots}).  For every $i\in J$ we now define the following element of the extended affine Weyl group
\begin{displaymath}
\gamma_i = t(\varpi_i^\vee)w_iw_0.
\end{displaymath}

Observe that $\gamma_i(0)=\varpi_i^\vee$ and $\gamma_0=w_0^2=1$. Proposition 6 from \cite{Bourbaki} \S $2$ VI states that we have equality
\begin{equation}\label{equality}
\{ w\in W^{\prime}_a ; w(\mathcal{A})=\mathcal{A} \} = \{ \gamma_i \}_{i\in J} 
\end{equation} 
and we can identify the group (\ref{equality}) with the finite Abelian group $P^\vee/Q^\vee$ by assigning to $\gamma_i$ the class of $\varpi_i$ mod $Q^\vee$ where $i\in J$. We see that the group $P^\vee/Q^\vee$ acts simply transitively on the special points. Since the affine Weyl group $W_a$ acts simply transitively on the alcoves it follows from the above that we have an isomorphism
\begin{equation}\label{extiso}
W_a \rtimes P^\vee/Q^\vee \overset{\sim}{\longrightarrow} P^\vee \rtimes W = W_a^\prime
\end{equation}
by assigning $(t(\mathbb{Z})w,\gamma_i)\mapsto t(\mathbb{Z}+w\varpi_i^\vee)ww_iw_0$. The extended affine Weyl group acts transitively on the connected components of $V^\circ$, but the action need not be free. The action of $P^\vee / Q^\vee$ on the fundamental alcove $\mathcal{A}$ can have fixed points. Also $W_a^\prime$ is in general not a Coxeter group. 

\begin{lem}\label{stab} Let $x\in \bar{\mathcal{A}}$, then
\begin{displaymath}
\Stab_{W_a^\prime}(x) = \Stab_{W_a}(x) \rtimes \Stab_{P^\vee/Q^\vee}(x)
\end{displaymath} 
\end{lem}

\begin{proof}Let $t(\mathbb{Z})w \in W_a$ and $\gamma \in P^\vee/Q^\vee$ be such that $t(\mathbb{Z})w\gamma(x)=x$. Define $y:=\gamma(x) \in \bar{\mathcal{A}}$. Now $t(\mathbb{Z})w(y)=x$ with $x,y\in \bar{\mathcal{A}}$, and because $\bar{\mathcal{A}}$ is a strict fundamental domain for the action of $W_a$ we can conclude $x=y$, so $\gamma(x)=x$. This also implies that $t(\mathbb{Z})w(x)=x$.  
\end{proof}

\subsection{The centralizer of an involution in a reflection group}

In this section we recall some known results on centralizers of involutions in reflection groups. Let $(W,S)$ be a finite Coxeter group and let $u\in W$ be an involution. We want to determine the centralizer $C_W(u)$ of $u$ in $W$. By the classification of involutions in Coxeter groups there is a subset $I\subseteq S$ such that $u$ is conjugate in $W$ to the involution $w_I$: the unique longest element $-1$ in the parabolic subgroup $W_I$. Felder en Veselov in \cite{Felder} observe the following.
\begin{prop}
If $W_I$ is a parabolic subgroup of $W$ that satisfies the $(-1)$ condition then $C_W(w_I)=N_W(W_I)$. 
\end{prop} 
\begin{proof}
The element $u$ is the unique longest element of $W_I$, so that $wuw^{-1}=u$ for all $w\in N_W(W_I)$ and consequently:  $N_W(W_I)<C_W(u)$. For the other inclusion let $w\in C_W(u)$. We need to prove that $ws_{\alpha_i}w^{-1}\in W_I$, or equivalently: $w\cdot \alpha_i \in R \cap V_I$ for all $i\in I$. This holds since the element $w$ preserves the eigenspace decomposition of the involution $u$.  
\end{proof}

Using this result we can use the classification of normalizers of parabolic subgroups of reflection groups by Howlett \cite{Howlett}. We want to mention here that a lot of the results of this section also appear in \cite{Looijenga}. Let $u=w_I$ and decompose $V$ into $\pm 1$-eigenspaces for $u$: $V=V_u^+ + V_u^-$ where $V_u^-=V_I$. This defines two orthogonal root systems and corresponding Coxeter groups
\[R_u^\pm=R\cap V_u^\pm  \quad , \quad W_u^\pm=W(R_u^\pm) . \]
Observe that $W_u^-=W_I$ and that the eigenspace $V_u^-$ is spanned by the roots of $R_u^-$. The group $W_u=W_u^-\times W_u^+$ is generated by all reflections that commute with $u$ and is contained in the centralizer $C_W(u)$ of $u$. This centralizer also contains a non-reflection part $G_u$ which now describe. Consider the two Weyl elements
\begin{align*}
\rho_\pm = \frac{1}{2} \sum_{\alpha \in R_u^\pm(+)} \alpha
\end{align*}  
where the sum runs over all positive roots of $R_u^\pm$, which we denote by $R_u^\pm(+)$. The set
\[ R_u^c = \{ \alpha \in R \ ; \ (\alpha,\rho_+)=(\alpha,\rho_-)=0 \} \]
is a root system which can be written as an orthogonal disjoint union of subroot systems $R_u^c=R_1\cup R_2$. These factors are isomorphic root systems and are exchanged by the involution $u$. The Weyl group $W(R_u^c)$ is a product 
\[W(R_u^c)=W(R_1)\times W(R_2)\]
whose factors are exchanged by conjugacy with $u$. The group $G_u$ consists of all elements of $W(R_u^c)$ that commute with $u$ and is the diagonal of this product
\[ G_u = \{ (w,uwu) \ ; \ w \in W(R_1) \} .\]
This group is generated by pairs of commuting reflections $s_\alpha s_{u\cdot \alpha}$ with $\alpha \in R_1$ and is isomorphic to $W(R_1)$. Now we can formulate the main theorem of this section. 

\begin{thm}The centralizer of an involution $u\in W$ splits as a semidirect product
\begin{align*}
C_W(u)  &\cong W_u \rtimes G_u \\
& \cong W_u^- \rtimes G_u^+ 
\end{align*}
where $G_u^+$ is the reflection group defined by $G_u^+=\{w\in W \ ; \ wI=I\}$ which contains $W_u^+$ as a normal subgroup.  
\end{thm}

\subsection{Root tori and their invariants}
If the type of a root system $R$ occurs in Equation \ref{delpezzotype} we say it is of del Pezzo type. Such root systems are products of root systems of type $ADE$ so that all roots have the same length. This implies that we can identify the root (resp weight) and the coroot (resp coweight) lattices $R$ and $R^\vee$ (resp $P$ and $P^\vee$). To simplify notation we use this identification from now on. To a root system $R$ of del Pezzo type we associate the complex torus $\mathbb{T}=\mathbb{C}^\ast \otimes P = \Hom(Q,\mathbb{C}^\ast)$.  It has a natural action of the Weyl group $W$. 
\begin{thm}
The quotient $W\backslash \mathbb{T}$ is an affine toric variety and the algebra of $W$-invariants of $\mathbb{Z}[Q]$ is the semi-group algebra given by
\begin{equation}\label{invariantiso}
\mathbb{Z}[Q]^W \cong \mathbb{Z}[P_+ \cap Q]
\end{equation}
where $P_+=\sum_{i=1}^n \mathbb{Z}_{\geq 0}\varpi_i$ is the lattice cone spanned by the fundamental weights. \end{thm}  

\begin{proof}
For the proof of the isomorphism of Equation \ref{invariantiso} we refer to \cite{Lorenz} Section $6.3.5$. The coordinate ring of $W\backslash \mathbb{T}$ is $\mathbb{C}[\mathbb{T}]^W \cong \mathbb{C} \otimes_{\mathbb{Z}} \mathbb{Z}[P_+ \cap Q]$. It is the complexification of a semi-group algebra and its spectrum is by definition an affine toric variety. 
\end{proof}

This theorem is a generalisation of classical exponential invariant theory for root systems as described in \cite{Bourbaki} VI \S 3. The main theorem of that section states that the algebra of $W$-invariants of $\mathbb{Z}[P]$ is a polynomial algebra 
\[ \mathbb{Z}[P]^W \cong \mathbb{Z}[P_+].\]

This is a toric analogue of a well known theorem of Chevalley. The algebra $\mathbb{C}[P]$ is the coordinate ring of the algebraic torus $T=\mathbb{C}^
\ast \otimes Q$. Since the $W$-invariants form a polynomial algebra we can rephrase the theorem as $W\backslash T \cong \mathbb{C}^n$. The torus $T$ is a finite cover of $\mathbb{T}$ where the group of deck transformations is isomorphic to $P/Q$ and there is an isomorphism of orbifolds
\[W\backslash \mathbb{T} \cong (P/Q)\backslash \mathbb{C}^n.\]

The action of $W$ on $\mathbb{T}$ on the complement of the mirrors $\mathbb{T}^\circ$ is not free in general: the group $P/Q$ can have fixed points in $\mathbb{T}^\circ$. The stabilizers are described by the following lemma.

\begin{lem}
For $t\in \mathbb{T}$ the stabilizer $\Stab_W(t)$ is the extension of a reflection group $\Stab_W^r(t)$ by a finite subgroup of the automorphisms of a fundamental alcove of $W_a$.
\end{lem}
\begin{proof}
Consider the exponential sequence

\begin{center}
\begin{tikzcd} 
0 \arrow{r} & \mathbb{Z} \arrow{r} & \mathbb{C} \arrow{r}{\exp} & \mathbb{C}^\ast \arrow{r} & 1
\end{tikzcd}
\end{center}
where $\exp: z \mapsto e^{2 \pi i z}$. By tensoring from the right with $P$ we obtain another exact sequence

\begin{center}
\begin{tikzcd}\label{exseq}
0 \arrow{r} & P  \arrow{r} & V_{\mathbb{C}} \arrow{r}{\exp} & \mathbb{T} \arrow{r} & 1
\end{tikzcd}
\end{center}

where $V_{\mathbb{C}} = \mathbb{C} \otimes V$ is the complexification of $V$. From the sequence we read off that $W\backslash \mathbb{T} \cong W_a^\prime \backslash V_\mathbb{C}$ where the extended affine group $W_a^\prime=P\rtimes W$ acts on $V_\mathbb{C}$ by the formula 
\[ t(\lambda)w \cdot (x+iy) = (w\cdot x+\mathbb{Z}) +i(w\cdot y)\]
for $\mathbb{Z} \in Q$ and $w\in W$. Write $z=\log t$, by Lemma \ref{stab} the group $\Stab_W(t)$  is isomorphic to 
\[ \Stab_{W_a^\prime}(z) \cong \Stab_{W_a}(z)\rtimes \Stab_{P/Q}(z).\] 
The group $\Stab_{W_a}(z)$ is a Weyl group generated by the reflection in the mirrors that contain $z$. 
\end{proof}

\subsection{Real root tori and their connected components}

Complex conjugacy on $\mathbb{C}^\ast$ defines a real form on the complex torus $\mathbb{T}=\mathbb{C}^\ast \otimes P$. This in turn defines a real form on the quotient $W \backslash \mathbb{T}$. Let $q:\mathbb{T} \rightarrow W\backslash \mathbb{T}$ be the quotient map. The real points of $W\backslash \mathbb{T}$ are the points $q(t)$ such that $t$ and $\bar{t}$ are in the same $W$-orbit so that

\[ q(t)\in (W\backslash \mathbb{T})(\mathbb{R}) \Longleftrightarrow w\cdot t=\bar{t} \quad \text{for some} \ w\in W.\] 
We will prove in Proposition \ref{toricTits} that we can assume that $w$ is an involution in $W$. Every involution in $W$ defines a real form on $\mathbb{T}$ by composing with complex conjugation.   The real points of such a real form are given by
\[ \mathbb{T}_u(\mathbb{R}) = \left\{ t\in \mathbb{T} \ ; \ u\cdot t = \bar{t} \right\}. \]
The following proposition is a slight modification of a result due to Tits (\cite{Looijenga}, Proposition 2.2) to the present situation. The proof is similar to the one given there.   

\begin{prop}[Tits]\label{toricTits}
\begin{displaymath}
q^{-1}(W \backslash \mathbb{T})(\mathbb{R}) = \bigcup_{u\in W ; u^2=1} \mathbb{T}_u(\mathbb{R})
\end{displaymath}
\end{prop}
\begin{proof}
Let $t\in \mathbb{T}$ be such that $w\cdot t = \bar{t}$ for some $w\in W$. We will prove that there is a $w^\prime$ in the reflection part $\Stab_W^r(t)$ of the stabilizer $\Stab_W(t)$ such that $u=ww^\prime$ is an involution in $\Stab_W(t)$. The reflection part of the stabilizer is a finite reflection group which acts on the tangent space $T_t\mathbb{T}$ through its complexified reflection representation. Since $\Stab_W(t)$ is also the stabilizer of $\bar{t}$ we see that
\[ w\Stab_W(t)w^{-1}=\Stab_W(w\cdot t)=\Stab_W(\bar{t})=\Stab_W(t). \]
This implies that $w$ permutes the chambers of $\Stab_W^r(t)$ so that we can find a $w^\prime \in \Stab_W^r(t)$ such that $u=ww^\prime$ leaves a chamber invariant. Since $\Stab_W^r(t)$ acts simply transitively on its chambers it follows that $u^2=1$.

\end{proof}


The group $W$ permutes the real tori $\mathbb{T}_u$ according to
\begin{displaymath}
w\cdot \mathbb{T}_u=\mathbb{T}_{wuw^{-1}}
\end{displaymath}
so $W$-equivalent real tori correspond to conjugate involutions. Furthermore the stabilizer of a real torus $\mathbb{T}_u(\mathbb{R})$ in $W$ is precisely the centralizer $C_W(u)$ of $u$ in $W$. We want to study the real tori $\mathbb{T}_u$ and especially their connected components in more detail. The involution $u$ acts naturally on the weight lattice $P$ and there exists a so called \emph{normal} basis for $P$ such that
\begin{equation}\label{latdecom}
\begin{aligned}
u &= I_{n_1} \oplus \left(-I_{n_2}\right) \oplus 
\left( \begin{array}{ccc} 0 & 1 \\ 1 & 0\end{array} \right)^{n_3} \\
P &= P_{1,u} \oplus P_{2,u} \oplus P_{3,u}.
\end{aligned}
\end{equation}
This is described in detail in \cite{Casselman}. The decomposition of Equation \ref{latdecom} is not unique but the triple $(n_1,n_2,n_3)$ which we call the \emph{type} of the involution $u\in W$ is an invariant of the involution. A choice of normal basis determines an isomorphism
\begin{equation}\label{realiso}
\mathbb{T}_u(\mathbb{R}) \cong (\mathbb{R}^\ast)^{n_1} \times (S^1)^{n_2} \times (\mathbb{C}^\ast)^{n_3}
\end{equation}
with $n_1,n_2,n_3\in \mathbb{N}$ and $n_1+n_2+2n_3=n$. This product consists of factors of \emph{split} ($\mathbb{R}^\ast$), \emph{compact} ($S^1$) and \emph{complex} ($\mathbb{C}^\ast$) type. To determine the numbers $n_i$ we have the following lemma from \cite{Casselman}.

\begin{lem}There are isomorphisms of abelian groups
\begin{align*}
\frac{\ker (u-1)}{\im(u+1)} \cong (\mathbb{Z}/2\mathbb{Z})^{n_1} \quad , \quad
 \frac{\ker (u+1)}{\im(u-1)} \cong (\mathbb{Z}/2\mathbb{Z})^{n_2}
\end{align*}
where the first of these can be identified with the component group $\pi_0(\mathbb{T}_u(\mathbb{R}))$.
\end{lem}
\begin{proof}We construct the first of these isomorphisms. After choosing a normal basis for $P$ we can use the lattice decomposition (\ref{latdecom}) to see that the lattice $\ker(u-1)$ is isomorphic to $P_{1,u}\oplus P_{3,u}$. Similarly the lattice $\im(u+1)$ is isomorphic to $2P_{1,u}\oplus P_{3,u}$ and the quotient of these lattices is $P_{1,u}/2P_{1,u}$. \end{proof}

To determine the number of connected components of $C_W(u)\backslash \mathbb{T}_u(\mathbb{R})$ we need to compute the number of orbits under the action of $C_W(u)$ on the connected components of $\mathbb{T}_u(\mathbb{R})$. The following lemma shows that in fact only $W_u^+$ acts non-trivially on the components. 
\begin{lem}
The group $W_u^- \rtimes G_u$ is contained in the kernel of the action of $C_W(u)$ on $\ker(u-1)/\im(u+1)$. 
\end{lem}

\begin{proof}
Suppose $x \in \ker (u-1)$. In particular $x \in V_u^+$ so that $w\cdot x=x$ for all $w\in W_u^-$. The group $G_u$ is generated by products of commuting reflections $s_\alpha s_{u\cdot \alpha}$ where $\alpha \in R_1$. Such an element acts trivially on the class of $x$ in $\ker(u-1)/\im(u+1)$ since
\begin{align*}
s_\alpha s_{u\cdot \alpha} \cdot x-x &= (u+1)\left((\alpha,x)\alpha \right) \in \im(u+1).
\end{align*}
\end{proof}

\subsection{Connected components of $W\backslash \mathbb{T}_1$}

For the trivial involution $u=1$ and the real torus $\mathbb{T}_1$ we have
\begin{align*}
\mathbb{T}_1(\mathbb{R}) \cong (\mathbb{R}^\ast)^n \quad , \quad \frac{\ker u-1}{\im u+1} \cong \frac{P}{2P} \quad , \quad C_W(u)=W.
\end{align*} 
The decomposition of this real torus into connected components is described by
\[ \mathbb{T}_1(\mathbb{R}) = \bigsqcup_{[\varpi] \in P/2P}\mathbb{T}_1^\varpi \quad  \text{where} \quad \mathbb{T}_1^\varpi = \exp \left( \frac{1}{2}\varpi + iV \right). \] 
We can use the basis of fundamental weights $\{\varpi_1,\ldots,\varpi_n\}$ of $P$ to identify $\mathbb{T}$ with $(\mathbb{C}^\ast)^n$ through the isomorphism
\begin{align*}
\mathbb{C}^\ast \otimes P &\rightarrow (\mathbb{C}^\ast)^n \\
\sum_{i=1}^n t_i \otimes \varpi_i &\mapsto (t_1,\ldots,t_n).
\end{align*}
In this way we can also identify the component group $P/2P$ of $\mathbb{T}_1$ with the subgroup $\{-1,1\}^n \subset (\mathbb{R}^\ast)^n$. 
\begin{rmk}An element of $P/2P$ can be represented by a colouring of the Coxeter diagram of $W$ where the $i$th node is coloured white if the corresponding coefficient of $\varpi_i$ is $1$ and coloured black if it is $-1$. To determine the action of $W$ on two-coloured Coxeter diagrams first observe that a simple reflection for $W$ acts on the fundamental weights as    
\begin{equation}\label{weightcalc}
s_i \cdot \varpi_j=
\begin{cases}
\varpi_i &  i\neq j \\
-\varpi_i+\sum_{k\in I_j}\varpi_j & i=j
\end{cases}
\end{equation} 
where the sum runs over the set $I_j$ of neighbouring vertices of the $j$th vertex of the Coxeter diagram. Now the generator $s_i$ only acts nontrivially if the $i$th node $v_i$ is black. In this case the action of $s_i$ changes the colour of all neighbouring vertices of $v_i$ but leaves $v_i$ unchanged. 
\end{rmk}

It is often convenient to use the group $\frac{1}{2}P/P$ for representing the connected components of $\mathbb{T}_u(\mathbb{R})$ instead of the group $P/2P$. The reason for this is that there are bijections of orbit spaces
\begin{align}\label{orbitspaceformula}
W\Big\backslash \left(\frac{1}{2}P/P\right) \cong \left(P\rtimes W\right) \Big\backslash \frac{1}{2}P \cong \left(P/Q\right)\Big\backslash \left(\frac{1}{2}P \cap \bar{\mathcal{A}}\right).
\end{align} 
We can count the points in the intersection $\frac{1}{2}P\cap \bar{\mathcal{A}}$ and the group $P/Q$ is typically small. Its action is easily determined for Weyl groups of type $ADE$. We do this for root systems of type $A_n$ in Example \ref{AN}.

\begin{ex}[$A_n$]\label{AN}We will use the above method to describe the orbit space $W\backslash \left( P/2P\right)$ for type $A_n$. This will be used frequently in the next section. Representatives for the orbits are given by
\[
\begin{cases}
\{0,\varpi_1,\ldots,\varpi_{n/2}\} & $n$ \ \text{even}\\
\{0,\varpi_1,\ldots,\varpi_{(n+1)/2}\} & $n$ \ \text{odd}.
\end{cases}
\] 
\end{ex}
\begin{proof}
For a root system of type $A_n$ all the roots have coefficient $1$ in the highest root so the fundamental alcove is the convex hull of the fundamental weights. From this we determine
\[ \bar{\mathcal{A}}\cap \frac{1}{2}P = \left\{ 0,\frac{\varpi_i}{2}, \frac{\varpi_i+\varpi_j}{2} \right\}_{1\leq i\neq j\leq n}.\] 
The group $P/Q$ is cyclic of order $n+1$ and is generated by $\gamma_1$ which acts as the permutation $(01\ldots n)$ on the indices of the fundamental weights $\{\varpi_i\}$. A small calculation shows that
\[\gamma_1 \left( \frac{\varpi_i}{2} \right) = \frac{\varpi_{i+1}+\varpi_1}{2}  \]
where we use the notation $\varpi_0=0$ and the indices are considered $\text{mod} \ n+1$. 
A typical $\gamma_1$-orbit (for which $n+1\neq 2i$) is of the form
\begin{align*}
\frac{\varpi_i}{2} &\mapsto \frac{\varpi_{i+1}+\varpi_1}{2} \mapsto \ldots \mapsto\frac{\varpi_n+\varpi_{n-i}}{2} \mapsto \frac{\varpi_{n-i+1}}{2} \\
&\mapsto \frac{\varpi_{n-i+2}+\varpi_{1}}{2} \mapsto \ldots \mapsto \frac{\varpi_{n}+\varpi_{i-1}}{2} \mapsto \frac{\varpi_i}{2}.
\end{align*} 
A representative is given by $\varpi_i/2$. If $n$ is even then all orbits are of this form and there are $n/2$ orbits. If $n$ is odd then there is one additional orbit with $n+1=2i$ given by
\[
\frac{\varpi_{(n+1)/2}}{2} \mapsto \frac{\varpi_{i+1}+\varpi_1}{2} \mapsto \ldots \mapsto\frac{\varpi_n+\varpi_{n-i}}{2} \mapsto \frac{\varpi_{(n+1)/2}}{2}.
\]
A representative is given by $\varpi_{(n+1)/2}$.
\end{proof}

For non-trivial involutions $u$ it is more complicated to determine the connected components of 
\[C_W(u)\backslash \mathbb{T}_u(\mathbb{R}).\]
After choosing a normal basis  for $P$ in which $u$ takes a normal form we can identify the component group $\pi_0(\mathbb{T}_u(\mathbb{R}))$ with $P_{1,u}/2P_{1,u}$. However since there is no canonical choice for $P_{1,u}$ we have to compute
\[ W_u^+\backslash \left( P_{1,u}/2P_{1,u} \right)\]
case by case. The real torus $\mathbb{T}_u(\mathbb{R})$ can be written as the disjoint union of its connected components in the following way
\begin{equation}\label{components}
\mathbb{T}_u(\mathbb{R}) =  \bigsqcup_{[\varpi] \in \frac{1}{2}{P_{1,u}}/{P_{1,u}}} \mathbb{T}_u^\varpi \quad \text{where} \quad \mathbb{T}_u^\varpi = \exp \left( \frac{1}{2}\varpi +iV_u^+ + V_u^- \right).
\end{equation}

\subsection{Connected components of real tori of type $E_7$}
In this section we determine all connected components of the space 
\[ C_W(u)\big\backslash \mathbb{T}_u(\mathbb{R})\]
where $u\in W$ is an involution in the Weyl group $W$ of type $E_7$ and $\mathbb{T}_u$ is the corresponding real torus. The results are listed in Table \ref{tablecom}. In the first column are the pairs of conjugacy classes of involutions $W$ we determined in Example \ref{invE7}. The total number of connected components equals $20$. Recall that the nodes of the $E_7$ diagram are numbered as in the diagram below. 
\begin{figure}[H]
\centering
\begin{tikzpicture}
  \tikzstyle{every node}=[circle,draw]  
    \node[label=60:$1$] (1) at ( 0,0) {};
    \node[label=60:$2$] (2) at ( 1,0) {};    
    \node[label=60:$3$] (3) at ( 2,0) {};
    \node[label=60:$4$] (4) at ( 3,0) {};
    \node[label=60:$7$] (7) at ( 2,1) {};
    \node[label=60:$5$] (5) at ( 4,0) {};
    \node[label=60:$6$] (6) at ( 5,0) {}; 
    \draw [-] (1) -- (2) -- (3) -- (4) -- (5) -- (6);
    \draw [-] (3) -- (7);   
\end{tikzpicture}
\caption{The labelling of the nodes of the Coxeter graph of type $E_7$.}
\end{figure}
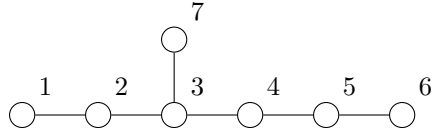

If $u$ is an involution with $n_1=0$ then the real torus $\mathbb{T}_u(\mathbb{R})$ is connected. Consequently the quotient $C_W(u)\backslash \mathbb{T}_u(\mathbb{R})$ is also connected. Now suppose that $n_1=1$ so that $P_{1,u}/2P_{1,u}\cong \mathbb{Z}/2\mathbb{Z}$. Since $\{0\}$ is a single $W_u^+$-orbit, so is $\{\varpi \}$ with $\varpi \in P$ a generator for $P_{1,u}/2P_{1,u}$ and there are two connected components. From Table \ref{tablecom} we see that this is the case for $A_1^3$ and $D_4$ which we represent by $I=\{s_2,s_4,s_7\}$ and $I=\{s_2,s_3,s_4,s_7\}$ respectively. In both cases the fundamental weight $\varpi_6$ is a generator for $P_{1,u}/2P_{1,u}$. For $n_1>2$ the situation becomes more complicated and we have to determine the action of $W_u^+$ on the generators of $P_{1,u}/2P_{1,u}$. These cases are $u=1,A_1,A_1^2$ or $A_1^{3\prime}$ and we treat them below. 

\begin{table}
\begin{equation*}
\begin{array}{cccccl}
\toprule
u\in W & n_1 & n_2 & n_3 & \# \text{components} & \text{representatives} \\
\midrule
1 & 7 & 0 & 0 & 4 & \{0,\varpi_5,\varpi_6,\varpi_7\} \\
E_7 & 0 & 7 & 0 & 1 & \{ 0 \} \\
A_1 & 5 & 0 & 1 & 3 & \{0,\varpi_3,\varpi_4\} \\
D_6 & 0 & 5 & 1 & 1 & \{ 0 \} \\
A_1^2 & 3 & 0 & 2 & 3 & \{0,\varpi_4,\varpi_5\} \\
D_4A_1 & 0 & 3 & 2 & 1 & \{ 0 \} \\
A_1^3 & 1 & 0 & 3 & 2 & \{0,\varpi_6\} \\
A_1^4 & 0 & 1 & 3 & 1 & \{ 0 \} \\
D_4 & 1 & 2 & 2 & 2 & \{ 0,\varpi_6\} \\
A_1^{3\prime} & 2 & 1 & 2 & 2 & \{0,\varpi_1\} \\
\bottomrule
\end{array}
\end{equation*}
\caption{The number of components of $C_W(u)\backslash \mathbb{T}_u(\mathbb{R})$ for all conjugacy classes of involutions $u\in W$. We also list the  corresponding representatives in $W_u^+\backslash \left( P_{1,u}/2P_{1,u} \right)$.}
\label{tablecom}
\end{table}

\begin{enumerate}
\item[$1 \quad$] The involution $u=1$ is of type $(7,0,0)$ so that we can use Formula \ref{orbitspaceformula}. The closure of the fundamental alcove $\overline{\mathcal{A}}$ intersected with lattice of half weights 
\begin{align*}
\bar{\mathcal{A}} \cap \frac{1}{2}P & =\Conv \left(0,\frac{\varpi_1}{2},\frac{\varpi_2}{3},\frac{\varpi_3}{4},\frac{\varpi_4}{3},\frac{\varpi_5}{2},\varpi_6, \frac{\varpi_7}{2} \right) \cap \frac{1}{2}P
\end{align*}
consists of the six elements $\{0,\varpi_1/2,\varpi_5/2,\varpi_6,\varpi_6/2,\varpi_7/2 \}$. The group $P/Q$ is of order two and acts on this set by $\gamma_6$ which interchanges $0\leftrightarrow \varpi_6$ and $\varpi_1/2 \leftrightarrow \varpi_5/2$. We conclude that there are four orbits in $W\backslash \left(P/2P\right)$ represented by $\{0,\varpi_5,\varpi_6,\varpi_7\}$.

\item[$A_1$ \quad] The involution $A_1$ is of type $(5,0,1)$. As a representative we pick $I=\{s_1\}$. Let $S_u$ be the matrix of $u$ with respect to the basis of fundamental weights for $P$ and let $B_u$ be a matrix whose columns represent a normal basis in the sense of Equation \ref{latdecom}. The normal basis for $P$ is not uniquely determined but we fix the choice below.
\begin{displaymath}
S_u = \begin{pmatrix}
-1 & 0 \\
1 & 1
\end{pmatrix}
\oplus I_5 \quad , \quad 
B_u = \begin{pmatrix}
1 & -1 \\
0 & 1
\end{pmatrix} \oplus I_5.
\end{displaymath}
A basis for $P_{1,u}$ and system of simple roots for $W_u^+$ are then given by respectively
\[  P_{1,u}= \mathbb{Z}\{ \varpi_7,\varpi_3,\varpi_4,\varpi_5,\varpi_6\} \ , \ \Delta(D_6)=\{\alpha_7,\alpha_3,\alpha_4,\alpha_5,\alpha_6,\alpha_I \} \] 
where $\alpha_I = e_0-e_3-e_4-e_5$. The lattice $P_{1,u}$ is a weight lattice of type $A_5$ and the group $W_u^+$ which acts on $P_{1,u}$ is of type $D_6$. All this is shown in the picture below. The black nodes represent the set $I$ of $W_u^-$, the crossed nodes the root system of $W_u^+$ and the grey nodes the fundamental weights of $P_{1,u}$. 
\begin{figure}[h]
\centering
\begin{tikzpicture}

	\node[circ,fill=black] (1) at (0,0) {};
	\node[circ] (2) at (1,0) {};
	\node[greycirc] (3) at (2,0) {};
	\node[greycirc] (4) at (3,0) {};
	\node[greycirc] (5) at (4,0) {};
	\node[greycirc] (6) at (5,0) {};
	\node[greycirc] (7) at (2,1) {};
	
	\node[cross,label=60:$\alpha_7$] (7a) at (2,1) {};
	\node[cross,label=60:$\alpha_3$] (3a) at (2,0) {};
	\node[cross,label=60:$\alpha_4$] (4a) at (3,0) {};
	\node[cross,label=60:$\alpha_5$] (5a) at (4,0) {};
	\node[cross,label=60:$\alpha_6$] (6a) at (5,0) {};
	\node[cross,label=60:$\alpha_I$] (b1) at (4,1) {};
	
	\draw [-] (1) -- (2) -- (3) -- (4) -- (5) -- (6);
	\draw [-] (3) -- (7);
	\draw [-] (5) -- (b1);
\end{tikzpicture} 
\caption{Constuction of the simple roots for $W_u^+$ with $u$ of type $A_1$}
\label{A1graph}
\end{figure}
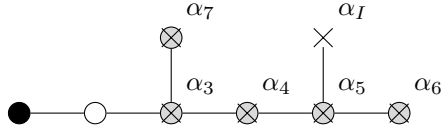

We need to determine the action of $W_u^+$ on $P_{1,u}$. First we observe that the parabolic subgroup $W(A_5)$ of $W_u^+$ generated by the reflections represented by grey nodes in the diagram is of type $A_5$ and acts on $P_{1,u}$ in the usual way. We see from example \ref{AN} that there are four orbits for $W(A_5)\backslash \left(P_{1,u}/2P_{1,u}\right)$  represented by $\{0,\varpi_7,\varpi_3,\varpi_4\}$. The remaining generating reflection $s_I$ acts on the basis for $P_{1,u}$ as
\[ s_{\alpha_I}(\varpi_7,\varpi_3,\varpi_4,\varpi_5,\varpi_6)=(\varpi_7+\varpi_5,\varpi_3+2\varpi_5,\varpi_4+\varpi_5,\varpi_5,\varpi_6 ).\]
A small calculation using Equation \ref{weightcalc} shows that
\begin{align*}
s_4s_3s_7s_I (\varpi_7) &= s_4s_3s_7(\varpi_7+\varpi_5)  \\
&= s_4s_3(\varpi_7+\varpi_3+\varpi_5)  \\
&= s_4(\varpi_3+\varpi_4+\varpi_5) \\
&= \varpi_4
\end{align*}     
so that the reflection $s_I$ exchanges the orbits $\varpi_7 \leftrightarrow \varpi_4$. This leaves three orbits for $W(D_6)\backslash \left( P_{1,u}/2P_{1,u} \right)$ represented by $\{0,\varpi_3,\varpi_4\}$. 

\item[$A_1^2$ \quad]The involution $A_1^2$ is of type $(3,0,2)$. As a representative we choose $I=\{s_1,s_6\}$. As a basis for the lattice $P_{1,u}$ we can choose $\mathbb{Z}\{\varpi_3,\varpi_4,\varpi_7\}$ which is of type $A_3$. The group $W_u^+$ is of type $D_4A_1$ with simple system 
\begin{align*}
\Delta(D_4A_1) = \{ \alpha_3,\alpha_4,\alpha_7,\alpha_I,\alpha_{II}
 \}
\end{align*} 
where $\alpha_I=e_0-e_3-e_6-e_7$ and $\alpha_{II}=-e_0+e_3+e_4+e_5$. The corresponding diagram is given below.
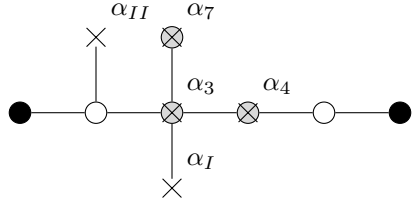
\begin{figure}[h]
  \centering
  \begin{tikzpicture}

	\node[circ,fill=black] (1) at (0,0) {};
	\node[circ] (2) at (1,0) {};
	\node[greycirc] (3) at (2,0) {};
	\node[greycirc] (4) at (3,0) {};
	\node[circ] (5) at (4,0) {};
	\node[circ,fill=black] (6) at (5,0) {};
	\node[greycirc] (7) at (2,1) {};

	\node[cross,label=60:$\alpha_3$] (3a) at (2,0) {};
	\node[cross,label=60:$\alpha_4$] (4a) at (3,0) {};
	\node[cross,label=60:$\alpha_7$] (7a) at (2,1) {};
	\node[cross,label=60:$\alpha_I$] (b1) at (2,-1) {};
	\node[cross,label=60:$\alpha_{II}$] (b2) at (1,1) {};
	
	\draw [-] (1) -- (2) -- (3) -- (4) -- (5) -- (6);
	\draw [-] (3) -- (7);
	\draw [-] (3) -- (b1);
	\draw [-] (2) -- (b2);
\end{tikzpicture} 
\caption{Constuction of the simple roots for $W_u^+$ with $u$ of type $A_1^2$}
\label{A2graph}
\end{figure}

The parabolic subgroup $W(A_3)$ of $W_u^+$ of type $A_3$ generated by the reflection represented by the grey nodes of the diagram acts on $P_{1,u}$ in the usual way. We can represent the orbits of $W(A_3)\backslash \left( P_{1,u}/2P_{1,u}\right)$ by $\{0,\varpi_3,\varpi_4\}$. The reflection $s_{\alpha_{II}}$ acts trivially on these orbits. The reflection $s_{\alpha_I}$ acts as
\[ s_{\alpha_I}(\varpi_7,\varpi_3,\varpi_4)=(\varpi_7+\varpi_3,\varpi_3,\varpi_4).\]
This action is identical to that of $s_7\in W(A_3)$ so the number of orbits of $W(D_4A_1)\backslash \left( P_{1,u}/2P_{1,u} \right)$ remains three with representatives $\{0,\varpi_3,\varpi_4\}$.  

\item[$A_1^{3\prime} \quad$]The involution $A_1^{3\prime}$ is of type $(2,1,2)$ and is represented by the elements $I=\{s_4,s_6,s_7\}$. The group $W_u^+$ is of type $D_4$ with simple system
\[ \Delta(D_4) = \{ \alpha_1,\alpha_2,\alpha_I,\alpha_{II} \} \]
where $\alpha_I=e_0-e_1-e_4-e_5$ and $\alpha_{II}=e_0-e_1-e_6-e_7$. As a basis for $P_{1,u}$ we can choose $\mathbb{Z}\{\varpi_1,\varpi_2 \}$. The diagram is shown below. 
\begin{figure}[h]
  \centering
  \begin{tikzpicture}

	\node[greycirc] (1) at (0,0) {};
	\node[greycirc] (2) at (1,0) {};
	\node[circ] (3) at (2,0) {};
	\node[circ,fill=black] (4) at (3,0) {};
	\node[circ] (5) at (4,0) {};
	\node[circ,fill=black] (6) at (5,0) {};
	\node[circ,fill=black] (7) at (2,1) {};
	
	\node[cross,label=60:$\alpha_{I}$] (b1) at (0,1) {};
	\node[cross,label=60:$\alpha_{II}$] (b2) at (-1,0) {};
	\node[cross,label=60:$\alpha_2$] (2a) at (1,0) {};
	\node[cross,label=60:$\alpha_1$] (1a) at (0,0) {};
	
	\draw [-] (1) -- (2) -- (3) -- (4) -- (5) -- (6);
	\draw [-] (3) -- (7);
	\draw [-] (1) -- (b1);
	\draw [-] (1) -- (b2);
\end{tikzpicture} 
\caption{Constuction of the simple roots for $W_u^+$ with $u$ of type $A_1^3$}
\label{A3graph}
\end{figure}
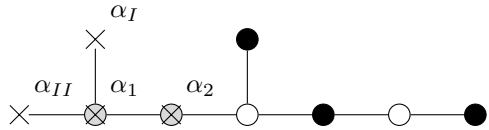

The parabolic subgroup $W(A_2)$ of $W_u^+$ generated by the reflections represented by the grey nodes in the diagram acts on $P_{1,u}$ in the usual way. The space  $W(A_2)\backslash \left( P_{1,u}/2P_{1,u} \right)$ consists of a single orbit represented by the element $\{\varpi_1 \}$. The reflections $s_{\alpha_I}$ and $s_{\alpha_{II}}$ both act as $s_2$ on $P_{1,u}$. So there are two $W_u^+$ orbits in $W(D_4)\backslash \left( P_{1,u}/2P_{1,u} \right)$ represented by the elements $\{0,\varpi_1\}$.

\end{enumerate}

\subsection{The complement of the mirrors}

In this section we prove that for a root system of type $ADE$ satisfying certain assumptions the connected components of the space $(W\backslash \mathbb{T}^\circ)(\mathbb{R})$ are of the form
\begin{equation}\label{concom}
\Stab_W(\mathbb{T}_u^\varpi) \backslash (\mathbb{T}_u^{\varpi} )^\circ \quad \text{where} \quad [\varpi]\in P_{1,u}/2P_{1,u}.
\end{equation}
This implies that removing the mirrors from $\mathbb{T}_u(\mathbb{R})$ does not add new components to the quotient $C_W(u)\backslash \mathbb{T}_u(\mathbb{R})$. In particular the number of connected components of $C_W(u)\backslash \mathbb{T}^\circ_u(\mathbb{R})$ for  involutions $u$ in $W(E_7)$ are the same as the numbers in Table \ref{tablecom}. 

\begin{dfn}
Let $q:\mathbb{T} \rightarrow W\backslash \mathbb{T}$ be the quotient map. The discriminant $D_{\mathbb{T}}$ is the set of critical values of $q$. It consists of union of the $q$-images of the toric mirrors and the $q$-image of the set
\[  \mathbb{T}_{P/Q} = \bigcup_{i} \exp(V_{\mathbb{C}}^{\gamma_i})  \] 
where we denote by $V_\mathbb{C}^{\gamma_i}$ the fixed points in $V_\mathbb{C}$ of the generator $\gamma_i$ for $P/Q$. 
\end{dfn}
\begin{lem}
The $q$-images of the real tori $\mathbb{T}_u(\mathbb{R})$ are disjoint in 
\[(W\backslash \mathbb{T})(\mathbb{R}) - D_{\mathbb{T}}(\mathbb{R}).\]
\end{lem} 
\begin{proof}
Suppose that $t\in \mathbb{T}_{u_1}\cap \mathbb{T}_{u_2}$ for involutions $u_1,u_2\in W$. This implies that $u_1 \cdot t = u_2\cdot t =\bar{t}$ so that in particular $u_1u_2 \cdot t = t$. But then $q(t)\in D_{\mathbb{T}}(\mathbb{R})$.  
\end{proof}

Since we are interested in connected components it suffices to consider the part of $D_\mathbb{T}(\mathbb{R})$ of codimension one in $(W\backslash \mathbb{T})(\mathbb{R})$. This motivates the following definition.

\begin{dfn}The \emph{real discriminant} $D_{\mathbb{T},\mathbb{R}}$ of $(W\backslash \mathbb{T})(\mathbb{R})$ is the closure of the nonsingular part of $D_\mathbb{T}(\mathbb{R})$. The difference is a $D_\mathbb{T}(\mathbb{R})-D_{\mathbb{T},\mathbb{R}}$ has codimension $\geq 2$ in $(W\backslash \mathbb{T})(\mathbb{R})$.
\end{dfn}

\begin{prop}If we assume that $\mathbb{T}_{P/Q}\cap \mathbb{T}_u$ has codimension $\geq 2$ for all involutions $u\in W$ then 
\[q^{-1}D_{\mathbb{T},\mathbb{R}}=\bigcup \left( \mathbb{T}_u\cap H_s \right) \] where the union runs over all involutions $u \in W$ and reflections $s\in W$ that commute with $u$.
\end{prop}
\begin{proof}
Under the assumption of the proposition the set $\mathbb{T}_{P/Q}$ does not contribute to the real discriminant. In this case an element $t\in \mathbb{T}_u$ is mapped to a nonsingular point of $D_\mathbb{T}(\mathbb{R})$ by $q$ if and only if there is a unique reflection $s\in W$ that fixes $t$. Since the reflection $usu$ also fixes $t$, we must have that $s$ commutes with $u$.
\end{proof}
The assumption is satisfied for type $E_7$. In that case $P/Q$ is generated by the involution $\gamma_6$. The locus of fixed points $V^{\gamma_6}$ has dimension four so that the codimension of $\mathbb{T}_{P/Q}\cap \mathbb{T}_u \geq 3$. In order to prove that the space of equation (\ref{concom}) is connected it is sufficient to prove that the space 
\[\Stab_W(\mathbb{T}_u^\varpi)\backslash \mathbb{T}_u^\varpi -D_{\mathbb{T}_u^\varpi,\mathbb{R}}  \]   
is connected. We prove the slightly stronger result that the quotient 
\[ \Stab_{W_u}(\mathbb{T}_u^\varpi)\backslash \mathbb{T}_u^\varpi -D_{\mathbb{T}_u^\varpi,\mathbb{R} } \]   
by the smaller group $\Stab_{W_u}(\mathbb{T}_u^\varpi)$ is connected. We start with a lemma on the decomposition $V=V_u^+ \oplus V_u^-$ into $\pm 1$-eigenspaces for $u$ on the weight lattice $P$. Denote by $P_u^+$ and $P_u^-$ the orthogonal projections of $P$ into $V_u^+$ and $V_u^-$ respectively.

\begin{lem}The lattice $P_u^-$ is equal to the weight lattice $P(W_u^-)$ of $W_u^-$. If we also assume that $-1 \in W$ then $P_u^+=P(W_u^+)$ and $R_u^+$ spans $V_u^+$.
\end{lem}

\begin{proof}
The simple system $\{\alpha_i\}_{i\in I}$ for the root system $R_u^-$ is a basis for $V_u^-$. The dual basis is given by $\{ \varpi_i^-\}_{i\in I}$ where $\varpi_i^-=\Proj_{V_u^-}(\varpi_i)$. We have
\begin{align*}
P(W_u^-) &= \{ \mathbb{Z} \in V_u^- \ ; \ (\mathbb{Z},\alpha_i)\in \mathbb{Z}  \quad \forall i\in I \} \\
&= \mathbb{Z} \left\{ \varpi_i^- \right\}_{i\in I} \\
&= P_u^-.
\end{align*}
If $-1\in W$ then for every involution $u\in W$ its opposite $-u$ is also an involution in $W$. Furthermore $V_u^\pm=V_{-u}^\mp$ and $W_u^\pm=W_{-u}^\mp$ so that we have equalities
\begin{displaymath}
P(W_u^+)=P(W_{-u}^-)=\Proj_{V_{-u}^-}(P)=\Proj_{V_u^+}(P).
\end{displaymath}
Similarly we have $\mathbb{R}R_u^+ = \mathbb{R}R_{-u}^- = V_{-u}^-=V_u^+$ so that $R_u^+$ spans $V_u^+$.
\end{proof}

\begin{thm}\label{technical}
Assume that $-1\in W$ and that $\mathbb{T}_{P/Q}\cap \mathbb{T}_u$ has codimension $\geq 2$ for all involutions $u\in W$. Let $\mathcal{A}_u^-$ be the fundamental alcove for the action of the affine Weyl group $Q_u^{-}\rtimes W_u^-$ on $V_u^-$ and let $\mathcal{C}_u^{\varpi +}$ be the fundamental chamber of the action of the Weyl group $\Stab_{W_u^+}\left(\frac{1}{2}\varpi \right)$ on the affine space $\frac{1}{2}\varpi+iV_u^+$. Then there is an isomorphism of orbifolds
\begin{displaymath}
\Stab_{W_u}(\mathbb{T}_u^\varpi) \big\backslash \mathbb{T}_u^\varpi - D_{\mathbb{T}_u^\varpi,\mathbb{R}} \cong \Stab_{P_u^{+}/Q_u^{+}}\left(\frac{1}{2}{\varpi} \right) \big\backslash \mathcal{C}_u^{\varpi +}
 \times \left( P_u^{-}/Q_u^{-} \right) \big\backslash \mathcal{A}_u^- .\end{displaymath}
\end{thm}   

\begin{proof}
Similar to the decomposition $V=V_u^+\oplus V_u^-$ there is a decomposition
\begin{align*}
\mathbb{T}_u^\varpi &= \exp \left(\frac{1}{2}\varpi +iV_u^+ + V_u^- \right) \\
& \cong \exp \left(\frac{1}{2}\varpi +iV_u^+\right) \times \exp \left(V_u^-\right)
\end{align*}
where $[\varpi] \in P_{1,u}/2P_{1,u}$ (so that in particular $\varpi \in V_u^+$). The stabilizer of $\mathbb{T}_u^\varpi$ in $W_u$ also splits into a product
\begin{align*}
\Stab_{W_u}\left( \mathbb{T}_u^\varpi \right) &\cong \Stab_{W_u^+} \exp\left( \frac{1}{2}\varpi + iV_u^+ \right) \times \Stab_{W_u^-} \exp\left( V_u^- \right) \\
&\cong \Stab_{P_u^+\rtimes W_u^+}\left(\frac{1}{2} \varpi \right) \times \Stab_{P_u^-\rtimes W_u^-} \left( V_u^- \right).
\end{align*}
The result now follows from applying Lemma \ref{stab} to these factors and taking the quotient. 
\end{proof}

\begin{rmk}While the roots of the Weyl group $W_u^+$ span the vector space $V_u^+$ the same need not be true for the Weyl group $\Stab_{W_u^+}(\frac{1}{2}\varpi)$, even if the group $W$ contains $-1$. An example is given by the Weyl group of type $E_7$ and the trivial involution $u=1$. In that case the Weyl group $\Stab_W(\frac{1}{2}\varpi_6)$ is of type $E_6$ and has rank six while $V$ is of dimension seven. A fundamental domain for this action is the product of a Weyl chamber of type $E_6$ and its orthogonal complement in $V$ which is a copy of $\mathbb{R}$. From this discussion we see that the chamber $\mathcal{C}_u^{\varpi +}$ is not a Weyl chamber in the traditional sense but the product of a Weyl chamber and an affine space. In particular it is connected, as is the alcove $\mathcal{A}_u^-$, which implies the following corollary.  
\end{rmk}

\begin{cor}Under the assumptions of Theorem \ref{technical}, the space 
\[ \Stab_{W_u}(\mathbb{T}_u^\varpi) \big\backslash \mathbb{T}_u^\varpi - D_{\mathbb{T}_u^\varpi,\mathbb{R}} \] 
is connected. 
\end{cor}

\section{The Geometry of the $20$ components}\label{geometry}

In this section we complete the correspondence between the $20$ connected components of the space $(W\backslash \mathbb{T}^\circ)(\mathbb{R})$ for a root system of type $E_7$ and the components of the moduli space $(\mathcal{Q}_1^\circ)^\mathbb{R}$. For each of these components we find a representative pair $\left( C,p \right)$ with $p\in C(\mathbb{R})$. The results are listed in the tables of Section \ref{tables}.

\begin{thm}
The curves in the tables of Section \ref{tables} represent the $20$ different components of $(\mathcal{Q}_1^\circ)^\mathbb{R}$.  
\end{thm}
\begin{proof}
It is clear that for the curves in the left columns the associated del Pezzo pair $(Y,Z)$ satisfies $Z^{\ns}(\mathbb{R}) \cong \mathbb{R}^\ast$ so that they belong to the space $\Stab_W(\mathbb{T}_u)\backslash \mathbb{T}_u^\circ$ for $u$ of type $1,A_1,A_1^2,A_1^3$ or $D_4$. Similarly the curves in the right column satisfy $Z^{\ns}(\mathbb{R})\cong S^1$ and belong to $\Stab_W(\mathbb{T}_{-u})\backslash \mathbb{T}_{-u}^\circ$ (so $-u$ is of type $E_7,D_6,D_4A_1,A_1^4$ or $A_1^{3\prime}$). Just check from the pictures whether $Z^{\ns}(\mathbb{R})$ has one or two components. With the exception of the two $M$-curves labeled $\varpi_6$ and $\varpi_7$ for all of the curves in the table the topological types of the pairs $\left( C(\mathbb{R}),T_pC(\mathbb{R} ) \right)$ are distinct. It is not possible to deform one of them into the other without passing through one of the strata $(\mathcal{Q}_1^\text{flex})^\mathbb{R}$, $(\mathcal{Q}_{1}^\text{bit})^\mathbb{R}$ or $(\mathcal{Q}_{1}^\text{hflex})^\mathbb{R}$ so that they lie indeed in different components of $(\mathcal{Q}_{1}^\circ)^\mathbb{R}$.  We need to prove that the $M$-curves labeled $\varpi_6$ and $\varpi_7$ are not in the same component. For this consider the affine quartics obtained by placing the tangent line $T_pC$ at infinity for these two curves. They are shown in Figure \ref{obstruction}. 

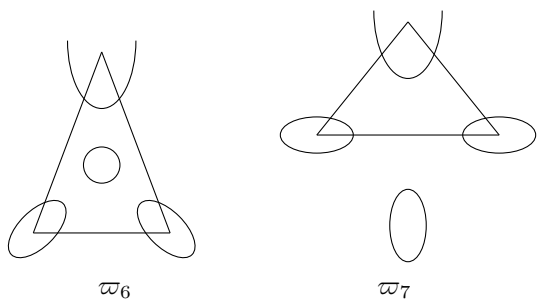
\begin{figure}[h]
\begin{displaymath}
\begin{array}{cc}
\begin{tikzpicture}[scale=0.3]
\draw (-1.5,5.5) arc (-180:0:1.5 and 3);
\draw[rotate = -45] (0,-4) ellipse (0.8 and 1.6);
\draw[rotate = 45] (0,-4) ellipse (0.8 and 1.6);
\draw (0,0) circle (0.8);
\draw (-3,-3) -- (3,-3);
\draw (-3,-3) -- (0,5);
\draw (3,-3) -- (0,5);
\end{tikzpicture}
\quad
&
\quad
\begin{tikzpicture}[scale=0.3]
\draw (-1.5,5.5) arc (-180:0:1.5 and 3);
\draw (-4,0) ellipse (1.6 and 0.8);
\draw (4,0) ellipse (1.6 and 0.8);
\draw (0,-4) ellipse (0.8 and 1.6);
\draw (-4,0) -- (4,0);
\draw (0,5) -- (-4,0);
\draw (0,5) -- (4,0);
\end{tikzpicture} \\
\varpi_6 & \varpi_7 
\end{array}
\end{displaymath}
\caption{The triangle forms an obstruction to deforming these two curves into each other by Bezout's Theorem.}
\label{obstruction}
\end{figure}

The triangle drawn in the picture forms an obstruction to deforming one into the other: it is not possible to move the central oval of the curve $\varpi_6$ out of the triangle without contradicting Bezout's theorem (a line intersects $C$ in four points). This is in agreement with Table $15$ in Appendix $1$ of \cite{Enriques} where certain affine $M$-quartics are classified. 
\end{proof}
 
To determine which of the curves in Section \ref{tables} corresponds to which component of $(\mathcal{Q}_1^\circ)^\mathbb{R}$ it is convenient to have an alternative description of the components of $\left( W\backslash \mathbb{T}^\circ\right)(\mathbb{R})$. For this we explicitly use the construction from the proof of Theorem \ref{bigrealthm} to associate to $\chi \in \mathbb{T}_u^\varpi(\mathbb{R})$ seven points in general position on the nonsingular locus of an irreducible real plane nodal cubic $Z \subseteq \mathbb{P}^2$. As before we identify $Z^{\ns}(\mathbb{C})\cong \mathbb{C}^\ast$ so that the points are defined by the formula  
\[ \chi \mapsto \left( P_1,\ldots,P_7 \right) \quad \text{with} \quad P_i=\chi \left( e_i-\frac{e_0}{3} \right)\]  
up to addition of an inflection point of $Z$. We start with the components corresponding to $M$-quartics.  

If $\chi$ is an element of the compact torus $\mathbb{T}_{-1} \cong \Hom(Q,S^1)$ then this construction determines seven points on the real point set $Z(\mathbb{R})$ of a real plane nodal cubic  with $Z^{\ns}(\mathbb{R}) \cong S^1$. 

If $\chi$ is an element of the split torus $\mathbb{T}_1 \cong \Hom(Q,\mathbb{R}^\ast)$ then the construction determines seven points on the real point set of a real plane nodal cubic with $Z^{\ns}(\mathbb{R})\cong \mathbb{R}^\ast$. If we choose the unique real inflection point of $Z$ as the unit element for the group law on $Z^{\ns}$ then these seven points are real. The Weyl group $W$ acts on $(\mathbb{R}^\ast)^7$ by permuting the coordinates and by triangular Cremona transformations in triples of points. For a seven-tuple $t=(t_1,\ldots,t_7)\in (\mathbb{R}^\ast)^7$ let $m_+$ denote the number of positive coordinates and $m_-$ the number of negative coordinates. The permutation orbit of $t$ is uniquely determined by the pair $(m_+,m_-)$. From Formula \ref{cremformula} we see that if we perform a triangular Cremona transformation in $t_i,t_j,t_k$ the sign of these points remains unchanged and the remaining points change sign if and only if one or three of the three points are negative. This describes the action of $W$ on the pairs $(m_+,m_-)$ and there are four orbits. These correspond to the four components of 
\[W\backslash \mathbb{T}_1^\circ = \bigsqcup_{[\varpi] \in W\backslash (P/2P)} \Stab_W(\mathbb{T}_1^\varpi) \big\backslash \left( \mathbb{T}_1^\varpi \right)^\circ \]
where $\varpi \in \{0,\varpi_5,\varpi_6,\varpi_7\}$. The precise correspondence is shown in Table \ref{orbitcor}. The stabilizers for the components in the table are calculated using Formula \ref{BorelSiebenthal} and Lemma \ref{stab}.   
\begin{table}[h]
\begin{displaymath}
\begin{array}{lll}
\toprule
\text{representative} & \Stab_W(\mathbb{T}_1^\varpi) & W\cdot (m_+,m_-) \\
\midrule
$[0]$ & W(E_7) & \{ (7,0) \} \\
{ [ \varpi_6 ] } & W(E_6)\rtimes \mathbb{Z}/2\mathbb{Z} & \{ (6,1),(2,5) \} \\
{ [ \varpi_5 ] }& W(D_6A_1) & \{ (5,2),(3,4),(1,6) \} \\
{ [\varpi_7 ] } & W(A_7) \rtimes \mathbb{Z}/2\mathbb{Z} & \{ (4,3),(0,7) \} \\
\bottomrule
\end{array}
\end{displaymath}
\caption{Connected components for $W\backslash \mathbb{T}_1(\mathbb{R})$. The first column lists the representatives in $W\backslash \left( P/2P \right)$. } 
\label{orbitcor}
\end{table}

To determine which picture from the tables in Section \ref{tables} for $u=1$ belongs to which of the components from Table \ref{orbitcor} we determine the adjacency relations between the five components corresponding to pointed $M$-quartics in $(\mathcal{Q}_{1}^\circ)^\mathbb{R}$. Two components are adjacent if their corresponding pointed quartics $(C,p)$ can be deformed into each other by moving through the stata $(\mathcal{Q}_{1}^\text{bit})^\mathbb{R}$ or $(\mathcal{Q}_{1}^\text{flex})^\mathbb{R}$ of codimension one. The effect of these two deformations is shown in Figure \ref{deformations}.

\begin{figure}[H]
\centering
\subfigure[Deforming through a flex.]{
\begin{tikzpicture}[scale=0.7]

\foreach \t in {0,...,6}
{
\draw[fill=black,xshift=5cm] ({(-1.2+0.1*\t)^2},{(-1.2+0.1*\t)^3}) circle (2pt);
\draw[fill=black] ({1.1+0.1*\t},{-sqrt((1.1+0.1*\t)^3-(1.1+0.1*\t)^2)}) circle (2pt);
\draw[fill=black,xshift=10cm] ({(-1.1-0.065*\t)^2-1},{(-1.1-0.065*\t)^3-(-1.1-0.065*\t)}) circle (2pt);
}

\draw[scale=1,domain=1:2,smooth,variable=\t, thick]
plot ({\t},{sqrt((\t)^3-(\t)^2)});
\draw[scale=1,domain=1:2,smooth,variable=\t , thick]
plot ({\t},{-sqrt((\t)^3-(\t)^2)});
\draw[fill=black] (0,0) circle (1pt);
\draw[scale=1,domain=-1.2:1.2,smooth,variable=\t, thick,xshift=5cm]
plot ({(\t)^2},{(\t)^3});
\draw[scale=1,domain=-1.5:1.5,smooth,variable=\t, thick,xshift=10cm]
plot ({(\t)^2-1},{(\t)^3-(\t)});

\end{tikzpicture}
}

\subfigure[Deforming through a bitangent.]{
\begin{tikzpicture}[scale=0.7]

\foreach \t in {1,...,6}
{
\draw[fill=black] ({(-1.1-0.065*\t)^2-1},{(-1.1-0.065*\t)^3-(-1.1-0.065*\t)}) circle (2pt);
\draw[fill=black,xshift=5cm] ({(-1.1-0.065*\t)^2-1},{(-1.1-0.065*\t)^3-(-1.1-0.065*\t)}) circle (2pt);
\draw[fill=black,xshift=10cm] ({(-1.1-0.065*\t)^2-1},{(-1.1-0.065*\t)^3-(-1.1-0.065*\t)}) circle (2pt);
}

\draw[fill=black] ({(-1.1)^2-1},{(-1.1)^3-(-1.1)}) circle (2pt);
\draw[fill=black,xshift=5cm] ({(-1)^2-1},{(-1)^3-(-1)}) circle (2pt);
\draw[fill=black,xshift=10cm] ({(-0.9)^2-1},{(-0.9)^3-(-0.9)}) circle (2pt);

\draw[scale=1,domain=-1.5:1.5,smooth,variable=\t, thick]
plot ({(\t)^2-1},{(\t)^3-(\t)});
\draw[scale=1,domain=-1.5:1.5,smooth,variable=\t, thick,xshift=5cm]
plot ({(\t)^2-1},{(\t)^3-(\t)});\draw[scale=1,domain=-1.5:1.5,smooth,variable=\t, thick,xshift=10cm]
plot ({(\t)^2-1},{(\t)^3-(\t)});

\end{tikzpicture}
}
\caption{The two ways of passing through a codimension one stratum in the moduli space $(\mathcal{Q}_1^\circ)^\mathbb{R}$.}
\label{deformations}

\end{figure}
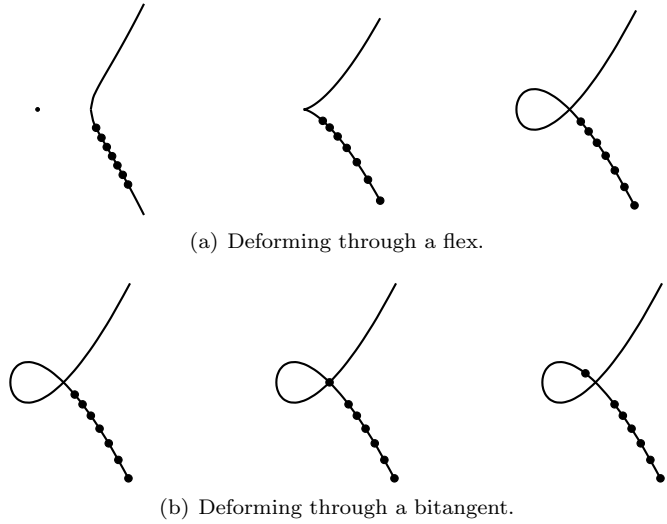

\begin{prop}
The adjacency graph for the five components of real pointed $M$-curves in $(\mathcal{Q}_{1}^\circ)^\mathbb{R}$ is given by:
\begin{center}
\begin{tikzpicture}
    \node (1) at ( 0,0) {$(7)$};
    \node (2) at ( 2,0) {$(7,0)$};    
    \node (3) at ( 4,0) {$(6,1)$};
    \node (4) at ( 6,0) {$(5,2)$};
    \node (5) at ( 8,0) {$(4,3)$};
    
    \draw [-] (1) -- (2) -- (3) -- (4) -- (5);       
\end{tikzpicture}
\end{center}
where we label components of $W\backslash \mathbb{T}_1^\circ$ by a representative for the corresponding orbit $W\cdot(m_+,m_-)$ and the component $W\backslash \mathbb{T}_{-1}^\circ$ by $(7)$.
\end{prop} 

\begin{proof}
A curve in the component coresponding to $W\backslash \mathbb{T}_{1}^\circ$ can only be deformed to one in the component of $W\backslash \mathbb{T}_{-1}^\circ$ by deforming through a flex point. This is a transition from $(7)$ to $(7,0)$. By repeatedly deforming through a bitangent one moves through the components 
\[(7,0)\leftrightarrow(6,1)\leftrightarrow(5,2)\leftrightarrow(4,3).\] This proves the proposition and shows that the pictures corresponding to the components are indeed the ones shown in the table for $u=1$.  
\end{proof} 

For $\chi \in \mathbb{T}_u(\mathbb{R})$ with $u$ of type $A_1^i$ with $i=1,2,3$ we can do a similar analysis. In this case the construction associates to $\chi$: $7-2i$ real points and $i$ pairs of complex conjugate points for a suitable representative $u$ (not involving the reflection $s_7$). For example the involution $u=s_6s_4$ of type $A_1^2$ acts as
\[ s_6s_4 \cdot (P_1,P_2,P_3,P_4,P_5,P_6,P_7) = (P_1,P_2,P_3,P_5,P_4,P_7,P_6) \]
on the $P_i$ so that $\chi \in \mathbb{T}_{s_6s_4}$ produces seven points in $Z^{\ns}(\mathbb{C})\cong \mathbb{C}^\ast$ with $P_1,P_2,P_3$ real points and $(P_4,P_5)$ and $(P_6,P_7)$ complex conjugate pairs. The centralizer $C_W(u)$ is more complicated in this case. It acts on the points by permutations preserving the real points and conjugate pairs and triangular Cremona transformations centered in a triples of real points or a real point and a pair of conjugate points. The orbits are calculated in Table \ref{otherorbit} and confirm the numbers we computed earlier in Table \ref{tablecom}. 
\begin{table}[h]
\centering
\begin{displaymath}
\begin{array}{lll}
\toprule
u & \text{representative} & C_W(u)\cdot (m_+,m_-) \\
\midrule
A_1 & [0] &\{ (5,0) \} \\
& [\varpi_4] & \{ (4,1),(2,3),(0,5) \} \\
& [\varpi_3] & \{ (3,2),(1,4) \} \\
A_1^2 & [0] & \{ (3,0) \} \\
 & [\varpi_4] & \{ (2,1),(0,3) \} \\
 & [\varpi_3] & \{ (1,2) \} \\
A_1^3 & [0] & \{ (1,0) \} \\
 & [\varpi_6] & \{(0,1)\} \\
\bottomrule
\end{array}
\end{displaymath}
\caption{Connected components of $C_W(u)\backslash \mathbb{T}_u(\mathbb{R})$ for $u$ of type $A_1^i$. The second column shows the representatives for $C_W(u)\backslash \left( P_{1,u}/2P_{1,u} \right)$ from Table \ref{tablecom}.}
\label{otherorbit}
\end{table}

For $\chi \in \mathbb{T}_u$ with $u$ of type $D_4$ or $A_3^\prime$ the situation is different. In this case $u$ acts as a nonstandard Cremona transformation on the points. In fact it acts as a de Jonqui\'eres involution of order three centered in $5$ of the points (for the definition we refer to \cite{Russo}). The curve $C(\mathbb{R})$ consists of two nested ovals and only the outer oval can contain an inflection point, otherwise we would again get a contradiction with Bezout's theorem. This implies that the component with $p$ on the outer oval is the unit component of $\mathbb{T}_u(\mathbb{R})$ for $u$ of type $D_4$ and $A_3^\prime$.  

\newpage

\section{Tables}\label{tables}

\begin{table}[H]
\centering
\begin{tabular}{ccc}
\toprule
\multicolumn{2}{c}{$1$} & $E_7$ \\
\midrule
\subfigure[$0$]{\includegraphics[trim = 250 0 250 0,clip,width=3.3cm]{Q4MUL0}} & 
\subfigure[$\varpi_6$]{\includegraphics[trim= 250 0 250 0,clip,width=3.3cm]{Q4MUL6}} & 
\multirow{2}{*}{\subfigure[$0$]{\includegraphics[trim = 250 0 250 0,clip,width=3.3cm]{Q4ADD}}} \\ 
\subfigure[$\varpi_5$]{\includegraphics[trim = 250 0 250 0,clip,width=3.3cm]{Q4MUL5}} & 
\subfigure[$\varpi_7$]{\includegraphics[trim = 250 0 250 0,clip,width=3.3cm]{Q4MUL7}} & \\

\toprule
\multicolumn{2}{c}{$A_1$} & $D_6$ \\
\midrule
\subfigure[$0$]{\includegraphics[trim = 250 0 250 0,clip,width=3.3cm]{Q3MUL0}} & 
\subfigure[$\varpi_4$]{\includegraphics[trim= 250 0 250 0,clip,width=3.3cm]{Q3MUL6}} & 
\multirow{2}{*}{ \subfigure[$0$]{\includegraphics[trim = 250 0 250 0,clip,width=3.3cm]{Q3ADD}}} \\ 
\multicolumn{2}{c}{\subfigure[$\varpi_3$]{\includegraphics[trim = 250 0 250 0,clip,width=3.3cm]{Q3MUL5}}} & \\ 
\bottomrule
\end{tabular}
\end{table}

\begin{table}[H]
\centering
\begin{tabular}{ccc}
\toprule
\multicolumn{2}{c}{$A_1^2$} & $D_4A_1$ \\
\midrule
\subfigure[$0$]{\includegraphics[trim = 250 0 250 0,clip,width=3.3cm]{Q2MUL0}} & 
\subfigure[$\varpi_4$]{\includegraphics[trim = 250 0 250 0,clip,width=3.3cm]{Q2MUL6}} & 
\multirow{2}{*}{\subfigure[$0$]{\includegraphics[trim = 250 0 250 0,clip,width=3.3cm]{Q2ADD}}} \\
\multicolumn{2}{c}{\subfigure[$\varpi_3$]{\includegraphics[trim = 250 0 250 0,clip,width=3.3cm]{Q2MUL5}} } & \\
\toprule
\multicolumn{2}{c}{$A_1^3$} & $A_1^4$ \\
\midrule
\subfigure[$0$]{\includegraphics[trim = 250 0 250 0,clip,width=3.3cm]{Q1MUL0}} & 
\subfigure[$\varpi_6$]{\includegraphics[trim = 250 0 250 0,clip,width=3.3cm]{Q1MUL6}} &
\subfigure[$0$]{\includegraphics[trim = 250 0 250 0,clip,width=3.3cm]{Q1ADD}} \\
\toprule
\multicolumn{2}{c}{$D_4$} & $A_1^{3\prime}$ \\
\midrule
\multirow{3}{*}{\subfigure[$0$]{\includegraphics[trim = 250 0 250 0,clip,width=3.3cm]{Q22MUL0}}}& 
\multirow{3}{*}{\subfigure[$\varpi_6$]{\includegraphics[trim = 250 0 250 0,clip,width=3.3cm]{Q22MUL6}}} &
\subfigure[$0$]{\includegraphics[trim = 250 0 250 0,clip,width=3.3cm]{Q22ADD}} \\
& & \subfigure[$\varpi_1$]{\includegraphics[trim = 250 0 250 0,clip,width=3.3cm]{Q22ADD2}} \\
\bottomrule
\end{tabular}
\end{table}

\bibliographystyle{plain}
\bibliography{master}

\end{document}